\documentclass[11pt]{article}

\usepackage{amsmath, amsfonts, amssymb, amsthm,  graphicx, enumerate, dsfont}
\usepackage[letterpaper, left=1truein, right=1truein, top = 1truein, bottom = 1truein]{geometry}

\usepackage[numbers, square]{natbib}

\usepackage[%
            CJKbookmarks=true,
            bookmarksnumbered=true,
            bookmarksopen=true,
            colorlinks=true,
            citecolor=red,
            linkcolor=blue,
            anchorcolor=red,
            urlcolor=blue,
            ]{hyperref}

\usepackage[usenames,dvipsnames]{color}

\usepackage[ruled, vlined, lined, commentsnumbered]{algorithm2e}

\usepackage{prettyref,soul,xspace}

\usepackage{tikz}
\usepackage{pgfplots,makecell}

\theoremstyle{remark}

\theoremstyle{plain}

\newtheorem{lemma}{Lemma}
\newtheorem{theorem}{Theorem}
\newtheorem{proposition}{Proposition}
\newtheorem{corollary}{Corollary}

\newcommand{\p}{\mathbb{P}}
\newcommand{\E}{\mathbb{E}}

\newcommand{\Ber}{\text{Bernoulli}}

\newcommand{\iid}{\stackrel{iid}{\sim}}

\newcommand{\br}[1]{\left( #1 \right)}

\newcommand{\cbr}[1]{\left\{ #1 \right\}}
\newcommand{\pbr}[1]{\p\left( #1 \right)}
\newcommand{\ebr}[1]{\exp\left( #1 \right)}
\newcommand{\abs}[1]{\left| #1 \right|}

\newcommand{\Binom}{\text{Binomial}}

\newcommand{\mathr}{\mathbb{R}}
\newcommand{\mathn}{\mathcal{N}}
\newcommand{\mathf}{\mathcal{F}}

\newcommand{\indic}[1]{{\mathbb{I}\left\{{#1}\right\}}}

\newcommand{\iprod}[2]{\left \langle #1, #2 \right\rangle}
\newcommand{\norm}[1]{\left\|{#1} \right\|}
\newcommand{\normt}[1]{\|{#1} \|}

\newcommand{\fnorm}[1]{\norm{#1}_{\rm F}}
\newcommand{\argmin}{\mathop{\rm argmin}}
\newcommand{\argmax}{\mathop{\rm argmax}}
\newcommand{\normf}[1]{\|{#1}\|_{\rm F}}

\newcommand{\one}{\mathds{1}}

\newcommand{\matho}{\mathcal{O}}

\renewcommand{\complement}{\mathsf{c}}

\usepackage{accents}

\title{Fundamental Limits of Spectral Clustering in Stochastic Block Models}

\author{Anderson Ye Zhang\\
~\\
University of Pennsylvania
}

\begin{document}
\maketitle

\begin{abstract}
Spectral clustering has been widely used for community detection in network sciences. While its empirical successes are well-documented, a clear theoretical understanding, particularly for sparse networks where degrees are much smaller than $\log n$, remains unclear. In this paper, we address this significant gap by demonstrating that spectral clustering offers exponentially small error rates when applied to sparse networks under Stochastic Block Models.  Our analysis provides sharp characterizations of its performance, backed by matching upper and lower bounds possessing an identical exponent with the same leading constant.  The key to our results  is a novel truncated $\ell_2$ perturbation analysis for eigenvectors, coupled with a new analysis idea of  eigenvectors truncation.

\end{abstract}

\section{Introduction}
Community detection \cite{abbe2017community, girvan2002community, zhao2011community} is a central problem in network science. The goal is to recover hidden community structures from network data and has broad applications in social science, neuroscience, computer science, and physics.   Among various approaches for community detection,
spectral clustering \cite{rohe2011spectral, von2008consistency, jin2022mixed, jin2015fast, lei2020consistency, paul2020spectral, gulikers2017spectral, zhang2020detecting, ghoshdastidar2017consistency, ke2022optimal, ke2023special, dhara2022power} is a particularly popular one and has achieved tremendous success.   It  first reduces the dimensionality of data by  a spectral decomposition and performs clustering in a reduced-dimension space.
 It  is computationally appealing, easy to implement, and has surprisingly good performance. 

 Driven by its popularity and success, there has been growing interest in theoretical and statistical analysis for the performance of spectral clustering. In most literature \cite{lei2015consistency, rohe2011spectral, qin2013regularized, zhou2019analysis, gao2017sbm_algo}, spectral clustering is shown to attain polynomially small errors. As a result, it is often used as a warm start to initialize  \cite{gao2017sbm_algo, lu2016statistical, xu2020optimal, han2020exact, gao2022community} delicate procedures in order to provably achieve exponentially small errors and even optimal statistical accuracy. However, spectral clustering performs exceptionally well numerically, indicating a gap between theory and practice. This raises important questions: \emph{Does spectral clustering achieve exponentially small errors? How small can it be?}

These questions are answered in  a seminal work \cite{abbe2020entrywise} that considers a very special setting where the network has two equal-sized communities and $p=a(\log n)/n, q=b(\log n)/n$ with fixed constants $a,b$. The spectral clustering procedure studied in  \cite{abbe2020entrywise} has a simple form: it only utilizes  the second leading eigenvector of the adjacency matrix and partitions the network using signs of coordinates. Under this setting, \cite{abbe2020entrywise} proves the spectral clustering achieves an optimal exponential rate.
However, the procedure, result, and proof technique of \cite{abbe2020entrywise} are limited to this special setting and cannot  be extended to even slightly more general cases especially when $p,q$ are of a smaller order of $(\log n)/n$. 
This leads to the following open problem: \emph{Can we establish sharp exponential error bound for spectral clustering on sparse networks when the degrees are far smaller than $\log n$?}
 
In this paper, we address this open problem by demonstrating that spectral clustering offers exponentially small error rates when applied to sparse networks under the Stochastic Block Model (SBM) \cite{holland1983stochastic} which is the most studied model for community detection.  In addition, we provide a matching lower bound that has the same exponent, including the leading constant, as the upper bound. The matching upper and lower bounds together give a sharp characterization of the performance of spectral clustering, demonstrating its ability and also limit. Hence, we refer to our results as the fundamental limits of spectral clustering in SBMs. We emphasize that this is different from information-theoretical analysis for SBMs, referred to as fundamental limits of SBMs  in \cite{abbe2017community} and minimax rates of SBMs in \cite{zhang2016minimax}. See Section \ref{sec:lower} for further elaboration.

Consider an $n$-node network with $k$ communities. For every two nodes, we observe an edge with probability $p$ if they belong to the same community and $q$ otherwise. The goal is to recover the hidden community structure $z^*$ given the network.
We study a popular spectral clustering procedure \cite{lei2015consistency, gao2017sbm_algo, zhou2019analysis, pensky2019spectral} $\hat z$
that is based on the eigendecomposition of the adjacency matrix. It first regularizes \cite{joseph2016impact} the network by removing high-degree nodes, a step that is necessary   for the concentration of sparse networks \cite{le2017concentration, chin2015stochastic}.  It then weights leading eigenvectors of the regularized adjacency matrix by corresponding eigenvalues, followed by the $k$-means clustering. Its clustering error can be measured by a loss $\ell(\hat z,z^*)$. See Sections \ref{sec:SBM}-\ref{sec:sc} for more details about the model and the algorithm.
 The main result of this paper is summarized below in Theorem \ref{thm:introduction}.
 
\begin{theorem}\label{thm:introduction}
Assume $k$ is a constant and all community sizes $n_1,n_2,\ldots,n_k$ are of the same order. Assume $p,q$ satisfy $0<q<p \leq 1/10$ and are of the same order. We further assume $\frac{n(p-q)^2}{p}\rightarrow\infty$. Then
\begin{align*}
&\E \ell(\hat z,z^*)\leq \ebr{-\br{1-o(1)}J_{\min}}+2n^{-3},\\
\text{and }\quad& \E \ell(\hat z,z^*)\geq \ebr{-\br{1+o(1)} J_{\min}}-2n^{-3},
\end{align*}
where %
\begin{align}\label{eqn:J_min}
J_{\min}:= \min_{1\leq a\neq b\leq k}\max_{t> 0} \br{(n_a-n_b)t \frac{p+q}{2} - n_a \log\br{q e^t + 1-q} - n_b\log \br{pe^{-t}+1-p}}.
\end{align}
\end{theorem}

Theorem \ref{thm:introduction} gives both upper and lower bounds for the partial recovery of spectral clustering $\hat z$. We note that the additive term $2n^{-3}$ in the error bounds can be replaced by $n^{-C}$ for an arbitrarily large constant $C>0$, and in general should be ignored. Then the upper and lower bounds are both in an exponential form with a matching asymptotic exponent $J_{\min}$. In this way, Theorem \ref{thm:introduction} gives an exact characterization of the error exponent of the spectral clustering, even including the sharp leading constant. Theorem \ref{thm:introduction} holds under mild conditions, allowing networks to have multiple imbalanced communities and to be sparse with $p,q\ll (\log n)/n$. The assumption ${n(p-q)^2}/{p}\rightarrow\infty$ is known to be the necessary and sufficient condition to have consistent community detection \cite{zhang2016minimax}.
 It is also worth mentioning that Theorem \ref{thm:introduction} characterizes precisely the performance of the spectral clustering for each instance of $z^*$, which is beyond the minimax framework that only focuses on the worst case of a large parameter space. %

The asymptotic exponent $J_{\min}$ in Theorem \ref{thm:introduction} has a
complicated dependence on $p,q$ and the community sizes. Despite  no explicit expression,
the quantity $J_{\min}$ is closely related to tail probabilities of Bernoulli random variables. Let $\{X_i\}$ and $\{Y_j\}$ be independent Bernoulli random variables with probabilities $q,p$ respectively.  By Chernoff bound (see Lemma \ref{lem:binomial_diff}), 
\begin{align*}
 \min_{1\leq a\neq b\leq k} -\log\pbr{\sum_{i\in[n_a]} X_i - \sum_{j\in[n_b]} Y_j \geq (n_a-n_b)\frac{p+q}{2}} =(1+o(1))J_{\min}.
\end{align*}
To further explain why $J_{\min}$ appears in Theorem \ref{thm:introduction} and is fundamental for the performance of the spectral clustering, we study an oracle estimator that is inspired by $\hat z$ but utilizes the unobserved population eigenstructure instead of the sample one. Its corresponding statistical accuracy turns out to be determined by the aforementioned tail probabilities (see Section \ref{sec:oracel}). 

Under the setting of Theorem \ref{thm:introduction}, $J_{\min}$ and  $n(p-q)^2/p$ can be shown to be of the same order (see Lemma \ref{lem:J}). Hence, the assumption $n(p-q)^2/p\rightarrow\infty$ in Theorem \ref{thm:introduction} can  be replaced by $J_{\min}\rightarrow\infty$ and   is the sufficient and necessary condition for $\hat z$ to have a vanishing error.
Theorem \ref{thm:introduction} also immediately indicates a sharp threshold of the exact recovery. When $J_{\min}\geq (1+\epsilon) \log n$ for any constant $\epsilon>0$,  $\hat z$ achieves the exact recovery (i.e., $\ell(\hat z,z^*)=0$) with high probability. When $J_{\min}\leq (1-\epsilon) \log n$ for any constant $\epsilon>0$, $\hat z$ fails to achieve the exact recovery with  constant probability, i.e., $\pbr{\ell(\hat z,z^*)\neq 0}\geq c$ for some constant $c>0$.

The most related work in the literature is  \cite{abbe2020entrywise}.
As mentioned earlier,  \cite{abbe2020entrywise} considers a very special setting where the network has two equal-sized communities and $p=a(\log n)/n, q=b(\log n)/n$ with fixed constants $a,b$ and its results cannot  be extended to even slightly more general cases.
On the contrary, our results hold for general SBMs  and cover the setting of \cite{abbe2020entrywise} as a special case. Other closely related papers including \cite{loffler2019optimality, abbe2022lp, zhang2022leaveone} obtain exponential error bounds for  spectral clustering under Gaussian and sub-Gaussian mixture models. However, they rely heavily on Gaussianity and sub-Gaussianity of data. Direct application of their results to networks or to other binary data can only lead to trivial upper bounds. 

The key  to Theorem \ref{thm:introduction} is a novel truncated $\ell_2$ perturbation analysis for eigenvectors and a proof idea of  eigenvectors truncation. Let $u$ be one of the leading $k$ eigenvectors of the regularized adjacency matrix.  
Analysis of the spectral clustering shows that tail probabilities of quantities in the form of $\sum_{i\in[n]}u_i(X_i -\E X_i)$ play a crucial role, where $\{X_i\}_{i\in[n]}$ are some Bernoulli random variables such as edges of a node. Direction applications of classical concentration inequalities for Bernoulli random variables often involve $\norm{u}_{\infty}$  that has some inevitable $\log n$ factor. To deal with this $\log n$ factor and to derive meaningful tail probabilities, $p$ has to be at least of an order $(\log n)/n$ as seen in literature \cite{chen2022partial, abbe2020entrywise, chen2019spectral}. However, in this paper, we consider sparse networks with the connectivity probability allowed to be far smaller than $(\log n)/n$ and the aforementioned analysis breaks down.
Instead, we truncate coordinates of $u$ by some carefully selected threshold $t_0$. The truncated eigenvector has an $\ell_\infty$ norm bounded by $t_0$  and  its inner product with $\{X_i-\E X_i\}$ is well-controlled with desired tail probabilities. The approximation error of replacing $u$ by its truncated counterpart turns out to be related to $\sum_{i\in[n]} \norm{u_i}^2\indic{\abs{u_i}\geq t_0}$, which we refer to as a truncated $\ell_2$ norm of $u$. We establish an upper bound for the  truncated $\ell_2$ norm in Theorem \ref{thm:truncated_l2} and further show such approximation error is negligible. The idea of eigenvector truncation and Theorem \ref{thm:truncated_l2} are critical to establishing sharp upper and lower bounds in Theorem \ref{thm:introduction}, and might be useful for fine-grained spectral perturbation analysis of other binary random matrices. To establish Theorem \ref{thm:introduction}, we also use a leave-one-out technique \cite{chen2021spectral, abbe2020entrywise} to decouple dependence between the eigenvectors and the regularized adjacency matrix.

We conclude this section by summarizing contributions of this paper:
\begin{enumerate}
\item Our result is the first in the literature to  show spectral clustering has exponentially small error rate for sparse networks.
\item Our characterization of the performance of spectral clustering is sharp and precise. In addition to the upper bound, we provide a matching lower bound that has the same exponent, including the leading constant.
\item Our results are backed by a novel spectral perturbation analysis that could be a valuable tool for other sparse graph problems.
\end{enumerate}

\paragraph{Organization. }
In Sections \ref{sec:SBM}-\ref{sec:sc}, we provide more details about  SBMs and give a detailed implementation of the spectral clustering. In Section \ref{sec:polynomial}, we give a preliminary polynomial upper bound for the performance of the spectral clustering. In Section \ref{sec:oracel}, we carry out an oracle analysis to provide intuition on the fundamental limits of the spectral clustering. 
We establish the upper bound part of Theorem \ref{thm:introduction} in Section \ref{sec:upper} and the lower bound part in Section \ref{sec:lower}.  The truncated $\ell_2$ perturbation analysis of eigenvectors is presented in  Section \ref{sec:truncated}. The proofs of the upper bound and the truncated $\ell_2$ perturbation analysis are given in Section \ref{sec:proof_upper} and Section \ref{sec:proof_truncated}, respectively. We include the proof of the lower bound in the appendix, along with proofs of the results of Section \ref{sec:pre} and all auxiliary lemmas.

\paragraph{Notation. } For any positive integer $r$, let $[r]:=\{1,2,\ldots, r\}$. For any vector $x$, we denote $\norm{x}_\infty:=\max_{i}\abs{x_i}$ to be its $\ell_{\infty}$ norm and $\norm{x}_1 = \sum_i \abs{x_i}$ to be its $\ell_1$ norm.
For a matrix $A$, denote $A_{i\cdot}$ and $A_{\cdot i}$ to be its $i$th row and column, receptively. We further denote $\norm{A}$ to be its operator norm, $\fnorm{A}$ to be its Frobenius norm, and $\norm{A}_{2,\infty}:= \max_i \norm{A_{i\cdot}}$ to be its maximum $\ell_2$ norms of rows. For any two numbers $a,b\in\mathr$, we denote $a\wedge b:= \min\{a,b\}$ and $a\vee b:= \max\{a,b\}$. We denote $\indic{\cdot}$ as the indicator function. 
For any two positive integers $a,b$, we denote $I_a$ to be the $a\times a$ identity matrix and $\matho(a,b)$  to be the set of all $a\times b$ matrices with orthogonal columns. For any random vectors $X,Y$, and $Z$, we use $X\perp Y|Z$ to mean $X$ and $Y$ are independent conditioned on $Z$. For any event $\mathcal{G}$, we denote $\mathcal{G}^\complement$ to be its complement.
 
\section{Preliminaries}\label{sec:pre}
\subsection{Community Detection and Stochastic Block Models}\label{sec:SBM}
Consider an $n$-node network with its adjacency matrix denoted by $A\in\{0,1\}^{n\times n}$ such that $A=A^T$ and $A_{ii}=0$ for all $i\in[n]$. Under the SBM, all edges $\{A_{ij}\}_{1\leq i<j\leq n}$ are independent Bernoulli random variables with probabilities depending on the underlying community structure.
Let $z^*\in[k]^n$ be a community assignment vector such that each coordinate indicates which community the corresponding node belongs to. 
We assume
\begin{align*}
\E A_{ij} = \begin{cases}
p,\text{ if }z^*_i=z^*_j,\\
q,\text { o.w.,}
\end{cases}
\end{align*}
for all $1\leq i<j\leq n$, where $0<q<p<1$. That is, nodes are more likely to be connected if they are from the same community: the probability for two nodes to be connected is $p$ if they belong to the same community and is $q$ otherwise. The goal of community detection is to estimate $z^*$ given the network $A$. Throughout the paper, we assume $k$, the number of communities, is known.

We denote $n_1,\ldots,n_k$ to be community sizes such that $n_a:= \sum_{i\in[n]}\indic{z^*_i =a}$ for all $a\in[k]$. Define $\beta := ({\min_{a\in[k]}n_a})/{(n/k)}$ such that $\beta n/k$ is the smallest community size. 
For any $z\in [k]^n$, its performance for community detection  can be measured by the following loss function \cite{gao2017sbm_algo}:
\begin{align*}
\ell(z,z^*) := \frac{1}{n} \min_{\phi\in\Phi} \sum_{i\in[n]} \indic{z_i= \phi(z^*_i)},
\end{align*}
where $\Phi:=\{\phi:\phi \text{ is a bijection from $[k]$ to $[k]$}\}$. It is a value between 0 and 1, giving the proportions of nodes mis-clustered in $z$ compared to $z^*$.

\subsection{Spectral Clustering}\label{sec:sc}
Spectral clustering refers to clustering procedures built upon the eigendecomposition or the singular value decomposition of matrices constructed from data. There exist different variants of spectral clustering for community detection \cite{newman2006modularity, priebe2019two, sarkar2015role, cape2019spectral, jin2015fast}. They  differ in the matrix on
which the spectral decomposition is applied and in what spectral components are used for the
subsequent clustering. The spectral clustering considered in this paper is a popular one and has been widely studied  in literature. It contains three steps summarized below with the detailed implementation given in Algorithm \ref{alg:1}.

\begin{algorithm}[ht]
\SetAlgoLined
\KwIn{Adjacency matrix $A\in\{0,1\}^{n\times n}$,  number of communities $k$, threshold $\tau$}
\KwOut{Community assignment $\hat z$}
 \nl Define  $d_i := \sum_{j\neq i} A_{ij}$ for all $i\in[n]$ to be degrees of $A$. Let $\tilde A$  be a trimmed version of $A$ by replacing its $i$th row and column by 0 whenever $d_i\geq \tau$
, for all $i\in[n]$. That is,  for all $i,j\in[n]$, 
\begin{align*}
\tilde A_{ij} :=\begin{cases}
A_{ij},\text{ if }d_i,d_j<\tau,\\
0,\text{ o.w.}
\end{cases}
\end{align*} \\
 \nl Let the eigendecomposition of $\tilde A$ be $\tilde A = \sum_{i\in[n]} \lambda_i u_iu_i^T$ with eigenvalues $\lambda_1\geq\ldots\geq \lambda_n$ and eigenvectors $u_1,\ldots, u_n\in\mathr^n$. Define $U:=(u_1,\ldots,u_k)\in\mathr^{n\times k}$ to be the leading  eigensapce and $\Lambda:=\text{diag}(\lambda_1,\ldots,\lambda_k)\in\mathr^{k\times k}$ to be the diagonal matrix with the leading $k$ eigenvalues. Denote $U_{1\cdot},\ldots,U_{n\cdot}\in\mathr^{1\times n}$ to be rows of $U$.\\
 \nl Apply $k$-means on the rows of $U\Lambda\in\mathr^{n\times k}$ and let $\hat z$ be the clustering output. That is, 
 \begin{align}
(\hat z,\{\hat \theta_1,\ldots,\hat \theta_k\}):= \argmin_{z\in[k]^n, \theta_1,\ldots,\theta_k\in\mathr^{1\times k}} \sum_{i\in[n]}\norm{U_{i\cdot}\Lambda - \theta_{z_i}}^2.\label{eqni:1}
 \end{align}
\caption{Spectral Clustering for Community Detection in Stochastic Block Models \label{alg:1}}
\end{algorithm}

In the first step, we regularize the adjacency matrix by removing high-degree nodes, i.e., nodes with degrees greater or equal to 
the threshold $\tau$. This step is necessary for sparse networks as the adjacency matrix $A$  is known to be away from its expectation $\E A$ when $p\ll (\log n)/n$. On the contrary, by zeroing out rows and columns of $A$ that correspond to the high-degree nodes, we have a trimmed adjacency matrix $\tilde A$ that is highly concentrated around $\E A$ \cite{le2017concentration, chin2015stochastic}.  
Primarily for the sake of theoretical analysis, we choose
\begin{align*}
\tau =20np,
\end{align*}
throughout the paper.
 However, it is conceivable to replace this with a more general form, such as $Cnp$, where $C$ represents a large constant. It is important to emphasize that the inclusion of $p$, an unknown parameter, in both $20np$ and $Cnp$ makes the threshold, $\tau$, impractical for direct application. In order to render Algorithm \ref{alg:1} more feasible in practice, one could follow the approach in \cite{gao2017sbm_algo,gao2018community} by setting $\tau$ to be $C'\bar{d}$. Here, $\bar{d}$ stands for the average degree and $C'$ is a sufficiently large constant. Nevertheless, for the purposes of streamlined theoretical exposition, we have opted for $\tau = 20np$ ino this paper.

 In the second step, we obtain two matrices $U,\Lambda$ through the eigendecomposition of  $\tilde A$, where $U$ is a matrix including the leading $k$ eigenvectors and $\Lambda$ is a diagonal matrix with the leading $k$ eigenvalues. Throughout the paper, we refer to $U$ as the leading eigenspace of $\tilde A$. Since eigenvectors are not of equal importance, they are weighted with their corresponding eigenvalues and $U\Lambda$ is used for the subsequent clustering.
 
 In the third step, we perform the $k$-means clustering on rows of $U\Lambda$ which are of dimension $k$.
 Compared to the adjacency matrix $A$, we greatly reduce the dimensionality of data from $n$ to $k$ as usually $k\ll n$ and perform clustering in a low-dimensional space. The $k$-means clustering returns a partition of the data and cluster centers, which are denoted as $\hat z$ and $\{\hat \theta_a\}_{a\in[k]}$ respectively.

Algorithm \ref{alg:1} effectively performs a rank-$k$ approximation of $\tilde A$ as part of its process. To elaborate, after constructing $\tilde A$ in Algorithm \ref{alg:1}, applying $k$-means directly to the rows of $U\Lambda$ is, in fact, equivalent to performing $k$-means on the rows of the rank-$k$ approximated matrix $U\Lambda U^T\in\mathbb{R}^{n\times n}$. This equivalence arises because the orthonormal columns of $U$ ensure that the Euclidean distance between the rows in $U\Lambda$ remains unchanged in $U\Lambda U^T$. Formally, for any two indices $i, j \in [n]$, the distance property $\|U_{i\cdot}\Lambda - U_{j\cdot}\Lambda\| = \|U_{i\cdot}\Lambda U^T - U_{j\cdot}\Lambda U^T\|$ holds, guaranteeing identical clustering results from $k$-means on either matrix.
Although the two approaches are equivalent in terms of clustering outcome, using $U\Lambda$ instead of $U\Lambda U^T$ for $k$-means offers computational benefits, as the former reduces the computational load and storage requirements, given the lower dimensionality of $U\Lambda$ compared to $U\Lambda U^T$. As a result, in the third step of Algorithm \ref{alg:1}, $k$-means is applied to $U\Lambda$ instead of the rank-$k$ approximation of $\tilde A$.

The current form of Algorithm \ref{alg:1} is tailored for the standard SBMs introduced in Section \ref{sec:SBM}. However, it holds potential for adaptation to more complex models such as bipartite SBMs, which are characterized by an asymmetric adjacency matrix with distinct row and column community structures, as studied in \cite{zhou2019analysis}. To accommodate the unique features of bipartite graphs, the algorithm requires modification as follows.
In the initial step, we compute both row-wise and column-wise degrees to identify and trim high-degree nodes, resulting in a matrix $\tilde A$. Subsequently, instead of eigendecomposition, singular value decomposition (SVD) is applied to $\tilde A$, yielding the left singular matrix $U$, right singular matrix $V$, and a diagonal matrix of singular values $\Lambda$. This adaptation leverages the asymmetry of the bipartite structure. Finally, community detection is performed by applying $k$-means clustering separately to the rows of $U\Lambda$ and the columns of $\Lambda V^T$, facilitating the recovery of the distinct row-wise and column-wise community structures.

It should be emphasized that Algorithm \ref{alg:1} is not a novel contribution of this paper. Our primary contribution lies not in the algorithm itself, but in the precise and detailed analysis of its performance.

\subsection{A Polynomial Upper Bound}\label{sec:polynomial}
In this section, we provide preliminary analysis for the spectral clustering and show it achieves a polynomial error rate. We first introduce a matrix $P\in\mathr^{n\times n}$ defined as
\begin{align*}
P_{ij}:=p\indic{z^*_i=z^*_j} + q\indic{z^*_i\neq z^*_j},\forall i,j\in[n].
\end{align*}
It can be viewed as a population matrix with $\tilde A$ being its sample counterpart.
Note that $P$ is different from the expectation matrix $\E A$ as the latter has all diagonal entries being zero due to the fact that the network has no self-loops. 

The matrix $P$ is rank-$k$. Let the eigendecomposition of $P$ be $P=\sum_{i\in[k]}\lambda^*_i u^*_i u^{*T}_i$ with eigenvalues $\lambda_1^*\geq\ldots\geq \lambda_k^*$ and eigenvectors $u_1^*,\ldots, u_n^*\in\mathr^n$. Define $U^*:=(u_1^*,\ldots,u_k^*)\in\mathr^{n\times k}$ to be the leading  eigenspace and $\Lambda^*:=\text{diag}(\lambda_1^*,\ldots,\lambda_k^*)\in\mathr^{k\times k}$ to be the diagonal matrix with the leading $k$ eigenvalues. Then we have $P=U^*\Lambda^* U^{*T}$
and
the matrix $U^*\Lambda^*$ satisfies the following property.

\begin{lemma}\label{lem:population}
The matrix $U^*\Lambda^*\in\mathr^{n\times k}$ has $k$ unique rows. To be more specific, there exist $\theta^*_1,\ldots,\theta^*_k\in\mathr^{1\times k}$ such that $(U^*\Lambda^*)_{i\cdot} = \theta^*_{z^*_i}$ for all $i\in[n]$. In addition, 
\begin{align*}
\norm{\theta^*_a - \theta^*_b} = \sqrt{n_a+n_b}(p-q),
\end{align*}
and $\normt{\theta^*_a}^2 = (p^2-q^2)n_a + q^2 n$, for all $a,b\in[k]$ such that $a\neq b$.
\end{lemma}

Lemma \ref{lem:population} reveals that rows of $U^*\Lambda^*$ are equal if their corresponding nodes belong to the same community and the $k$ unique rows $\{\theta^*_a\}$ are separated from each other.
As a result, if the $k$-means clustering is performed on rows of $U^*\Lambda^*$, we will have a perfect partition of the network with clustering centers $\{\theta^*_a\}$. But this is unrealistic as  $U^*\Lambda^*$ is unobserved. A natural idea is to use the eigendecomposition of the adjacency matrix $A$. Intuitively, if  $A$ and $P$ are close, then $U^*\Lambda^*$ is close to its counterpart obtained from $A$, and consequently the clustering error using the latter is small. However, it is known in random matrix and graph theory that sparse random graphs do not concentrate \cite{le2017concentration}, meaning that $A$ is away from $P^*$. Recent literature reveals that the removal of the high degree vertices enforces concentration \cite{le2017concentration}, which motivates the use of $\tilde A$ instead of $A$ in Algorithm \ref{alg:1}. The concentration of $\tilde A$ around $P$ (note that $P$ are $\E A$ are nearly equal as  they only differ in diagonal entries) is given in the following Lemma \ref{lem:tilde_A_concentration}.

\begin{lemma}\label{lem:tilde_A_concentration}
There exists a constant $C_0>0$ such that
\begin{align}\label{eqn:concentration}
\norm{\tilde A - \E A} \leq C_0\sqrt{np}
\end{align}
with probability at least $1-2n^{-3}$.
\end{lemma}

In Lemma \ref{lem:tilde_A_concentration}, $C_0$ represents an unspecified absolute constant. Lemma \ref{lem:tilde_A_concentration} is given as Lemma 12 and proved in \cite{chin2015stochastic}. Despite that \cite{chin2015stochastic} only states that the upper bound holds with probability $1-o(1)$, its proof gives an explicit expression for  the probability that is at least $1-2n^{-3}$. However, \cite{chin2015stochastic} does not give an explicit value for $C_0$. By scrutinizing its proof, $1-2n^{-3}$ in Lemma \ref{lem:tilde_A_concentration} can be generalized to $1-n^{-r}$ for any constant $r>0$, and then $C_0$ can be denoted as $C_0(r)$, a function of $r$. A similar result is given as Theorem 1.1 and proved in \cite{le2017concentration}. The proofs in \cite{chin2015stochastic} and \cite{le2017concentration} are lengthy and technical. For these reasons, and to maintain focus on the core contributions of our work, we choose not to include the proof of the lemma in this paper and refer readers to these sources for a  broader context of graph concentration.

In the remaining part of the paper, we will analyze the performance of the spectral clustering under the with-high-probability event that $\tilde A$ is well-concentrated around $\E A$.
Denote an event
\begin{align}\label{eqn:event_f_def}
\mathf:=\mathbb{I}\{\normt{\tilde A - \E A}\leq C_0\sqrt{np}\},
\end{align}
where $C_0$ is the constant from the statement of Lemma \ref{lem:tilde_A_concentration}. 
  Then $\pbr{\mathf}\geq 1-2n^{-3}$. Under $\mathf$, by classical spectral perturbation theory, we immediately have the following preliminary result.

\begin{proposition}\label{prop:prelim}
Assume the event $\mathf$ holds.
There exist constants $C_1,C_2>0$ and some $\phi\in\Phi$ such that if $\frac{n(p-q)^2}{\beta^{-2}k^3 p}\geq C_1$, we have
\begin{align}
&\frac{1}{n}\sum_{i\in[n]}\indic{\hat z_i \neq \phi(z^*_i)} \leq C_2 \frac{k^2\beta^{-1}p}{n(p-q)^2},\label{eqn:prop_prelim_1}\\
\text{and }\quad & \max_{a\in[k]} \norm{\hat \theta_{\phi(a)} U^T - \theta^*_a U^{*T}}\leq C_2 \beta^{-0.5} k\sqrt{p}.\label{eqn:prop_prelim_2}
\end{align}
\end{proposition}

The upper bound in (\ref{eqn:prop_prelim_1}) is essentially equal to the reciprocal of ${n(p-q)^2}/{p}$, a quantity regarded as the signal-to-noise ratio in the community detection literature \cite{zhang2016minimax}. Hence,  (\ref{eqn:prop_prelim_1}) decreases polynomially as ${n(p-q)^2}/{p}$ grows and we refer to it as a polynomial upper bound. Proposition \ref{prop:prelim} also gives an upper bound for the deviation between cluster centers $\{\hat \theta_a\}$ and their population counterpart $\{\theta^*_a\}$.
Results similar to Proposition \ref{prop:prelim} for spectral clustering can be found in \cite{lei2015consistency, loffler2019optimality, zhang2022leaveone}. Proposition \ref{prop:prelim} serves as the starting point for our further analysis toward Theorem \ref{thm:introduction}.

\subsection{Oracle Analysis and Exponents}\label{sec:oracel}
In this section, we provide heuristic arguments to explain why the spectral clustering has exponential error bounds and to derive the exponent $J_{\min}$. For any two positive integers $m_1,m_2$, define
\begin{align}\label{eqn:J_def}
J_{m_1,m_2,p,q}:= \max_{t>0} \br{(m_1-m_2)t \frac{p+q}{2} - m_1 \log\br{q e^t + 1-q} - m_2\log \br{pe^{-t}+1-p}}.
\end{align}
Then $J_{\min} = \min_{1\leq a\neq b \leq k} J_{n_a,n_b,p,q}$. For simplicity, we consider a two-community SBM.
In addition, instead of analyzing $\hat z$, we study a simplified procedure $\check z\in\{1,2\}^{n}$ defined as follows:
\begin{align*}
\check z_i: = \begin{cases}
1,\text{ if }\norm{A_{i\cdot}U^* - \theta^*_1}\leq \norm{A_{i\cdot} U^*- \theta^*_2},\\
2,\text{ o.w..}
\end{cases}
\end{align*}
We refer to $\check z$ as an oracle estimator since it involves  the unknown $U^*$ and $\theta^*_1,\theta^*_2$ and is not practical.
Nevertheless, $\check z$ is closely related to $\hat z$ as we elaborate below. Note that $\hat z$ performs clustering on rows of $\tilde A U$ as $\tilde A U = U\Lambda$. The $k$-means clustering (\ref{eqni:1}) implies $\hat z_i$ is equal to 1 if $\normt{\tilde A_{i\cdot} U - \hat \theta_1}\leq \normt{\tilde A_{i\cdot} U - \hat \theta_2}$ and is equal to 2 otherwise for each $i\in[n]$. As a comparison, $\check z$ has a similar form but with $\tilde A U$  replaced by $AU^*$ and $\hat \theta_1,\hat \theta_2$ replaced by $\theta^*_1,\theta^*_2$. Since $\theta^*_1,\theta^*_2$ are provided, $\check z$ is more of a classification procedure than a clustering method. Despite all these discrepancies, $\check z$ captures the key ingredient of $\hat z$ and analyzing $\check z$ reveals  fundamentally important properties of the spectral clustering.

The following proposition 
characterizes the statistical accuracy of the oracle estimator with both upper and lower bounds. 
From Proposition \ref{prop:exponent}, the oracle estimator has an exponential classification error with $J_{\min}$ being the exact asymptotic exponent. Though it only considers $k=2$ case, it can be generalized to multi-community cases with $J_{\min}$ appearing in the exponent.
\begin{proposition}\label{prop:exponent}
Consider a two-community SBM with community sizes $n_1,n_2$ both of the order  $n$. %
In addition, assume $0<q<p\leq 1/2$, $p,q$ are of the same order, and $\frac{n(p-q)^2}{p}\rightarrow\infty$, we have
\begin{align*}
&\E \br{ \frac{1}{n}\sum_{i\in[n]}\indic{\check z_i \neq z^*_i} }\leq \ebr{- (1-o(1))\br{J_{n_1,n_2,p,q}\wedge J_{n_2,n_1,p,q}}}\\
\text{and }\quad &\E \br{ \frac{1}{n}\sum_{i\in[n]}\indic{\check z_i \neq z^*_i} }\geq \ebr{- (1+o(1))\br{J_{n_1,n_2,p,q}\wedge J_{n_2,n_1,p,q}}}.
\end{align*}
\end{proposition}

To explain why $J_{n_1,n_2,p,q}\wedge J_{n_2,n_1,p,q}$ appears in  Proposition \ref{prop:exponent}, let us consider any $i\in[n]$  such that $z^*_i = 2$. By simple algebra (see proof of Proposition \ref{prop:exponent}), the event $\indic{\check z_i \neq z^*_i}$ can be written equivalently as
\begin{align}
 \indic{{\sum_{j:z^*_j=1} A_{ij} - \sum_{j\neq i:z^*_j=2}A_{ij} }\geq (n_1-n_2)\frac{p+q}{2} }.\label{eqni:2}
\end{align}
Note that for any $j\in[n]$, $A_{ij}$ is a Bernoulli random variable with probability $q$ if $z^*_j=1$ and $p$ if $z^*_j=2$ and $j\neq i$. In addition, $|\{j:z^*_j=1\}|=n_1$ and $|\{j\neq i:z^*_j=2\}|=n_2-1$. Let $\{X_i\}$ and $\{Y_j\}$ be independent Bernoulli random variables with probabilities $q,p$ respectively. 
 Then $\E \indic{\check z_i \neq z^*_i}$ is essentially equal to
\begin{align*}
\pbr{\sum_{l\in[n_1]} X_l - \sum_{j\in[n_2]} Y_j \geq \br{n_1 - n_2}\frac{p+q}{2}},
\end{align*}
where we ignore a minor difference between $n_2$ and $n_2-1$ which is negligible. For this tail probability, a direct application of Chernoff bound leads to an upper bound $\ebr{-J_{n_1,n_2,p,q}}$, and its lower bound $\ebr{-(1+o(1)) J_{n_1,n_2,p,q}}$ can be established using the Cram\'{e}r-Chernoff argument. Similarly, $J_{n_2,n_1,p,q}$ appears in the analysis for any $i\in[n]$ such that $z^*_i=1$. Between these two exponents, the smaller one of $J_{n_1,n_2,p,q}$ and $J_{n_2,n_1,p,q}$ dominates and leads to Proposition \ref{prop:exponent}.

\section{Main Results}

\subsection{Upper Bound}\label{sec:upper}
In this section, we present one main result of the paper: a sharp upper bound for the performance of the spectral clustering. Theorem \ref{thm:upper} is essentially the upper bound part of Theorem \ref{thm:introduction} but with an explicit formula  for the $o(1)$ term in the exponent.
\begin{theorem}\label{thm:upper}
Under the assumption that $\beta^{-1}$, $k=O(1)$, $0<q< p\leq 1/10$,  $\frac{p}{q}=O(1)$, and $\frac{n(p-q)^2}{p}\rightarrow\infty$, there exists some constant $C>0$ such that 
\begin{align*}
\E \ell(\hat z,z^*)\leq \ebr{-\br{1-C\br{\log\br{\frac{n(p-q)^2}{p}}}^{-\frac{1}{2}}} J_{\min}}+2n^{-3}.
\end{align*}
\end{theorem}
The upper bound in Theorem \ref{thm:upper}  involves two terms, the exponential term is essentially from $\E  \ell(\hat z,z^*)\indic{\mathf}$ where we study $\ell(\hat z,z^*)$ under the event $\mathf$ and the term $2n^{-3}$ comes from $\pbr{\mathf^\complement}\leq 2n^{-3}$ (see Lemma \ref{lem:tilde_A_concentration}). %
Since $2n^{-3}$ can be improved into $n^{-C}$ for any constant $C>0$ by scrutinizing the proof of Lemma  \ref{lem:tilde_A_concentration}, it should be generally ignored. By Lemma \ref{lem:J}, $J_{\min}$ can be shown to be of the order $n(p-q)^2/p $. Hence, $n(p-q)^2/p\rightarrow\infty$ is the sufficient condition for $\hat z$ to have a vanishing error. When $J_{\min}\geq (1+\epsilon) \log n$ for any constant $\epsilon>0$, Theorem \ref{thm:upper} immediately implies that $\hat z$ achieves the exact recovery with high probability. Theorem \ref{thm:upper} requires that the number of communities $k$ remains constant. However, our preliminary result in Proposition \ref{prop:prelim}, allows $k$ to increase with $n$, provided it does not grow faster than $(\frac{n(p-q)^2}{\beta^{-2} p})^\frac{1}{3}$. This more restrictive condition on $k$ in our main results arises from the limitations of our current proof techniques.

The proof of Theorem \ref{thm:upper} is quite involved. Below we give an overview of the proof and highlight challenges we face and new ideas and techniques we use to address them. Since $\tilde A U=U\Lambda$, the $k$-means clustering is performed on rows of $\tilde A U$ and $\hat z_i$ satisfies $\hat z_i = \argmin_{a\in[k]}\normt{\tilde A_{i\cdot} U - \hat \theta_a}$ for each $i\in[n]$.
Consider any node $i\in[n]$ and any $a\neq z^*_i$. The node will be incorrectly clustered if $\normt{\tilde A_{i\cdot} U - \hat \theta_a} \leq \normt{\tilde A_{i\cdot} U - \hat \theta_{z^*_i}}$ happens, an event that is about $(\tilde A_{i\cdot} - \E A_{i\cdot})U$. To analyze it, we further decompose it into two terms: one is essentially in a form of $( A_{i\cdot}-\E A_{i\cdot}) U^*$ and the other one is $(\tilde A_{i\cdot} - \E A_{i\cdot})(I-U^*U^{*T})U$.

The analysis for the term $( A_{i\cdot}-\E A_{i\cdot}) U^*$ is relatively straightforward. It eventually leads to the following event
\begin{align}
\mathbb{I}\Bigg\{{\sum_{j:z^*_j=a}  A_{ij} - \sum_{j:z^*_j=z^*_i,j\neq i} A_{ij}}  \geq (n_a -n_{z^*_i})\frac{p+q}{2} -o(1) (n_a + n_{z^*_i})\frac{p-q}{2}\Bigg\}, \label{eqni:3}
\end{align}
which mimics the event (\ref{eqni:2}) of the oracle estimator. 
The existence of $o(1)(n_a + n_{z^*_i})(p-q)/2$ is due to the discrepancy between the spectral clustering and the oracle estimator. The probability of the event (\ref{eqni:3}) leads to an upper bound $\ebr{-(1-o(1))J_{n_a,n_{z^*},p,q}}$. Going through all $i\in[n]$ and $a\in[k]$, the largest one  is $\ebr{-(1-o(1))J_{\min}}$ and appears in the upper bound of Theorem \ref{thm:upper}.

The proof of  Theorem \ref{thm:upper} mostly focuses on analyzing tail probabilities of  $(\tilde A_{i\cdot} - \E A_{i\cdot})(I-U^*U^{*T})U$ which is the most difficult and technical part toward establishing Theorem \ref{thm:upper}. There are two main challenges. The first challenge is that $\tilde A_{i\cdot}$ and $U$ are heavily dependent on each other. This challenge can be addressed by a refined leave-one-out technique to decouple the dependence. The second challenge is more critical and is the reason why we need to develop new techniques. For the purpose of illustration, let us consider a simplified setting: we want to study tail probabilities of $(X-\E X)u$ where $X\in\{0,1\}^{1\times n}$ with $X_i\iid \Ber(p)$ and is independent of $u$, a column of $(I-U^*U^{*T})U$. Existing literature \cite{chen2022partial, abbe2020entrywise, chen2019spectral} typically applies classical tail probabilities of Bernoulli random variables (e.g., Bernstein inequality) which involve both $\normt{u}$ and $\normt{u}_\infty$. Between these two norms, the former one is well-controlled but the latter one is essentially $\normt{U}_{2,\infty}$ that has some  inevitable $\log n$ factor. To deal with this $\log n$ factor and to derive meaningful tail probabilities, $p$ has to be much bigger than $(\log n)/n$. However, in this paper, we consider sparse networks with the connectivity probability allowed to be far smaller than $(\log n)/n$ and the aforementioned analysis breaks down.

To analyze $(\tilde A_{i\cdot} - \E A_{i\cdot})(I-U^*U^{*T})U$, we develop a new technical tool that avoids the use of $\norm{u}_{\infty}$ by truncating its coordinates. More accurately, 
to avoid dealing with $\norm{U}_{2,\infty}$, we truncate rows of $U$ if their $\ell_2$ norms are above a certain threshold. 
Let $t>0$ and define a mapping $f_t:\mathr^{1\times k}\rightarrow \mathr^{1\times k}$ such that 
\begin{align}\label{eqn:ft_def}
f_t(x) =\begin{cases}
x,\text{ if }\norm{x}\leq t,\\
tx/\norm{x},\text{ if }\norm{x}>t.
\end{cases}
\end{align}
That is, $f_t(x)$ truncates $x$ if $\norm{x}> t$ so that $\norm{f_t(x)}\leq t$ is always satisfied. If $k=1$, $f_t$ is a truncation operator such that $f_t(x)= x\indic{\abs{x}\leq t} + t\indic{x>t} -t\indic{x<-t}$ for any scalar $x$.
Apply $f_t$ to rows of $U$ and we get a matrix $\bar U\in\mathr^{n\times k}$ defined as
\begin{align}\label{eqn:proof_11}
\bar U_{i\cdot} := f_t(U_{i\cdot}),\forall i\in[n].
\end{align}
The definition of $\bar U$ ensures $\normt{\bar U}_{2,\infty}\leq t$.
Note that $\bar U$ depends on the value of $t$. In the proof of Theorem \ref{thm:upper}, we set  $t=t_0$ where $t_0$ is defined later in (\ref{eqn:t_0}).
With $\bar U$, we decompose $(\tilde A_{i\cdot} - \E A_{i\cdot})(I-U^*U^{*T})U$ into 
\begin{align}
\hspace*{-.2cm}(\tilde A_{i\cdot} - \E A_{i\cdot})(I-U^*U^{*T})U &= (\tilde A_{i\cdot} - \E A_{i\cdot})(I-U^*U^{*T})\bar U + (\tilde A_{i\cdot} - \E A_{i\cdot})(I-U^*U^{*T})(U-\bar U).\label{eqni:4}
\end{align}
We now use classical concentration inequalities to analyze $(\tilde A_{i\cdot} - \E A_{i\cdot})(I-U^*U^{*T})\bar U$ (after decoupling the dependence between $\tilde A_{i\cdot}$ and $\bar U$ by the leave-one-out technique) thanks to that the controlled $\normt{\bar U}_{2,\infty}$. It leads to negligible tail probabilities that can be absorbed into $\ebr{-(1-o(1))J_{\min}}$. Aggregating all $i\in[n]$, the impact of $ (\tilde A_{i\cdot} - \E A_{i\cdot})(I-U^*U^{*T})(U-\bar U)$, the second term in (\ref{eqni:4}), turns out to be related to a truncated $\ell_2$ norm $\sum_{i\in[n]}\norm{U_{i\cdot}}^2\indic{\norm{U_{i\cdot}} \geq t_0}$. According to Theorem \ref{thm:truncated_l2} below, this truncated $\ell_2$ norm is exponentially small. With the help of Theorem \ref{thm:truncated_l2}, the contribution of $ \{(\tilde A_{i\cdot} - \E A_{i\cdot})(I-U^*U^{*T})(U-\bar U)\}$ toward $\ell(\hat z,z^*)$ is also negligible and can be absorbed into $\ebr{-(1-o(1))J_{\min}}$ as well. In this way, we derive the exponential term $\ebr{-(1-o(1))J_{\min}}$ in Theorem \ref{thm:upper}.

\subsection{Truncated $\ell_2$ Perturbation Analysis for Eigenspaces}\label{sec:truncated}
\begin{theorem}\label{thm:truncated_l2}
Define
\begin{align}\label{eqn:t_0}
t_0 :=   %
\frac{160k^2}{\beta}\sqrt{\frac{k}{\beta n}}.
\end{align}
 Under the assumption that $\beta^{-1},k=O(1),0<q<p\leq 1/2$, and $\frac{n(p-q)^2}{p}\rightarrow\infty$, we have
\begin{align*}
\E \br{\sum_{i\in[n]}\norm{U_{i\cdot}}^2\indic{\norm{U_{i\cdot}} \geq t_0}} \indic{\mathf}\leq t_0^2n\ebr{  -\frac{3n(p-q)^2}{2p}}.
\end{align*}
\end{theorem}

Theorem \ref{thm:truncated_l2} provides a truncated $\ell_2$ norm for $U$ and is crucial to establishing  Theorem \ref{thm:upper}. The choice of $t_0$ in (\ref{eqn:t_0}) is carefully chosen. Note that $\normt{U^*}_{2,\infty}=  \sqrt{k/(\beta n)}$ according to Lemma \ref{lem:U_star}. When $\beta, k$ are both constants, $t_0$ is equal to $\normt{U^*}_{2,\infty}$ multiplied by a large constant. Since $U$ can be seen as a perturbation of $U^*$,  the truncated $\ell_2$ norm can also be interpreted as a perturbation bound for $U$ and $U^*$. Particularly, it focuses on rows of $U$ with norms that far exceed the baseline $\normt{U^*}_{2,\infty}$.
Theorem \ref{thm:truncated_l2} immediately implies the number of such rows is exponentially small.

\begin{corollary}\label{cor:truncated_hamming}
Under the same assumption as in Theorem \ref{thm:truncated_l2},  we have
\begin{align*}
\E  \br{\sum_{i\in[n]}\indic{\norm{U_{i\cdot}}\geq  t_0}} \indic{\mathf} \leq n\ebr{  -\frac{3n(p-q)^2}{2p}}.
\end{align*}
\end{corollary}

Theorem \ref{thm:truncated_l2} and Corollary \ref{cor:truncated_hamming} together provide insights on why the eigenvector truncation idea works in analyzing $(\tilde A_{i\cdot} - \E A_{i\cdot})(I-U^*U^{*T})U$ in the proof of Theorem \ref{thm:upper}.  Let us revisit the simplified setting $(X-\E X)u$ discussed above. The magnitude of $\norm{u}_\infty$ does break down classical concentration inequalities. However, these concentration inequalities are usually sharp for the worst case: their upper bounds hold for all weight vectors $u$ with given $\norm{u}$ and $\norm{u}_\infty$. 
On the contrary, the particular $u$ appearing in our analysis is very special: it has a small number of high magnitude coordinates (Corollary \ref{cor:truncated_hamming}) which have an exponentially small aggregated $\ell_2$ norm (Theorem \ref{thm:truncated_l2}), an important property completely ignored if classical concentration inequalities are applied. Instead, in our analysis, we fully exploit this structure by singling out high-magnitude coordinates of $u$ using the eigenvector truncation and eliminate the use of $\norm{u}_\infty$. This is the intuition behind (\ref{eqni:4}) which is crucial to proving Theorem \ref{thm:upper}. The idea of eigenvector truncation, the decomposition (\ref{eqni:4}), and Theorem \ref{thm:truncated_l2} might be useful for spectral perturbation analysis of other random matrix problems and models.

The proof of Theorem \ref{thm:truncated_l2} is also complicated. Since $U=\tilde AU \Lambda^{-1}$, for each $i\in[n]$, the event $\norm{U_{i\cdot}}^2\indic{\norm{U_{i\cdot}} \geq t_0}$ is essentially also about $(\tilde A_{i\cdot}-\E A_{i\cdot}) U$, the same as in the proof of Theorem \ref{thm:upper}.  As a result, 
we face the same challenges that appear in the analysis of Theorem \ref{thm:upper}, and we address them with the same eigenvector truncation arguments and techniques. Similar to how we establish Theorem \ref{thm:upper}, we decompose $(\tilde A_{i\cdot}-\E A_{i\cdot}) U$ into two parts: one eventually leads to a well-controlled tail probability of  $(A_{i\cdot}-\E A_{i\cdot}) U^*$, and the other one involves $(\tilde A_{i\cdot} - \E A_{i\cdot})(I-U^*U^{*T})U$. For the latter one, we use (\ref{eqni:4}) again, handle the quantity $(\tilde A_{i\cdot} - \E A_{i\cdot})(I-U^*U^{*T})\bar U $ with the help of the leave-one-out technique, and aggregate   $(\tilde A_{i\cdot} - \E A_{i\cdot})(I-U^*U^{*T})\bar U$ across all $i\in[n]$ which leads to $\sum_{i\in[n]}\norm{U_{i\cdot}}^2\indic{\norm{U_{i\cdot}} \geq t_0}$. The remaining part of the proof is different from that of Theorem \ref{thm:upper}.
So far, we obtain an inequality:
 its left-hand side is our target of Theorem \ref{thm:truncated_l2}, the truncated $\ell_2$ norm $\sum_{i\in[n]}\norm{U_{i\cdot}}^2\indic{\norm{U_{i\cdot}}\geq t_0}$; its right-hand side involves various terms with one being the truncated $\ell_2$ norm as well. The one appearing on the right-hand side (from the aggregation of $\{(\tilde A_{i\cdot} - \E A_{i\cdot})(I-U^*U^{*T})\bar U\}$) can be shown to have a small constant factor and consequently can be absorbed into the one on the left-hand side. The inequality then becomes an upper bound for the truncated $\ell_2$ norm and leads to Theorem \ref{thm:truncated_l2}.

\subsection{Lower Bound}\label{sec:lower}

Following the proof of Theorem \ref{thm:upper} with some modifications, we are able the obtain a matching lower bound presented below in Theorem \ref{thm:lower}. It corresponds to the lower bound part of Theorem \ref{thm:introduction}. 
\begin{theorem}\label{thm:lower}
Under the same assumption as in Theorem \ref{thm:upper}, there exists some constant $C'>0$ such that 
\begin{align*}
\E \ell(\hat z,z^*)\geq \ebr{-\br{1+C' \br{\log\br{\frac{n(p-q)^2}{p}}}^{-\frac{1}{4}}} J_{\min}}- 2n^{-3}.
\end{align*}
\end{theorem}

By Lemma \ref{lem:J}, $J_{\min}$ is of the order $n(p-q)^2/p $. Theorem \ref{thm:lower} indicates $n(p-q)^2/p\rightarrow\infty$ is the necessary condition for $\hat z$ to have a vanishing error. When $J_{\min}\leq  (1-\epsilon) \log n$ for any constant $\epsilon>0$, Theorem \ref{thm:upper} immediately implies that $\hat z$ has a constant probability of not achieving the exact recovery. 

Theorem \ref{thm:lower} complements Theorem \ref{thm:upper} and together they provide a precise characterization of the performance of the spectral clustering $\hat z$, which is referred to as its fundamental limit in this paper, demonstrating its ability and limit. This differs from the information-theoretical perspective on SBMs, referred to  as fundamental limits of SBMs  in \cite{abbe2017community} and minimax rates of SBMs in \cite{zhang2016minimax}. Specifically, \cite{abbe2017community} establishes a sharp threshold for exact community recovery, delineating a boundary where below it, no algorithm can succeed, and above it, there exists an algorithm capable of exact recovery.  \cite{zhang2016minimax} gives the minimax risk  of the community detection problem: $\inf_{\tilde z} \sup_{z^*} \E \ell(\tilde z,z^*)$ where the infimum is over all possible algorithms. In contrast, our focus is on the spectral clustering algorithm $\hat z$, with Theorems \ref{thm:upper} and \ref{thm:lower} providing bounds for $\E \ell(\hat z,z^*)$.

From Theorems \ref{thm:upper} and \ref{thm:lower}, the efficacy of $\hat z$ for exact recovery hinges on whether $J_{\min}/\log n > 1+ \epsilon$ or $J_{\min}/\log n < 1- \epsilon$, for some constant $\epsilon > 0$. For SBMs where all connectivity probabilities are of order $(\log n)/n$, \cite{abbe2017community} showed that exact recovery is feasible by some algorithm if and only if the Chernoff-Hellinger divergence $D_+ > 1$. Although both $J_{\min}/\log n$ and $D_+$ relate to tail probabilities of Bernoulli random variables, they are distinct; the former is specific to the performance of spectral clustering, while the latter arises from the information-theoretic analysis.

\section{Proofs of Theorem \ref{thm:upper} and Theorem \ref{thm:truncated_l2}}

In this section, we give detailed proofs of Theorem \ref{thm:upper} and Theorem \ref{thm:truncated_l2}. 
Before giving the proofs,
 we first introduce some useful concepts and tools and state some important properties.

~\\
\emph{Preliminary results from Proposition \ref{prop:prelim}.}
Recall the definition of the event $\mathf$ in (\ref{eqn:event_f_def}).  
Under this event, Proposition \ref{prop:prelim} gives a preliminary analysis for the performance of the spectral clustering, showing that there exists some constant $C_2>0$ and some $\phi\in\Phi$ such that (\ref{eqn:prop_prelim_1}) and (\ref{eqn:prop_prelim_2}) hold as   $\frac{n(p-q)^2}{\beta ^{-2}k^3p}\rightarrow\infty$. Without loss of generality, we can assume $\phi=\text{Id}$. Then (\ref{eqn:prop_prelim_2}) leads to
\begin{align}\label{eqnp:upper_2}
 \max_{a\in[k]} \norm{\hat \theta_{a} U^T - \theta^*_a U^{*T}}&\leq C_2 \beta^{-0.5} k\sqrt{p}.
\end{align}
In addition, we have
\begin{align}\label{eqnpn:33}
\E \ell(\hat z,z^*)\indic{\mathf}=\E\frac{1}{n}\sum_{i\in[n]} \indic{\hat z_i \neq z^*_i}\indic{\mathf}.
\end{align}

~\\
\emph{Facts about $U$.}
Define 
\begin{align*}
S:=\cbr{i\in[n]: d_i\geq \tau},
\end{align*}
to be the set that includes all indexes of rows and columns zeroed out in the spectral clustering where the value of $\tau$ is given in Algorithm \ref{alg:1}.
Then $\tilde A$ can be written equivalently as
\begin{align*}
\tilde A_{ij}=A_{ij}\indic{i\notin S\text{ and }j\notin S},\forall i,j\in[n].
\end{align*}
Since $U$ is the leading eigenspace of $\tilde A$, we have
\begin{align}\label{eqn:UH_1}
U_{i\cdot}=0,\forall i\in S,
\end{align}
and consequently $\norm{U_{i\cdot}}=0,\forall i\in S$. In addition, we have
\begin{align}\label{eqn:UH_3}
(\tilde A_{i\cdot} - \E A_{i\cdot})U  =  (A_{i\cdot}- \E A_{i\cdot}) U,\forall i\notin S.
\end{align}
This is because due to (\ref{eqn:UH_1}), the left-hand side of (\ref{eqn:UH_3}) is equal to $\sum_{j\notin S} (\tilde A_{ij}-\E A_{ij}) U_{j\cdot}=\sum_{j\notin S}  (A_{ij}-\E A_{ij}) U_{j\cdot}$ which is equal to the right-hand side of (\ref{eqn:UH_3}).

~\\
\emph{Row-wise truncation of $U$.} In the proofs of the main theorems, we need to study quantities that involve rows of $A$ and $U$. To avoid dealing with $\norm{U}_{2,\infty}$, we  truncate rows of $U$ if their $\ell_2$ norms are above a certain threshold. 
Recall the definitions of $f_t$ and $\bar U$ in (\ref{eqn:ft_def}) and (\ref{eqn:proof_11}).
Since $\norm{f_t(x) - x}\leq \norm{x}\indic{\norm{x}\geq t}$, we have
\begin{align}\label{eqnpn:1}
\normf{U - \bar U}^2 &= \sum_{i\in[n]} \norm{U_{i\cdot} - \bar U_{i\cdot}}^2 \leq \sum_{i\in[n]} \norm{U_{i\cdot}}^2 \indic{\norm{U_{i\cdot}} \geq t}.
\end{align}
Note that $\bar U$ depends on the value of $t$. In the proofs of the main theorems, we will first consider any $t>0$ and then set $t=t_0$.

~\\
\emph{Leave-one-out counterparts of $U$ and $\bar U$.} One challenge is that rows of $\tilde A$ and $U$ are not independent of each other. To decouple the dependence,
we  introduce leave-one-out counterparts of $U$ and $\bar U$. For any $i\in[n]$,  define $\tilde A^{(i)}\in \{0,1\}^{n\times n}$ to be the leave-one-out counterpart of $\tilde A$ such that
\begin{align}\label{eqn:proof_15}
\tilde A_{jl}^{(i)} := \tilde A_{jl}\indic{j\neq i\text{ and }l\neq i},\forall j,l\in[n].
\end{align}
That is, $\tilde A^{(i)}$ is obtained from $\tilde A$ by zeroing out its $i$th row and column. Let $U^{(i)}\in\mathr^{n\times k}$ be the leading  eigenspace of $\tilde A^{(i)}$. Under the event $\mathf$ and $\frac{n(p-q)^2}{\beta ^{-2}k^4p}\rightarrow\infty$,  conditions in Lemma \ref{lem:UU_difference} are satisfied, which leads to
\begin{align}\label{eqn:proof_6}
\fnorm{U -U^{(i)}W^{(i)}}\leq 6k^{1.5}\norm{U_{i\cdot}},
\end{align}
for some orthogonal matrix $W^{(i)}\in\matho(k,k)$. Apply $f_t$ to rows of $U^{(i)}$ and we have $\bar U^{(i)}\in\mathr^{n\times k}$:
\begin{align}\label{eqn:proof_26}
\bar U_{j\cdot}^{(i)} := f_t( U_{j\cdot}^{(i)}),\forall j\in[n].
\end{align}
Since $W^{(i)}$ is an orthogonal matrix, for any $j\in[n]$, we have $\normt{U_{j\cdot}^{(i)}W^{(i)}} = \normt{U_{j\cdot}^{(i)}}$ and consequently,
$\bar U_{j\cdot}^{(i)} W^{(i)} = f_t(U_{j\cdot}^{(i)}W^{(i)})$. By Lemma \ref{lem:ft}, we have $\normt{\bar U_{j\cdot}   - \bar U_{j\cdot}^{(i)}W^{(i)}} = \normt{f_t(U_{j\cdot}) - f_t(U_{j\cdot}^{(i)}W^{(i)})} \leq \normt{U_{j\cdot} - U_{j\cdot}^{(i)}W^{(i)}}$. Then (\ref{eqn:proof_6}) gives
\begin{align}\label{eqnpn:2}
\fnorm{\bar U  -\bar U^{(i)}W^{(i)}} \leq \fnorm{U -U^{(i)}W^{(i)}} \leq 6k^{1.5}\norm{U_{i\cdot}}.
\end{align}
Similar to $\bar U$, $\bar U^{(i)}$ depends on the value of $t$. In the proofs of the main theorems, we will first consider any $t>0$ and then set $t=t_0$.

~\\
\emph{Decomposition of rows of $\tilde A$.} For each $i\in[n]$, $\tilde A_{i\cdot}$ and the leave-one-out eigenspace $\bar U^{(i)}$ still depend on each other, mainly due to the removal of high-degree  nodes and $S$.
To further decouple the dependence,
we need to have a decomposition for $\tilde A_{i\cdot}$.
Define a set 
\begin{align}\label{eqnpn:4}
 S_i:=\cbr{j\in[n]:j\neq i\text{ and }\sum_{l\neq i,j}A_{jl}\geq \tau-1}.
\end{align}
By its definition, $ S_i$ is independent of $A_{i\cdot}$. We also have 
\begin{align}\label{eqnpn:3}
S\cup\{i\} \subset S_i\cup \{i\} .
\end{align}
 This is because  for each $j\notin i$, we have $d_j - (\sum_{l\neq i,j}A_{jl})=A_{ji} =A_{ij}\in\{0,1\}$.  As a result, if $j\in S$, it must be in $S_i$ as well; if $j\in S_i$, it is in $S$ if and only if $A_{ij}=1$.   This leads to the independence between $S\cup\{i\}$ and coordinates in $A_{i\cdot}$ with indexes not in $S_i\cup\{i\}$, conditioned on the set $S_i\cup\{i\}$.  Define $\check A_{i\cdot},\check P_{i\cdot}$ such that
\begin{align}\label{eqnpn:7}
\check A_{ij} := A_{ij}\indic{j\neq i \text{ and }j\in S_i^\complement} \text{ and }\check P_{ij} := P_{ij}\indic{j\neq i \text{ and }j\in S_i^\complement},\forall j\in[n].
\end{align}
Then the aforementioned independence can be stated equivalently as
\begin{align*}
(\check A_{i\cdot} - \check P_{i\cdot}) \perp  S\cup\{i\} | S_i\cup \{i\}.
\end{align*}
By its definition in (\ref{eqn:proof_15}), $\tilde A^{(i)}$ is obtained from $A$ by zeroing out rows and columns in $S\cup\{i\}$. In addition, $\bar U^{(i)}$ is a function of $\tilde A^{(i)}$. Hence, 
\begin{align}\label{eqn:independence}
(\check A_{i\cdot} - \check P_{i\cdot}) \perp \tilde A^{(i)} | S_i\cup\{i\} \text{ and }(\check A_{i\cdot} - \check P_{i\cdot}) \perp \bar U^{(i)} | S_i\cup\{i\},
\end{align}
for all $i\in[n]$.

~\\
\emph{Decouple dependence between rows of $\tilde A$ and $\mathf$.} For each $i\in[n]$, $\check A_{i\cdot} - \check P_{i\cdot}$ introduced above and the event $\mathf$ are dependent on each other. To decouple this dependence, we introduce 
\begin{align}\label{eqnpn:5}
\mathf_i := \indic{\normt{\tilde A^{(i)} -\E A} \leq (C_0+2)\sqrt{np}}.
\end{align}
 Denote $[\E A]^{(i)}\in[0,1]^{n\times n}$ to be a matrix that is equal to $\E A$ but with its $i$th row and column zeroed out. Then $\tilde A^{(i)} - [\E A]^{(i)}$ is equal to $\tilde A -\E A$   but with its $i$th row and column zeroed out. Then we have $\normt{\tilde A^{(i)} - [\E A]^{(i)}}\leq \normt{\tilde A -\E A}$ and consequently $\normt{\tilde A^{(i)} -\E A}\leq \normt{\tilde A -\E A} + \norm{\E A - [\E A]^{(i)}}\leq \normt{\tilde A -\E A} + 2 \norm{\E A_{i\cdot}}\leq \normt{\tilde A -\E A} + 2\sqrt{n}p $.  Hence
\begin{align*}
\indic{\mathf}\leq \indic{\mathf_i},\forall i\in[n].
\end{align*}
Since $\mathf_i$ is a function of $\tilde A^{(i)}$, from (\ref{eqn:independence}) we have
\begin{align}\label{eqnpn:6}
(\check A_{i\cdot} - \check P_{i\cdot}) \perp \mathf_i | S_i\cup\{i\} ,
\end{align}
for all $i\in[n]$. By Lemma \ref{lem:bar_U_i_U_star_diff}, we have
\begin{align}\label{eqnpn:8}
\norm{\br{I_n - U^*U^{*T}}\bar U^{(i)}}_{2,\infty}\leq t + \sqrt{\frac{k}{\beta n}}
\end{align}
for all $i\in[n]$.
In addition, under the event $\mathf_i$,  by Lemma \ref{lem:bar_U_i_U_star_diff},  as $\frac{n(p-q)^2}{\beta^{-2}k^2p}\rightarrow\infty$, we have
\begin{align}
\fnorm{\br{I_n - U^*U^{*T}} \bar U^{(i)}} \leq   \frac{2\sqrt{2}(C_0+3)k\sqrt{knp}}{\beta n(p-q)}\label{eqn:proof_22}
\end{align}
for all $i\in[n]$ when $t>\sqrt{k/(\beta n)}$.

\subsection{Proof of Theorem \ref{thm:upper}}\label{sec:proof_upper}
Note that
\begin{align}
\E \ell(\hat z,z^*)&\leq \E \ell(\hat z,z^*)\indic{\mathf} + \pbr{\mathf^\complement} \nonumber\\
&\leq  \E  \frac{1}{n}\sum_{i\in[n]}\indic{\hat z_i\neq z^*_i}\indic{\mathf}  + 2n^{-3} \nonumber\\
&=\E  \frac{1}{n}\sum_{i\notin S}\indic{\hat z_i\neq z^*_i}\indic{\mathf} +  \frac{1}{n}\E \abs{S} + 2n^{-3} \nonumber\\
&\leq  \frac{1}{n}  \E \sum_{i\notin S}\indic{\hat z_i\neq z^*_i}\indic{\mathf} +   \ebr{-128np} + 2n^{-3},\label{eqnpn:27}
\end{align}
where the second inequality is by (\ref{eqnpn:33}) and 
the last inequality is by Lemma \ref{lem:S_size}. In the following proof, we are going to establish an upper bound for $ \E \sum_{i \notin S}\indic{\hat z_i\neq z^*_i}\indic{\mathf}$.

Now consider a fixed $i\notin S$. According to the objective function of the $k$-means, we have the following inequality
\begin{align}
\indic{\hat z_i \neq z_i^*} \indic{\mathf} &\leq  \indic{\min_{a\neq z^*_i}\norm{U_{i\cdot}\Lambda - \hat \theta_a} \leq \norm{U_{i\cdot}\Lambda - \hat \theta_{z^*_i}}}\indic{\mathf}\nonumber\\
&\leq \sum_{a\neq z^*_i}\indic{\norm{U_{i\cdot}\Lambda - \hat \theta_a}\leq \norm{U_{i\cdot}\Lambda - \hat \theta_{z^*_i}}}\indic{\mathf}. \label{eqnpn:24}
\end{align}

Further consider a fixed $a\in[k]\setminus\{z^*_i\}$.
 We are going to study the event that $U_{i\cdot}\Lambda $ is closer to  an incorrect center $\hat \theta_a$ than the correct center $\hat \theta_{z^*_i}$.  We have
 \begin{align}\label{eqnp:upper_1}
&\indic{\norm{U_{i\cdot}\Lambda - \hat \theta_a}\leq \norm{U_{i\cdot}\Lambda - \hat \theta_{z^*_i}}} \indic{\mathf}  =  \indic{\iprod{U_{i\cdot}\Lambda - \hat \theta_{z^*_i}}{ \hat \theta_a - \hat \theta_{z^*_i}} \geq \frac{1}{2} \norm{ \hat \theta_a - \hat \theta_{z^*_i}}^2}\indic{\mathf}.
 \end{align}
 We are going to have a decomposition of $\langle U_{i\cdot}\Lambda - \hat \theta_{z^*_i}, \hat \theta_a - \hat \theta_{z^*_i}\rangle$.
Due to $U\Lambda = \tilde AU$, we have $U_{i\cdot}\Lambda = \tilde A_{i\cdot} U =  (\tilde A_{i\cdot}- \E A_{i\cdot}) U +  (\E A_{i\cdot}) U$. We have
\begin{align}
\iprod{U_{i\cdot}\Lambda - \hat \theta_{z^*_i}}{ \hat \theta_a - \hat \theta_{z^*_i}} &= \iprod{(\tilde A_{i\cdot}- \E A_{i\cdot}) U}{ \hat \theta_a - \hat \theta_{z^*_i}} + \iprod{ (\E A_{i\cdot}) U- \hat \theta_{z^*_i}}{ \hat \theta_a - \hat \theta_{z^*_i}} \nonumber\\
&=(\tilde A_{i\cdot}- \E A_{i\cdot}) U \br{ \hat \theta_a - \hat \theta_{z^*_i}}^T + \iprod{ (\E A_{i\cdot}) U- \hat \theta_{z^*_i}}{ \hat \theta_a - \hat \theta_{z^*_i}} \nonumber\\
&=( A_{i\cdot}- \E A_{i\cdot}) U \br{ \hat \theta_a - \hat \theta_{z^*_i}}^T + \iprod{ (\E A_{i\cdot}) U- \hat \theta_{z^*_i}}{ \hat \theta_a - \hat \theta_{z^*_i}} \nonumber\\
&=( A_{i\cdot}- \E A_{i\cdot}) U^* \br{\theta^*_a - \theta^*_{z^*_i}}^T  + ( A_{i\cdot}- \E A_{i\cdot}) \br{U\hat \theta_a^T - U^* \theta_a^{*T} - U\hat \theta_{z^*_i}^T + U^* \theta^{*T}_{z^*_i}} \nonumber \\
&\quad + \iprod{ (\E A_{i\cdot}) U- \hat \theta_{z^*_i}}{ \hat \theta_a - \hat \theta_{z^*_i}},\label{eqnp:upper_4}
\end{align}
where the third equation is by (\ref{eqn:UH_3}). 
Using Lemma \ref{lem:P_line_diff},  the first term in (\ref{eqnp:upper_4}) can be simplified
\begin{align*}
( A_{i\cdot}- \E A_{i\cdot}) U^* \br{\theta^*_a - \theta^*_{z^*_i}}^T 
&= (p-q)\br{\sum_{j:z^*_j=a} ( A_{ij} -\E A_{ij}) - \sum_{j:z^*_j=z^*_i,j\neq i} ( A_{ij} -\E A_{ij}) }.
\end{align*}
Together with (\ref{eqnp:upper_1}) and (\ref{eqnp:upper_4}), we have
\begin{align}
&\indic{\norm{U_{i\cdot}\Lambda - \hat \theta_a}\leq \norm{U_{i\cdot}\Lambda - \hat \theta_{z^*_i}}} \indic{\mathf} \nonumber \\
&=\mathbb{I}\bigg\{(p-q)\br{\sum_{j:z^*_j=a} ( A_{ij} -\E A_{ij}) - \sum_{j:z^*_j=z^*_i,j\neq i} ( A_{ij} -\E A_{ij}) } \nonumber \\
&\quad + ( A_{i\cdot}- \E A_{i\cdot}) \br{U\hat \theta_a^T - U^* \theta_a^{*T} - U\hat \theta_{z^*_i}^T + U^* \theta^{*T}_{z^*_i}} \nonumber \\
&\quad + \iprod{ (\E A_{i\cdot}) U- \hat \theta_{z^*_i}}{ \hat \theta_a - \hat \theta_{z^*_i}}\geq \frac{1}{2} \norm{ \hat \theta_a - \hat \theta_{z^*_i}}^2\bigg\}\indic{\mathf}.\label{eqnp:upper_5}
\end{align}

Next, we are going to give an upper bound for (\ref{eqnp:upper_5}). Since $\E A_{i\cdot} $ and $P_{i\cdot}$ only differ in their $i$th coordinates with $A_{ii}=0$ and $P_{ii}=p$, we have 
\begin{align*}
 \norm{ (\E A_{i\cdot}) U- \hat \theta_{z^*_i}} &=   \norm{ P_{i\cdot}U - P_{ii}U_{i\cdot}- \hat \theta_{z^*_i}} \leq \norm{ P_{i\cdot}U - \hat \theta_{z^*_i}}  + p\norm{U_{i\cdot}}  =  \norm{ \theta^*_{z^*_i}U^{*T}U- \hat \theta_{z^*_i}}+ p\norm{U_{i\cdot}} \\
 & =   \norm{ \theta^*_{z^*_i}U^{*T}U- \hat \theta_{z^*_i}U^TU}+ p\norm{U_{i\cdot}}  = \norm{ \theta^*_{z^*_i}U^{*T}- \hat \theta_{z^*_i}U^T}+ p\norm{U_{i\cdot}} \\
 &\leq \max_{b\in[k]}\norm{U\hat \theta_b^T - U^* \theta_b^{*T}} + p\leq C_2 \beta^{-0.5} k\sqrt{p} + p \leq 2C_2 \beta^{-0.5} k\sqrt{p},
\end{align*}
where the second to last inequality is due to (\ref{eqnp:upper_2}).
We also have
\begin{align*}
&\norm{ ( A_{i\cdot}- \E A_{i\cdot}) \br{U\hat \theta_a^T - U^* \theta_a^{*T} - U\hat \theta_{z^*_i}^T + U^* \theta^{*T}_{z^*_i}}}\\
&\leq  \norm{ ( A_{i\cdot}- \E A_{i\cdot})U^*U^{*T} \br{U\hat \theta_a^T - U^* \theta_a^{*T} - U\hat \theta_{z^*_i}^T + U^* \theta^{*T}_{z^*_i}}} \\
&\quad + \norm{ ( A_{i\cdot}- \E A_{i\cdot})(I_n-U^*U^{*T}) \br{U\hat \theta_a^T - U^* \theta_a^{*T} - U\hat \theta_{z^*_i}^T + U^* \theta^{*T}_{z^*_i}}}\\
&\leq  2 \norm{ ( A_{i\cdot}- \E A_{i\cdot})U^*}\max_{b\in[k]}\norm{U\hat \theta_b^T - U^* \theta_b^{*T}} +   \norm{ ( A_{i\cdot}- \E A_{i\cdot})(I_n-U^*U^{*T}) \br{U\hat \theta_a^T - U\hat \theta_{z^*_i}^T }} \\
&\leq  2C_2 \beta^{-0.5} k\sqrt{p} \norm{ ( A_{i\cdot}- \E A_{i\cdot})U^*} +   \norm{ ( A_{i\cdot}- \E A_{i\cdot})(I_n-U^*U^{*T}) U} \norm{\hat \theta_a - \hat \theta_{z^*_i}},
\end{align*}
where in the last inequality we use (\ref{eqnp:upper_2}) again.
Then (\ref{eqnp:upper_5}) leads to
\begin{align}
&\indic{\norm{U_{i\cdot}\Lambda - \hat \theta_a}\leq \norm{U_{i\cdot}\Lambda - \hat \theta_{z^*_i}}} \indic{\mathf} \nonumber \\
&\leq \mathbb{I}\Bigg\{(p-q)\br{\sum_{j:z^*_j=a} ( A_{ij} -\E A_{ij}) - \sum_{j:z^*_j=z^*_i,j\neq i} ( A_{ij} -\E A_{ij}) }  + 2C_2 \beta^{-0.5} k\sqrt{p}\norm{ ( A_{i\cdot}- \E A_{i\cdot})U^*}  \nonumber\\
&\quad +\norm{ ( A_{i\cdot}- \E A_{i\cdot})(I_n-U^*U^{*T}) U} \norm{\hat \theta_a - \hat \theta_{z^*_i}} + 2C_2 \beta^{-0.5} k\sqrt{p}\norm{ \hat \theta_a - \hat \theta_{z^*_i}}\nonumber\\
& \quad \geq \frac{1}{2} \norm{ \hat \theta_a - \hat \theta_{z^*_i}}^2\Bigg\}\indic{\mathf}.\label{eqnp:upper_3}
\end{align}
To further simplify the above display, we need to study $\normt{ \hat \theta_a - \hat \theta_{z^*_i}}$. From Lemma \ref{lem:population}, we have $\normt{\theta^*_a -\theta^*_{z^*_i}} =\sqrt{n_a+n_{z^*_i}}(p-q) \geq \sqrt{2\beta n/k}(p-q)$.
Under the event $\mathf$, from  (\ref{eqnp:upper_2}), we have
\begin{align}
 \norm{ \hat \theta_a - \hat \theta_{z^*_i}} &= \norm{ \hat \theta_a U^T- \hat \theta_{z^*_i}U^T }  \nonumber\\
 &=  \norm{ \hat \theta_a U^T - \theta^*_aU^{*T} +  \theta^*_aU^{*T} -  \theta^*_{z^*_i}U^{*T} +  \theta^*_{z^*_i}U^{*T}- \hat \theta_{z^*_i}U^T }  \nonumber\\
 &\geq  \norm{ \theta^*_aU^{*T} -  \theta^*_{z^*_i}U^{*T}} - 2\max_{b\in[k]} \norm{\theta^*_bU^{*T}- \hat \theta_{b}U^T }  \nonumber\\
 & = \norm{ \theta^*_a -  \theta^*_{z^*_i}} - 2\max_{b\in[k]} \norm{\theta^*_bU^{*T}- \hat \theta_{b}U^T }  \nonumber\\
 &\geq \sqrt{n_a + n_{z^*_i}}(p-q) -2C_2 \beta^{-0.5} k\sqrt{p}, \label{eqnp:upper_8}
\end{align}
which can be further lower bounded by $\frac{1}{2} \sqrt{\frac{2\beta n}{k}} (p-q)$ as $\frac{n(p-q)^2}{\beta^{-2}k^3p}\rightarrow\infty$. 
Then,
\begin{align*}
&\frac{1}{2} \norm{ \hat \theta_a - \hat \theta_{z^*_i}}^2 - 2C_2 \beta^{-0.5} k\sqrt{p}\norm{ \hat \theta_a - \hat \theta_{z^*_i}} \\
& = \frac{1}{2}\br{1- \frac{2C_2 \beta^{-0.5} k\sqrt{p}}{\norm{ \hat \theta_a - \hat \theta_{z^*_i}}}}\norm{ \hat \theta_a - \hat \theta_{z^*_i}}^2\\
&\geq \frac{1}{2} \br{1- \frac{2C_2 \beta^{-0.5} k\sqrt{p}}{\frac{1}{2} \sqrt{\frac{2\beta n}{k}} (p-q)}}\br{\sqrt{n_a + n_{z^*_i}}(p-q) -2C_2 \beta^{-0.5} k\sqrt{p}}^2\\
&\geq \frac{1}{2}\br{1-C_3\sqrt{\frac{\beta^{-2}k^3p}{n(p-q)^2}}} (n_a + n_{z^*_i})(p-q)^2,
\end{align*}
for some constant $C_3>0$, where the last inequality holds as  $\frac{n(p-q)^2}{\beta^{-2}k^3p}\rightarrow\infty$.  Similar to (\ref{eqnp:upper_8}), we also have
\begin{align*}
 \norm{ \hat \theta_a - \hat \theta_{z^*_i}}  &\leq   \norm{ \theta^*_a -  \theta^*_{z^*_i}} + 2\max_{b\in[k]} \norm{\theta^*_bU^{*T}- \hat \theta_{b}U^T } \\
 &\leq \sqrt{n_a + n_{z^*_i}}(p-q) +2C_2 \beta^{-0.5} k\sqrt{p}\\
 &\leq 1.1 \sqrt{n_a + n_{z^*_i}}(p-q) ,
\end{align*}
where the last inequality is under the assumption $\frac{n(p-q)^2}{\beta^{-2}k^3p}\rightarrow\infty$. Then (\ref{eqnp:upper_3}) becomes
\begin{align}
&\indic{\norm{U_{i\cdot}\Lambda - \hat \theta_a}\leq \norm{U_{i\cdot}\Lambda - \hat \theta_{z^*_i}}} \indic{\mathf} \nonumber \\
&\leq \mathbb{I}\Bigg\{(p-q)\br{\sum_{j:z^*_j=a} ( A_{ij} -\E A_{ij}) - \sum_{j:z^*_j=z^*_i,j\neq i} ( A_{ij} -\E A_{ij}) }  + 2C_2 \beta^{-0.5} k\sqrt{p}\norm{ ( A_{i\cdot}- \E A_{i\cdot})U^*}  \nonumber\\
&\quad + 1.1 \sqrt{n_a + n_{z^*_i}}(p-q)\norm{ ( A_{i\cdot}- \E A_{i\cdot})(I_n-U^*U^{*T}) U}\nonumber\\
& \quad \geq \frac{1}{2}\br{1-C_3\sqrt{\frac{\beta^{-2}k^3p}{n(p-q)^2}}} (n_a + n_{z^*_i})(p-q)^2\Bigg\}\indic{\mathf}.\label{eqnp:upper_10}
\end{align}

We are going to split the indication function in (\ref{eqnp:upper_10}) into several ones using the following fact: $\indic{x_1 + x_2\geq x_3+x_4}\leq \indic{x_1\geq x_3} + \indic{x_2\geq x_4}$ for all $x_1,x_2,x_3,x_4\in\mathr$.
 Consider a positive sequence $\rho=o(1)$  whose value will be determined later. We have
\begin{align}
&\indic{\norm{U_{i\cdot}\Lambda - \hat \theta_a}\leq \norm{U_{i\cdot}\Lambda - \hat \theta_{z^*_i}}} \indic{\mathf} \nonumber \\
&\leq \mathbb{I}\Bigg\{{\sum_{j:z^*_j=a} ( A_{ij} -\E A_{ij}) - \sum_{j:z^*_j=z^*_i,j\neq i} ( A_{ij} -\E A_{ij}) }  \geq \frac{1}{2}\br{1-4\rho-C_3\sqrt{\frac{\beta^{-2}k^3p}{n(p-q)^2}}} (n_a + n_{z^*_i})(p-q)\Bigg\}\indic{\mathf}    \nonumber\\
&\quad +  \indic{\norm{ ( A_{i\cdot}- \E A_{i\cdot})U^*}  \geq \frac{\rho  (n_a + n_{z^*_i})(p-q)^2}{2C_2 \beta^{-0.5} k\sqrt{p}}} \indic{\mathf}  \nonumber\\
&\quad +  \indic{1.1 \norm{ ( A_{i\cdot}- \E A_{i\cdot})(I_n-U^*U^{*T}) U} \geq \rho  \sqrt{n_a + n_{z^*_i}}(p-q)  } \indic{\mathf}  \nonumber\\
&\leq \mathbb{I}\Bigg\{{\sum_{j:z^*_j=a} ( A_{ij} -\E A_{ij}) - \sum_{j:z^*_j=z^*_i,j\neq i} ( A_{ij} -\E A_{ij}) }  \geq \frac{1}{2}\br{1-4\rho-C_3\sqrt{\frac{\beta^{-2}k^3p}{n(p-q)^2}}} (n_a + n_{z^*_i})(p-q)\Bigg\}   \nonumber\\
&\quad +  \indic{\norm{ ( A_{i\cdot}- \E A_{i\cdot})U^*}  \geq \frac{\rho n (p-q)^2}{C_2 \beta^{-1.5} k^2\sqrt{p}}}  \nonumber\\
&\quad +  \indic{ \norm{ ( A_{i\cdot}- \E A_{i\cdot})(I_n-U^*U^{*T}) U} \geq \rho\sqrt{\frac{\beta n}{k}}(p-q)  } \indic{\mathf}  \nonumber\\
& =: G_{1,i,a} + G_{2,i} + G_{3,i} \label{eqnpn:25}
\end{align}
where $G_{2,i}$ and $G_{3,i}$ do not depend on $a$. 

To further  decompose $G_{3,i}$, first we replace $A_{i\cdot}$ by $\tilde A_{i\cdot}$ as
\begin{align*}
 ( A_{i\cdot}- \E A_{i\cdot})(I_n-U^*U^{*T}) U &= ( \tilde A_{i\cdot}- \E A_{i\cdot})(I_n-U^*U^{*T}) U  +  ( A_{i\cdot}- \tilde A_{i\cdot})(I_n-U^*U^{*T}) U  \\
 &= ( \tilde A_{i\cdot}- \E A_{i\cdot})(I_n-U^*U^{*T}) U  -  ( A_{i\cdot}- \tilde A_{i\cdot})U^*U^{*T} U,
\end{align*}
where we use (\ref{eqn:UH_3}) in the last inequality. Since $A_{ij} =\tilde A_{ij}$ for all $j\notin S$ and $ \tilde A_{ij}=0$ for all $j\in S$, we have
 $( A_{i\cdot}- \tilde A_{i\cdot})U^*U^{*T} U = \sum_{j\in S} A_{ij} U^*_{j\cdot} U^{*T} U$. As a result, $\normt{( A_{i\cdot}- \tilde A_{i\cdot})U^*U^{*T} U}= \normt{ \sum_{j\in S} A_{ij} U^*_{j\cdot} U^{*T} U} \leq \normt{ \sum_{j\in S} A_{ij} U^*_{j\cdot} } \leq ( \sum_{j\in S} A_{ij})\norm{U^*}_{2,\infty} \leq \sqrt{\frac{k}{\beta n}}  \sum_{j\in S} A_{ij} $ where the last inequality is due to (\ref{lem:U_star}). Hence,
\begin{align*}
\norm{ ( A_{i\cdot}- \E A_{i\cdot})(I_n-U^*U^{*T}) U} &\leq \norm{ ( \tilde A_{i\cdot}- \E A_{i\cdot})(I_n-U^*U^{*T}) U } + \sqrt{\frac{k}{\beta n}} \sum_{j\in S} A_{ij}.
\end{align*}
Recall that $\bar U$ and $\bar U^{(i)}$ are defined in (\ref{eqnpn:1}) and (\ref{eqn:proof_26}) by applying $f_t(\cdot)$ to rows of $U$ and $U^{(i)}$ where $t>0$.
For now, let us consider any $t>0$. We will set $t=t_0$ later in the proof.
Using $U = \br{U- \bar U} + \br{\bar U -\bar U^{(i)}  W^{(i)}} + \bar U^{(i)}  W^{(i)}$,  we have
\begin{align*}
\norm{ ( A_{i\cdot}- \E A_{i\cdot})(I_n-U^*U^{*T}) U} &\leq  \norm{ ( \tilde A_{i\cdot}- \E A_{i\cdot})(I_n-U^*U^{*T}) (U -\bar U) } \\
&\quad +   \norm{ (\tilde A_{i\cdot}- \E A_{i\cdot})(I_n-U^*U^{*T}) (\bar U -\bar U^{(i)}  W^{(i)})}\\
&\quad +  \norm{ ( \tilde  A_{i\cdot}- \E A_{i\cdot})(I_n-U^*U^{*T})  \bar U^{(i)}  W^{(i)}}+\sqrt{\frac{k}{\beta n}} \sum_{j\in S} A_{ij}.
\end{align*}
The third term in the above display can be further simplified:
\begin{align*}
&\norm{ ( \tilde  A_{i\cdot}- \E A_{i\cdot})(I_n-U^*U^{*T})  \bar U^{(i)}  W^{(i)}} = \norm{ ( \tilde  A_{i\cdot}- \E A_{i\cdot})(I_n-U^*U^{*T})  \bar U^{(i)}  } \\
& \leq  \norm{ ( \check  A_{i\cdot}- \check P_{i\cdot})(I_n-U^*U^{*T})  \bar U^{(i)}  } +  \norm{\br{ ( \tilde  A_{i\cdot}- \E A_{i\cdot})- ( \check  A_{i\cdot}- \check P_{i\cdot})}(I_n-U^*U^{*T})  \bar U^{(i)}  }\\
& \leq  \norm{ ( \check  A_{i\cdot}- \check P_{i\cdot})(I_n-U^*U^{*T})  \bar U^{(i)}  } + \norm{ ( \tilde  A_{i\cdot}- \E A_{i\cdot})- ( \check  A_{i\cdot}- \check P_{i\cdot})}_1 \norm{(I_n-U^*U^{*T})  \bar U^{(i)} }_{2,\infty},
\end{align*}
where the last inequality is by H\"{o}lder's inequality.
Note that $(\tilde A_{ij}-\E A_{ij})-(\check A_{ij}-\check P_{ij}) =( A_{ij}-\E A_{ij})\indic{j\in S^\complement \cap S_i}$ for each $j\in[n]$. Together with (\ref{eqnpn:8}), we have
\begin{align*}
&\norm{ ( \tilde  A_{i\cdot}- \E A_{i\cdot})(I_n-U^*U^{*T})  \bar U^{(i)}  W^{(i)}} \\
&\quad \leq  \norm{ ( \check  A_{i\cdot}- \check P_{i\cdot})(I_n-U^*U^{*T})  \bar U^{(i)}  }  + \sum_{j\in S^C\cap S_i} \abs{A_{ij}-\E A_{ij}} \br{t+\sqrt{\frac{k}{\beta n}}}\\
&\quad \leq  \norm{ ( \check  A_{i\cdot}- \check P_{i\cdot})(I_n-U^*U^{*T})  \bar U^{(i)}  }  + \sum_{j\in S_i} \abs{A_{ij}-\E A_{ij}} \br{t+\sqrt{\frac{k}{\beta n}}}.
\end{align*}
Hence, the above displays lead to
\begin{align}
&\norm{ ( A_{i\cdot}- \E A_{i\cdot})(I_n-U^*U^{*T}) U} \nonumber\\
&\leq  \norm{ ( \tilde A_{i\cdot}- \E A_{i\cdot})(I_n-U^*U^{*T}) (U -\bar U) }  +   \norm{ (\tilde A_{i\cdot}- \E A_{i\cdot})(I_n-U^*U^{*T}) (\bar U -\bar U^{(i)}  W^{(i)})}\nonumber\\
&\quad + \norm{ ( \check  A_{i\cdot}- \check P_{i\cdot})(I_n-U^*U^{*T})  \bar U^{(i)}  }  +\sqrt{\frac{k}{\beta n}} \sum_{j\in S} A_{ij}+ \br{t+\sqrt{\frac{k}{\beta n}}}  \sum_{j\in  S_i} \abs{A_{ij}-\E A_{ij}}. \label{eqnpn:10}
\end{align}
Then
\begin{align}
G_{3,i} &\leq  \indic{ \norm{ ( \tilde A_{i\cdot}- \E A_{i\cdot})(I_n-U^*U^{*T}) (U-\bar U)} \geq \frac{\rho}{4}\sqrt{\frac{\beta n}{k}}(p-q)  } \indic{\mathf}  \nonumber\\
&\quad +  \indic{ \norm{ ( \tilde A_{i\cdot}- \E A_{i\cdot})(I_n-U^*U^{*T}) (\bar U -\bar U^{(i)}  W^{(i)})} \geq \frac{\rho}{4}\sqrt{\frac{\beta n}{k}}(p-q)  } \indic{\mathf}  \nonumber\\
&\quad +  \indic{ \norm{ ( \check A_{i\cdot}- \check P_{i\cdot})(I_n-U^*U^{*T})  \bar U^{(i)} } \geq \frac{\rho}{4}\sqrt{\frac{\beta n}{k}}(p-q)  } \indic{\mathf}  \nonumber\\
&\quad +  \indic{\sqrt{\frac{k}{\beta n}} \sum_{j\in S} A_{ij}  \geq \frac{\rho}{8}\sqrt{\frac{\beta n}{k}}(p-q)  } \indic{\mathf}  \nonumber\\
&\quad +  \indic{ \br{t+\sqrt{\frac{k}{\beta n}}}  \sum_{j\in  S_i} \abs{A_{ij}-\E A_{ij}} \geq \frac{\rho}{8}\sqrt{\frac{\beta n}{k}}(p-q)  } \indic{\mathf}  \nonumber\\
&=: H_{1,i} + H_{2,i} + H_{3,i} + H_{4,i} + H_{5,i}. \label{eqnpn:28}
\end{align}

So far, by (\ref{eqnpn:24}), (\ref{eqnpn:25}), and (\ref{eqnpn:28}), we have
\begin{align}
  &\frac{1}{n}  \E \sum_{i\notin S}\indic{\hat z_i\neq z^*_i}\indic{\mathf} \nonumber\\
   &\leq \frac{1}{n} \E\sum_{i\notin S} \sum_{a\neq z^*_i}   \br{G_{1,i,a} + G_{2,i} + G_{3,i}}\nonumber\\
  &\leq  \frac{1}{n} \sum_{i\in [n]}  \sum_{a\neq z^*_i} \E G_{1,i,a} + \frac{k}{n}\sum_{i\in [n]} \E G_{2,i} +  \frac{k}{n}\sum_{i\in [n]}\E H_{1,i} + \frac{k}{n}\sum_{i\in [n]}\E H_{2,i} \nonumber\\
 & \quad + \frac{k}{n}\sum_{i\in [n]}\br{\E H_{3,i} + \E H_{4,i} + \E H_{5,i}}. \label{eqnpn:26}
  \end{align}
We are going to analyze terms in (\ref{eqnpn:26}) one by one. Consider any $i\in[n]$ and any $a\neq z^*_i$.

For $G_{1,i,a}$, we have
\begin{align*}
\E G_{1,i,a} &\leq \pbr{{\sum_{j:z^*_j=a} ( A_{ij} -\E A_{ij}) - \sum_{j:z^*_j=z^*_i,j\neq i} ( A_{ij} -\E A_{ij}) }  \geq \frac{1}{2}\br{1-4\rho-C_3\sqrt{\frac{\beta^{-2}k^3p}{n(p-q)^2}}} (n_a + n_{z^*_i}-1)(p-q)}\\
&\leq \ebr{-J_{n_a,n_{z^*_i}-1,p,q} + \br{2\rho+\frac{C_3}{2}\sqrt{\frac{\beta^{-2}k^3p}{n(p-q)^2}}}(n_a+n_{z^*_i}-1) \frac{(p-q)^2}{q}}\\
&\leq  \ebr{-J_{n_a,n_{z^*_i},p,q} + \frac{(p-q)^2}{4q}  + \br{2\rho+\frac{C_3}{2}\sqrt{\frac{\beta^{-2}k^3p}{n(p-q)^2}}}(n_a+n_{z^*_i}-1) \frac{(p-q)^2}{q}}\\
&\leq \ebr{ - \br{1- \frac{\br{2\rho+\frac{C_3}{2}\sqrt{\frac{\beta^{-2}k^3p}{n(p-q)^2}}}(n_a+n_{z^*_i}-1)\frac{(p-q)^2}{q}+ \frac{(p-q)^2}{4q} }{n_{z^*_i}  \frac{(p-q)^2}{8p}}}J_{n_a,n_{z^*_i},p,q} }\\
&\leq \ebr{-\br{1- \br{2\rho+\frac{C_3}{2}\sqrt{\frac{\beta^{-2}k^3p}{n(p-q)^2}} + \frac{1}{n}} \frac{16kp}{\beta q} }J_{n_a,n_{z^*_i},p,q} }
\end{align*}
where the first inequity  is due to  Lemma \ref{lem:binomial_diff} and the second and third inequalities are due to Lemma \ref{lem:J}.

For $G_{2,i}$, we have  $\normt{ ( A_{i\cdot}- \E A_{i\cdot})U^*} \leq \sum_{j\in[k]}  |( A_{i\cdot}- \E A_{i\cdot})U^*_{\cdot j}|$. Note that $U^*_{\cdot j}$ is a unit vector with $\normt{U^*_{\cdot j}}_\infty \leq \sqrt{k/(\beta n)}$ for each $j\in[k]$ according to Lemma \ref{lem:U_star}. By the union bound and Bernstein inequality, we  have
\begin{align*}
\E G_{2,i} &\leq \sum_{j\in[k]} \pbr{\abs{ ( A_{i\cdot}- \E A_{i\cdot})U^*_{\cdot j}}  \geq \frac{\rho n(p-q)^2}{C_2 \beta^{-1.5} k^3\sqrt{p}}}\leq 2\sum_{j\in[k]} \ebr{- \frac{\frac{1}{2}\br{\frac{\rho n(p-q)^2}{C_2 \beta^{-1.5} k^3\sqrt{p}}}^2}{p\norm{U^*_{j\cdot}}^2 + \frac{1}{3}\norm{U^*_{j\cdot}}_\infty \frac{\rho n(p-q)^2}{C_2 \beta^{-1.5} k^3\sqrt{p}} }}\\
&\leq 2k\ebr{- \frac{\frac{1}{2}\br{\frac{\rho n(p-q)^2}{C_2 \beta^{-1.5} k^3\sqrt{p}}}^2}{p + \frac{1}{3} \sqrt{\frac{k}{\beta n}} \frac{\rho n(p-q)^2}{C_2 \beta^{-1.5} k^3\sqrt{p}} }}
\leq 2k \ebr{-\br{\frac{1}{4} \frac{\rho^2 n^2(p-q)^4}{C_2^2 \beta^{-3} k^6 p^2}}\wedge \br{\frac{3}{4}\frac{\sqrt{np}\rho n(p-q)^2}{C_2 \beta^{-2} k^{3.5} p}}}\\
&\leq k\ebr{-\frac{10n(p-q)^2}{p}},
\end{align*}
where the last inequality holds as long as $\frac{\rho^2n(p-q)^2}{\beta^{-3} k^6 p}\rightarrow\infty$ and $\frac{\rho\sqrt{np}}{\beta^{-2}k^{3.5}}\rightarrow\infty$.

In the remaining part of the proof, we set $t=t_0$ whose value is given (\ref{eqn:t_0}). 
We are going to analyze $H_{1,i}, H_{2,i}, H_{3,i}, H_{4,i},$ and $H_{5,i}$ term by term. For $H_{1,i}$, we have
\begin{align}
\sum_{i\in[n]}H_{1,i} &\leq \sum_{i\in[n]} \frac{16k }{\rho^2 \beta n(p-q)^2}  \norm{ ( \tilde A_{i\cdot}- \E A_{i\cdot})(I_n-U^*U^{*T}) (U-\bar U)}^2\indic{\mathf} \nonumber\\
&= \frac{16k }{\rho^2 \beta n(p-q)^2} \br{ \sum_{i\in[n]} \norm{ ( \tilde A_{i\cdot}- \E A_{i\cdot})(I_n-U^*U^{*T}) (U-\bar U)}^2}\indic{\mathf} \nonumber\\
&= \frac{16k }{\rho^2 \beta n(p-q)^2}\fnorm{(\tilde A - \E A)(I_n-U^*U^{*T}) (U-\bar U)}^2\indic{\mathf}\nonumber\\
&\leq  \frac{16k }{\rho^2 \beta n(p-q)^2} \norm{\tilde A - \E A}^2 \fnorm{(I_n-U^*U^{*T}) (U-\bar U)}^2\indic{\mathf} \nonumber\\
&\leq  \frac{16C_0^2kp }{\rho^2 \beta (p-q)^2} \fnorm{U-\bar U}^2\indic{\mathf}, \label{eqnpn:12}
\end{align}
where in the last equation, we use the fact that $\{( \tilde A_{i\cdot}- \E A_{i\cdot})(I_n-U^*U^{*T}) (U-\bar U)\}_{i\in[n]}$ are rows of $(\tilde A - \E A)(I_n-U^*U^{*T}) (U-\bar U)$, and 
in the last inequality, we use $\normt{\tilde A - \E  A}\leq C_0\sqrt{np}$ under the event $\mathf$. By (\ref{eqnpn:1}) and Theorem \ref{thm:truncated_l2}, we have
\begin{align*}
\E \sum_{i\in[n]}H_{1,i}  &\leq  \frac{16C_0^2kp }{\rho^2 \beta (p-q)^2} \E  \sum_{i\in[n]} \norm{U_{i\cdot}}^2 \indic{\norm{U_{i\cdot}}\indic{\mathf} \geq t_0} \leq  \frac{16C_0^2kp }{\rho^2 \beta (p-q)^2} t_0^2n\ebr{  -\frac{3n(p-q)^2}{2p}}\\
&\leq nk\ebr{  -\frac{3n(p-q)^2}{2p}},
\end{align*}
where the second and the third inequalities hold under the assumption $\beta^{-1},k=O(1),p\leq 1/2$, and $\frac{\rho^2n(p-q)^2}{p}\rightarrow\infty$.

For $H_{2,i}$, using (\ref{eqnpn:2}), we have
\begin{align}
 \norm{ (\tilde A_{i\cdot}- \E A_{i\cdot})(I_n-U^*U^{*T}) (\bar U -\bar U^{(i)}  W^{(i)})} &\leq  \norm{  \tilde A_{i\cdot}- \E A_{i\cdot}}\norm{\bar U -\bar U^{(i)}  W^{(i)}} \nonumber\\
 &\leq  \norm{  \tilde A- \E A} 6k^{1.5}\norm{U_{i\cdot}}.  \label{eqnpn:11}
\end{align}
Then
\begin{align*}
H_{2,i} &\leq   \indic{ \norm{  \tilde A- \E A} 6k^{1.5}\norm{U_{i\cdot}}\geq \frac{\rho}{4}\sqrt{\frac{\beta n}{k}}(p-q)  } \indic{\mathf}\\
&\leq  \indic{6C_0k^{1.5}\sqrt{np}\norm{U_{i\cdot}} \geq \frac{\rho}{4}\sqrt{\frac{\beta n}{k}}(p-q)  } \indic{\mathf}\\
&\leq \indic{\norm{U_{i\cdot}} \geq t_0} \indic{\mathf},
\end{align*}
where the last inequality holds under the assumption that $\beta^{-1},k=O(1),p\leq 1/2$, and $\frac{\rho^2n(p-q)^2}{p}\rightarrow\infty$. Then by Corollary \ref{cor:truncated_hamming}, we have
\begin{align*}
\E \sum_{i\in[n]} H_{2,i}\leq \E \sum_{i\in[n]} \indic{\norm{U_{i\cdot}} \geq t_0} \indic{\mathf}  \leq n\ebr{  -\frac{3n(p-q)^2}{2p}}.
\end{align*}

For $H_{3,i}$, using  (\ref{eqnpn:5}), we have
\begin{align*}
H_{3,i}&\leq  \indic{ \norm{ ( \check A_{i\cdot}- \check P_{i\cdot})(I_n-U^*U^{*T})  \bar U^{(i)} } \geq \frac{\rho}{4}\sqrt{\frac{\beta n}{k}}(p-q)  } \indic{\mathf_i} 
\end{align*}
and then
\begin{align*}
\E H_{3,i} &\leq \E \br{\E  \br{\indic{ \norm{ ( \check A_{i\cdot}- \check P_{i\cdot})(I_n-U^*U^{*T})  \bar U^{(i)} } \geq \frac{\rho}{4}\sqrt{\frac{\beta n}{k}}(p-q)  } \indic{\mathf_i}  \Bigg| S_i \cup\{i\}}}.
\end{align*}
Note that according to (\ref{eqn:independence}) and (\ref{eqnpn:6}), we have independence between $ \check A_{i\cdot}- \check P_{i\cdot}$ and $ \bar U^{(i)}$, $\mathf_i$ when conditioned on $S_i\cup\{i\}$. On the other hand, we have (\ref{eqn:proof_22}) holds under the event $\mathf_i$ and the assumption $\frac{n(p-q)^2}{\beta^{-2}k^2p}\rightarrow\infty$. Together with (\ref{eqnpn:8}), we have 
\begin{align*}
&\E H_{3,i} \\
&\leq \E \br{ \sup_{\Delta\in\mathr^{n\times k}:\fnorm{\Delta}\leq \frac{4(C_0+3)k\sqrt{knp}}{\beta n(p-q)},\norm{\Delta}_{2,\infty}\leq 2t_0}\E  \br{\indic{ \norm{ ( \check A_{i\cdot}- \check P_{i\cdot})\Delta } \geq \frac{\rho}{4}\sqrt{\frac{\beta n}{k}}(p-q)  } \indic{\mathf_i}  \Bigg| S_i \cup\{i\}}}\\
&\leq \E \br{ \sup_{\Delta\in\mathr^{n\times k}:\fnorm{\Delta}\leq \frac{4(C_0+3)k\sqrt{knp}}{\beta n(p-q)},\norm{\Delta}_{2,\infty}\leq 2t_0}\E  \br{\indic{ \norm{ ( \check A_{i\cdot}- \check P_{i\cdot})\Delta } \geq \frac{\rho}{4}\sqrt{\frac{\beta n}{k}}(p-q)  } \Bigg| S_i \cup\{i\}}}\\
&\leq \E \br{ \sup_{\Delta\in\mathr^{n\times k}:\fnorm{\Delta}\leq \frac{4(C_0+3)k\sqrt{knp}}{\beta n(p-q)},\norm{\Delta}_{2,\infty}\leq 2t_0}\E  \sum_{l\in[k]}\sum_{\alpha\in\{-1,1\}}  \br{\indic{ \alpha ( \check A_{i\cdot}- \check P_{i\cdot})\Delta_{\cdot l} \geq \frac{\rho}{8k}\sqrt{\frac{\beta n}{k}}(p-q)  } \Bigg| S_i \cup\{i\}}}\\
&\leq   \sum_{l\in[k]}\sum_{\alpha\in\{-1,1\}} \E \br{ \sup_{\Delta\in\mathr^{n\times k}:\fnorm{\Delta}\leq \frac{4(C_0+3)k\sqrt{knp}}{\beta n(p-q)},\norm{\Delta}_{2,\infty}\leq 2t_0}\E  \br{\indic{ \alpha ( \check A_{i\cdot}- \check P_{i\cdot})\Delta_{\cdot l} \geq \frac{\rho}{8k}\sqrt{\frac{\beta n}{k}}(p-q)  } \Bigg| S_i \cup\{i\}}}.
\end{align*}
Define 
\begin{align}\label{eqnpn:16}
\mathcal{W}:=\{w\in\mathr^n:\norm{w}\leq \frac{4(C_0+3)k\sqrt{knp}}{\beta n(p-q)},\norm{w}_\infty \leq 2t_0\}.
\end{align}
Then,
\begin{align}
\E H_{3,i} &\leq \sum_{l\in[k]}\sum_{\alpha\in\{-1,1\}} \E \br{\sup_{w\in\mathcal{W}} \E  \br{\indic{ \alpha ( \check A_{i\cdot}- \check P_{i\cdot})w\geq \frac{\rho}{8k}\sqrt{\frac{\beta n}{k}}(p-q)  } \Bigg| S_i \cup\{i\}}}\nonumber\\
& = k \sum_{\alpha\in\{-1,1\}} \E \br{\sup_{w\in\mathcal{W}} \E  \br{\indic{ \alpha ( \check A_{i\cdot}- \check P_{i\cdot})w\geq \frac{\rho}{8k}\sqrt{\frac{\beta n}{k}}(p-q)  } \Bigg| S_i \cup\{i\}}} \nonumber\\
&=k\sum_{\alpha\in\{-1,1\}} \E \br{\sup_{w\in \mathcal{W}} \pbr{ \alpha\sum_{j\in S_i^\complement} (A_{ij}-\E A_{ij})w_j\geq \frac{\rho}{8k}\sqrt{\frac{\beta n}{k}}(p-q)   \Bigg| S_i \cup\{i\}}}, \label{eqnpn:15}
\end{align}
where the last equation is due to the definition of $\check A_{i\cdot}$ and $\check P_{i\cdot}$ in (\ref{eqnpn:7}). Now consider any $\alpha\in\{-1,1\}$ and any $w\in\mathcal{W}$. Then by Chernoff bound, the  conditional probability in the above display can be upper bounded by
\begin{align}
&\log \pbr{ \alpha\sum_{j\in S_i^\complement} (A_{ij}-\E A_{ij})w_j\geq \frac{\rho}{8k}\sqrt{\frac{\beta n}{k}}(p-q)   \Bigg| S_i \cup\{i\}} \nonumber\\
&\leq -s\frac{\rho}{8k}\sqrt{\frac{\beta n}{k}}(p-q)   +  \log \E \ebr{ s\alpha\sum_{j\in S_i^\complement} (A_{ij}-\E A_{ij})w_j  \Bigg| S_i \cup\{i\} } \nonumber\\
&=  -s\frac{\rho}{8k}\sqrt{\frac{\beta n}{k}}(p-q)   +  \log \prod_{j\in S_i^\complement} \E \ebr{ s\alpha  (A_{ij}-\E A_{ij})w_j } \nonumber\\
&=  -s\frac{\rho}{8k}\sqrt{\frac{\beta n}{k}}(p-q)   +  \sum_{j\in S_i^\complement}\log\E \ebr{ s\alpha  (A_{ij}-\E A_{ij})w_j } \nonumber\\
&\leq    -s\frac{\rho}{8k}\sqrt{\frac{\beta n}{k}}(p-q)   + ps^2\norm{w}^2 \ebr{{s} \norm{w}_\infty} \nonumber\\
&\leq   -s\frac{\rho}{8k}\sqrt{\frac{\beta n}{k}}(p-q)   + ps^2 \br{ \frac{4(C_0+3)k\sqrt{knp}}{\beta n(p-q)}}^2 \ebr{2{s}t_0}, \label{eqnpn:17}
\end{align}
for any $s>0$,
where the second to last inequality is due to Lemma \ref{lem:weighted_bernoulli_simplified} under the assumption $p\leq 1/2$. Choose $s= \frac{8k}{\rho} \sqrt{\frac{k}{\beta n}} \frac{4n(p-q)}{p}$ so that the first term in the above display is equal to $-\frac{4n(p-q)^2}{p}$. Then, we have
\begin{align*}
&\log \pbr{ \alpha\sum_{j\in S_i^\complement} (A_{ij}-\E A_{ij})w_j\geq \frac{\rho}{8k}\sqrt{\frac{\beta n}{k}}(p-q)   \Bigg| S_i \cup\{i\}} \\
&\leq -\frac{4n(p-q)^2}{p} + p \br{  \frac{8k}{\rho} \sqrt{\frac{k}{\beta n}} \frac{4n(p-q)}{p} \frac{4(C_0+3)k\sqrt{knp}}{\beta n(p-q)}}^2 \ebr{ \frac{8k}{\rho} \sqrt{\frac{k}{\beta n}} \frac{4n(p-q)}{p}2t_0} \\
& = -\frac{4n(p-q)^2}{p} + \frac{128^2(C_0+3)^2 k^6}{\rho^2 \beta^3} \ebr{ \frac{64\times 160 k^4(p-q)}{\rho \beta^2 p}},
\end{align*}
where in the last equation we use the definition of $t_0$ in (\ref{eqn:t_0}). When $\beta^{-1},k=O(1)$ and $\frac{n(p-q)^2}{p}\rightarrow\infty$, a sufficient condition for the second term in the above display to be dominated by the first term is that $\rho$ satisfies 
$\rho^{-1}=o\br{\log\br{\frac{n(p-q)^2}{p}}}$. Then
\begin{align*}
&\log \pbr{ \alpha\sum_{j\in S_i^\complement} (A_{ij}-\E A_{ij})w_j\geq \frac{\rho}{8k}\sqrt{\frac{\beta n}{k}}(p-q)   \Bigg| S_i \cup\{i\}} \leq -\frac{2n(p-q)^2}{p}.
\end{align*}
Hence,
\begin{align*}
\E H_{3,i}&\leq k\sum_{\alpha\in\{-1,1\}} \E \ebr{ -\frac{2n(p-q)^2}{p}}\leq 2k \ebr{ -\frac{2n(p-q)^2}{p}}.
\end{align*}

For $H_{4,i}$, we first have $\sum_{j\in S}A_{ij}\leq \sum_{j\in S_i} A_{ij}$ due to (\ref{eqnpn:3}) and the fact $A_{ii}=0$. Note that conditioned on $S_i$,  $\sum_{j\in S_i} A_{ij}\sim \Binom(\abs{S_i},p)$. This leads to $\E \sum_{j\in S_i} A_{ij} = p \E \abs{S_i}\leq np\ebr{-128np}$ where the last inequality is by Lemma \ref{lem:S_size}. Then
\begin{align*}
\E H_{4,i} & \leq \E \indic{\sum_{j\in S_i} A_{ij} \geq \frac{\rho \beta n(p-q)}{8k}} \leq 
\frac{8k}{\rho \beta n(p-q)} \E  \sum_{j\in S_i} A_{ij} \leq   \frac{8kp}{\rho \beta (p-q)}\ebr{-128np}\\
&=8\sqrt{\frac{p}{\rho^2 \beta^2 n(p-q)^2}} k\sqrt{np}\ebr{-128np} \leq  k\ebr{-64 np},
\end{align*}
where the last inequality holds under the assumption that $\frac{\rho^2n(p-q)^2}{\beta^{-2}p}\rightarrow\infty$.
Similarly, for $H_{5,i}$, we have
\begin{align*}
\E H_{5,i} & \leq \E  \indic{2t_0  \sum_{j\in  S_i} \abs{A_{ij}-\E A_{ij}} \geq \frac{\rho}{8}\sqrt{\frac{\beta n}{k}}(p-q)  } \leq  \frac{2t_0}{\frac{\rho}{8}\sqrt{\frac{\beta n}{k}}(p-q)}\E  \sum_{j\in  S_i} \abs{A_{ij}-\E A_{ij}}\\
&\leq \frac{2t_0}{\frac{\rho}{8}\sqrt{\frac{\beta n}{k}}(p-q)} 2p\E \abs{S_i}\leq \frac{2t_0}{\frac{\rho}{8}\sqrt{\frac{\beta n}{k}}(p-q)} 2np\ebr{-128np}\leq  k\ebr{-64 np},
\end{align*}
where the last inequality holds under the assumption that $\beta^{-1},k=O(1)$,  and $\frac{\rho^2n(p-q)^2}{p}\rightarrow\infty$. 

Now we can combine the above results together to derive the final conclusion. From (\ref{eqnpn:26}), we have
\begin{align*}
 & \frac{1}{n}  \E \sum_{i\notin S}\indic{\hat z_i\neq z^*_i}\indic{\mathf}  \\
  &\leq  \frac{1}{n} \sum_{i\in S}  \sum_{a\neq z^*_i} \ebr{-\br{1- \br{2\rho+\frac{C_3}{2}\sqrt{\frac{\beta^{-2}k^3p}{n(p-q)^2}} + \frac{1}{n}} \frac{16kp}{\beta q} }J_{n_a,n_{z^*_i},p,q} } \\
  &\quad + \frac{k}{n}\sum_{i\in[n]} k\ebr{-\frac{10n(p-q)^2}{p}} +  \frac{k}{n}n\ebr{  -\frac{3n(p-q)^2}{2p}}\\
  &\quad + \frac{k}{n}\sum_{i\in[n]} \br{2k \ebr{ -\frac{2n(p-q)^2}{p}} + 2k\ebr{-64 np}}\\
  &\leq k\ebr{-\br{1- \br{2\rho+\frac{C_3}{2}\sqrt{\frac{\beta^{-2}k^3p}{n(p-q)^2}} + \frac{1}{n}} \frac{16kp}{\beta q} }\min_{1\leq a\neq b\leq k}J_{n_a,n_{b},p,q} }\\
  &\quad + 4k^2\ebr{  -\frac{3n(p-q)^2}{2p}}.
\end{align*}
So far, we require $\beta^{-1},k=O(1)$, $0<q<p\leq 1/2$,  $\frac{n(p-q)^2}{p}\rightarrow\infty$, and $\rho^{-1}=o\br{\log\br{\frac{n(p-q)^2}{p}}}$.
Then using (\ref{eqnpn:27}), we have
\begin{align}
&\E \ell(\hat z,z^*) \nonumber\\
 &\leq k\ebr{-\br{1- \br{2\rho+\frac{C_3}{2}\sqrt{\frac{\beta^{-2}k^3p}{n(p-q)^2}} + \frac{1}{n}} \frac{16kp}{\beta q} }\min_{1\leq a\neq b\leq k}J_{n_a,n_{b},p,q} } \nonumber\\
  &\quad + 4k^2\ebr{  -\frac{3n(p-q)^2}{2p}} +   \ebr{-128np} + 2n^{-3}. \label{eqnpn:29}
\end{align}
When $p\leq 1/10$ is further assumed, by Lemma \ref{lem:J}, we have $J_{\min}=\min_{1\leq a\neq b\leq k}J_{n_a,n_{b},p,q} \leq \max_{a\in[k]} n_a 4(p-q)^2/(3p)\leq 4n(p-q)^2/(3p)$. As a result, the exponents $128np $ and $\frac{3n(p-q)^2}{2p} $  in (\ref{eqnpn:29}) are bigger than $J_{\min}$. We can let $\rho^{-1}=\br{\log\br{\frac{n(p-q)^2}{p}}}^\frac{1}{2}$.
With $\frac{p}{q}=O(1)$ further assumed, 
there exists some constant $C_4$ such that
\begin{align*}
&\E \ell(\hat z,z^*)\indic{\mathf} \leq \ebr{-\br{1-C_4 \br{\log\br{\frac{n(p-q)^2}{p}}}^{-\frac{1}{2}}}J_{\min} }+ 2n^{-3}.
\end{align*}

\subsection{Proof of Theorem \ref{thm:truncated_l2}}\label{sec:proof_truncated}

We first have a decomposition of rows of $U$. Consider any $i\in[n]$.
Due to $U\Lambda = \tilde AU$, we have  $U = \tilde AU\Lambda^{-1}$ and its $i$th row satisfies $U_{i\cdot} = \tilde A_{i\cdot}U\Lambda^{-1}$. Then
we have
\begin{align*}
U_{i\cdot} &=  \tilde A_{i\cdot}U\Lambda^{-1} = (\E A_{i\cdot})U\Lambda^{-1}  + (\tilde A_{i\cdot}-\E A_{i\cdot})U\Lambda^{-1}.
\end{align*}
This gives us
\begin{align}\label{eqn:proof_1}
\norm{U_{i\cdot}} \indic{\mathf} &\leq  \norm{(\E A_{i\cdot})U\Lambda^{-1}} \indic{\mathf} + \norm{ (\tilde A_{i\cdot}-\E A_{i\cdot})U}\norm{\Lambda^{-1}} \indic{\mathf}.
\end{align}
For the  term $ \normt{(\E A_{i\cdot})U\Lambda^{-1}} $ in (\ref{eqn:proof_1}), we have
\begin{align*}
(\E A_{i\cdot})U\Lambda^{-1} & = P_{i\cdot}U\Lambda^{-1} - p U_{i\cdot}\Lambda^{-1} \\
&= U^*_{i\cdot} \Lambda^* U^{*T} U\Lambda^{-1} - p U_{i\cdot}\Lambda^{-1}\\
& =U^*_{i\cdot} U^{*T} U + U^*_{i\cdot} \br{\Lambda^* U^{*T} U - U^{*T} U\Lambda}\Lambda^{-1}-  p U_{i\cdot}\Lambda^{-1}.
\end{align*}
Hence,
\begin{align*}
\norm{(\E A_{i\cdot})U\Lambda^{-1}}&\leq  \norm{ U^*_{i\cdot} }  +  \norm{ U^*_{i\cdot} } \norm{\Lambda^* U^{*T} U - U^{*T} U\Lambda} \norm{\Lambda^{-1}}  +  p\norm{ U_{i\cdot}}\norm{\Lambda^{-1}}.
\end{align*}
Note that $\norm{ U_{i\cdot}}\leq \norm{U}=1$. 
From Lemma \ref{lem:U_star}, we have $\norm{ U^*_{i\cdot} }  \leq 1/\sqrt{{\beta n/k}}$. 
Note that  $ \Lambda^* U^{*T} = U^{*T}  P$ and $U \Lambda  = \tilde AU^T$. We have $\Lambda^* U^{*T} U - U^{*T} U \Lambda =U^{*T}  PU - U^{*T}\tilde AU^T=U^{*T}  (P-\tilde A)U$ and thus $\normt{\Lambda^* U^{*T} U - U^{*T} U\Lambda}\leq \normt{\tilde A - P}$. 
By Lemma \ref{lem:Lambda_star}, we have
 $\normt{\Lambda^{-1}}\leq 2k/(\beta n(p-q))$. As a result,
\begin{align*}
\norm{(\E A_{i\cdot})U\Lambda^{-1}}\indic{\mathf}&\leq \sqrt{\frac{k}{\beta n}} + \sqrt{\frac{k}{\beta n}}C_0\sqrt{np} \frac{2k}{\beta n(p-q)} + p  \frac{2k}{\beta n(p-q)}\leq 2\sqrt{\frac{k}{\beta n}},
\end{align*}
where the last inequality holds under the assumption  that $\frac{n(p-q)^2}{\beta ^{-2}k^2p}\rightarrow\infty$.
Then (\ref{eqn:proof_1}) leads to
\begin{align}\label{eqn:proof_2}
\norm{U_{i\cdot}}\indic{\mathf} \leq  2\sqrt{\frac{k}{\beta n}} + \frac{4k}{\beta n(p-q)} \norm{ (\tilde A_{i\cdot}-\E A_{i\cdot})U} \indic{\mathf}.
\end{align}

From (\ref{eqn:UH_1}), we have
\begin{align}\label{eqn:proof_14}
\sum_{i\in[n]}\norm{U_{i\cdot}}^2\indic{\norm{U_{i\cdot}} \geq t} = \sum_{i\notin S}\norm{U_{i\cdot}}^2\indic{\norm{U_{i\cdot}} \geq t}.
\end{align}
Hence, we only need to consider indexes not in $S$.  From now on, consider any $i\notin S$. Then together with (\ref{eqn:UH_3}), (\ref{eqn:proof_2}) leads to
\begin{align}\label{eqnpn:9}
\norm{U_{i\cdot}}\indic{\mathf} \leq  2\sqrt{\frac{k}{\beta n}} + \frac{4k}{\beta n(p-q)} \norm{ ( A_{i\cdot}-\E A_{i\cdot})U} \indic{\mathf}.
\end{align}
We are going to decompose $ ( A_{i\cdot}-\E A_{i\cdot})U$ in a similar way as in the proof of Theorem \ref{thm:upper}. 
Recall that $\bar U$ and $\bar U^{(i)}$ are defined in (\ref{eqnpn:1}) and (\ref{eqn:proof_26}) by applying $f_t(\cdot)$ to rows of $U$ and $U^{(i)}$ where $t>0$.
For any $t>0$,
we have
\begin{align*}
& \norm{ ( A_{i\cdot}-\E A_{i\cdot})U}  \\
 &\leq  \norm{ ( A_{i\cdot}-\E A_{i\cdot})U^*U^{*T} U} +  \norm{ ( A_{i\cdot}-\E A_{i\cdot})(I-U^*U^{*T})U}\\
 &\leq \norm{ ( A_{i\cdot}-\E A_{i\cdot})U^*} +  \norm{ ( A_{i\cdot}-\E A_{i\cdot})(I-U^*U^{*T})U}\\
 &\leq \norm{ ( A_{i\cdot}-\E A_{i\cdot})U^*} + \norm{ ( \tilde A_{i\cdot}- \E A_{i\cdot})(I_n-U^*U^{*T}) (U -\bar U) }  +   \norm{ (\tilde A_{i\cdot}- \E A_{i\cdot})(I_n-U^*U^{*T}) (\bar U -\bar U^{(i)}  W^{(i)})}\\
&\quad + \norm{ ( \check  A_{i\cdot}- \check P_{i\cdot})(I_n-U^*U^{*T})  \bar U^{(i)}  }  +\sqrt{\frac{k}{\beta n}} \sum_{j\in S} A_{ij}+\br{t+\sqrt{\frac{k}{\beta n}}}  \sum_{j\in  S_i} \abs{A_{ij}-\E A_{ij}},
\end{align*}
where the last inequality is by (\ref{eqnpn:10}).
From (\ref{eqnpn:11}), we have $ \normt{ (\tilde A_{i\cdot}- \E A_{i\cdot})(I_n-U^*U^{*T}) (\bar U -\bar U^{(i)}  W^{(i)})}\leq 6C_0k^{1.5}\sqrt{np}\norm{U_{i\cdot}}$ under the event $\mathf$. Then (\ref{eqnpn:9}) gives
\begin{align*}
\norm{U_{i\cdot}}\indic{\mathf} &\leq \Bigg( 2\sqrt{\frac{k}{\beta n}} + \frac{4k}{\beta n(p-q)}   \norm{ ( A_{i\cdot}-\E A_{i\cdot})U^*}  \\
&\quad +  \frac{4k}{\beta n(p-q)}   \norm{ ( \tilde A_{i\cdot}- \E A_{i\cdot})(I_n-U^*U^{*T}) (U -\bar U) }  \\
&\quad + \frac{24C_0k^{2.5}\sqrt{np}}{\beta n(p-q)}\norm{U_{i\cdot}}  + \frac{4k}{\beta n(p-q)}  \norm{ ( \check  A_{i\cdot}- \check P_{i\cdot})(I_n-U^*U^{*T})  \bar U^{(i)}  }   + \frac{4k}{\beta n(p-q)} \sqrt{\frac{k}{\beta n}} \sum_{j\in S} A_{ij}   \\
&\quad  +\frac{4k}{\beta n(p-q)} \br{t+\sqrt{\frac{k}{\beta n}}}   \sum_{j\in  S_i} \abs{A_{ij}-\E A_{ij}} \Bigg)\indic{\mathf} .
\end{align*}
Under the assumption that $\frac{n(p-q^2)}{\beta^{-2}k^5p}\rightarrow\infty$, we have $\br{1- \frac{24C_0k^{2.5}\sqrt{p}}{\beta \sqrt{n}(p-q)}} >\frac{1}{2}$. After arrangement, we have 
\begin{align*}
\norm{U_{i\cdot}}\indic{\mathf} &\leq  \Bigg(4\sqrt{\frac{k}{\beta n}} + \frac{8k}{\beta n(p-q)}   \norm{ ( A_{i\cdot}-\E A_{i\cdot})U^*}  \\
&\quad +  \frac{8k}{\beta n(p-q)}   \norm{ ( \tilde A_{i\cdot}- \E A_{i\cdot})(I_n-U^*U^{*T}) (U -\bar U) }  \\
&\quad +  \frac{8k}{\beta n(p-q)}  \norm{ ( \check  A_{i\cdot}- \check P_{i\cdot})(I_n-U^*U^{*T})  \bar U^{(i)}  }  \\
&\quad + \frac{8k}{\beta n(p-q)} \sqrt{\frac{k}{\beta n}} \sum_{j\in S} A_{ij}  +\frac{8k}{\beta n(p-q)} \br{t+\sqrt{\frac{k}{\beta n}}}    \sum_{j\in  S_i} \abs{A_{ij}-\E A_{ij}} \Bigg)\indic{\mathf} .
\end{align*}
By Lemma \ref{lem:indicator}, we have
\begin{align*}
&\norm{U_{i\cdot}}^2\indic{\norm{U_{i\cdot}}\geq t}\indic{\mathf} \\
&\leq 25\Bigg( \br{4\sqrt{\frac{k}{\beta n}} }^2\indic{4\sqrt{\frac{k}{\beta n}}  \geq \frac{t}{5}} + \br{\frac{8k}{\beta n(p-q)}  \norm{ ( A_{i\cdot}-\E A_{i\cdot})U^*}}^2 \indic{ \frac{8k}{\beta n(p-q)}   \norm{ ( A_{i\cdot}-\E A_{i\cdot})U^*} \geq \frac{t}{5}}  \\
&\quad + \br{ \frac{8k \norm{ ( \tilde A_{i\cdot}- \E A_{i\cdot})(I_n-U^*U^{*T}) (U -\bar U) }}{\beta n(p-q)}   }^2\indic{ \frac{8k \norm{ ( \tilde A_{i\cdot}- \E A_{i\cdot})(I_n-U^*U^{*T}) (U -\bar U) }}{\beta n(p-q)}   \geq \frac{t}{5}} \\
&\quad + \br{  \frac{8k}{\beta n(p-q)}  \norm{ ( \check  A_{i\cdot}- \check P_{i\cdot})(I_n-U^*U^{*T})  \bar U^{(i)}  } }^2 \indic{  \frac{8k}{\beta n(p-q)}  \norm{ ( \check  A_{i\cdot}- \check P_{i\cdot})(I_n-U^*U^{*T})  \bar U^{(i)}  }  \geq \frac{t}{5}}\\
&\quad + \br{ \frac{8k}{\beta n(p-q)} \sqrt{\frac{k}{\beta n}} \sum_{j\in S} A_{ij}}^2\indic{ \frac{8k}{\beta n(p-q)} \sqrt{\frac{k}{\beta n}} \sum_{j\in S} A_{ij} \geq \frac{t}{5}} \\
&\quad + \br{\frac{8k}{\beta n(p-q)} \br{t+\sqrt{\frac{k}{\beta n}}}   \sum_{j\in  S_i} \abs{A_{ij}-\E A_{ij}}}^2\indic{\frac{8k}{\beta n(p-q)} \br{t+\sqrt{\frac{k}{\beta n}}}   \sum_{j\in  S_i} \abs{A_{ij}-\E A_{ij}}\geq \frac{t}{5}}   \Bigg)\indic{\mathf}.
\end{align*}
For any $t>20\sqrt{\frac{k}{\beta n}}$, we have $\br{4\sqrt{\frac{k}{\beta n}} }^2\indic{4\sqrt{\frac{k}{\beta n}}  \geq \frac{t}{5}}=0$. With this and the fact that an indication function is always smaller or equal to 1,  the above display can be simplified into
\begin{align*}
&\norm{U_{i\cdot}}^2\indic{\norm{U_{i\cdot}}\geq t}\indic{\mathf} \\
&\leq 25 \br{\frac{8k}{\beta n(p-q)}  \norm{ ( A_{i\cdot}-\E A_{i\cdot})U^*}}^2 \indic{ \frac{8k}{\beta n(p-q)}   \norm{ ( A_{i\cdot}-\E A_{i\cdot})U^*} \geq \frac{t}{5}}  \\
&\quad + 25\br{\frac{8k}{\beta n(p-q)}}^2\norm{ ( \tilde A_{i\cdot}- \E A_{i\cdot})(I_n-U^*U^{*T}) (U -\bar U) }^2\indic{\mathf}\\
&\quad + 25\br{  \frac{8k}{\beta n(p-q)}  \norm{ ( \check  A_{i\cdot}- \check P_{i\cdot})(I_n-U^*U^{*T})  \bar U^{(i)}  } }^2 \indic{  \frac{8k}{\beta n(p-q)}  \norm{ ( \check  A_{i\cdot}- \check P_{i\cdot})(I_n-U^*U^{*T})  \bar U^{(i)}  }  \geq \frac{t}{5}}\indic{\mathf}\\
&\quad + 25\br{\frac{8k}{\beta n(p-q)}}^2 \frac{k}{\beta n } \br{ \sum_{j\in S} A_{ij}}^2 +  25\br{\frac{16k}{\beta n(p-q)}}^2t^2\br{  \sum_{j\in  S_i} \abs{A_{ij}-\E A_{ij}}}^2\\
&=: H_{0,i}' + H_{1,i}' + H_{3,i}' + H_{4,i}' + H_{5,i}'.
\end{align*}
Note the similarity between $H_{1,i}'$, $H_{3,i}' $, $H_{4,i}' $, $H_{5,i}' $ and $H_{1,i}$, $H_{3,i} $, $H_{4,i} $, $H_{5,i} $ defined in the proof of Theorem \ref{thm:upper}. Summing over all $i\notin S$, we have
\begin{align}
 \sum_{i \notin S} \norm{U_{i\cdot}}^2 \indic{\norm{U_{i\cdot}} \geq t} &\leq \sum_{i\in[n]}H_{0,i}' + \sum_{i\in[n]}H_{1,i}' +\sum_{i\in[n]}H_{3,i}' +\sum_{i\in[n]}H_{4,i}' +\sum_{i\in[n]}H_{5,i}'. \label{eqnpn:13}
\end{align}

For $\sum_{i\in[n]}H_{1,i}'$, following the analysis to establish (\ref{eqnpn:12}) for $\sum_{i\in[n]}H_{1,i}$ in the proof of Theorem \ref{thm:upper}, we have
\begin{align*}
\sum_{i\in[n]} H_{1,i}'&\leq 25\br{\frac{8k}{\beta n(p-q)}}^2 \sum_{i\in[n]}\norm{ ( \tilde A_{i\cdot}- \E A_{i\cdot})(I_n-U^*U^{*T}) (U -\bar U) }^2\indic{\mathf}\\
&\leq  25\br{\frac{8k}{\beta n(p-q)}}^2  \br{C_0\sqrt{np}}^2 \fnorm{U - \bar U}^2 \indic{\mathf}\\
&\leq \frac{40^2 C_0^2 k^2 p}{\beta^2n(p-q)^2} \fnorm{U - \bar U}^2 \indic{\mathf}.
\end{align*}
Using (\ref{eqnpn:1}) and (\ref{eqn:UH_1}), we have
\begin{align*}
\sum_{i\in[n]} H_{1,i}'&\leq  \frac{40^2 C_0^2 k^2 p}{\beta^2n(p-q)^2} \sum_{i\in[n]} \norm{U_{i\cdot}}^2 \indic{\norm{U_{i\cdot}} \geq t} =   \frac{40^2 C_0^2 k^2 p}{\beta^2n(p-q)^2} \sum_{i \notin S} \norm{U_{i\cdot}}^2 \indic{\norm{U_{i\cdot}} \geq t}.
\end{align*}
Under the assumption $\frac{n(p-q)^2}{\beta^{-2} k^2p}\rightarrow\infty$, we have $\br{1-  \frac{40^2 C_0^2 k^2 p}{\beta^2n(p-q)^2}  }\geq \frac{1}{2}$. Rearranging  (\ref{eqnpn:13}), we have
\begin{align*}
 \sum_{i \notin S} \norm{U_{i\cdot}}^2 \indic{\norm{U_{i\cdot}} \geq t} &\leq 2\sum_{i\in[n]}H_{0,i}' + 2\sum_{i\in[n]}H_{3,i}' + 2\sum_{i\in[n]}H_{4,i}' +2\sum_{i\in[n]}H_{5,i}'.
\end{align*}
Using (\ref{eqn:proof_14}), we have
\begin{align}\label{eqnpn:14}
\E  \sum_{i \in [n]} \norm{U_{i\cdot}}^2 \indic{\norm{U_{i\cdot}} \geq t} &\leq 2 \E \sum_{i\in[n]}H_{0,i}' + 2\E \sum_{i\in[n]}H_{3,i}' + 2\E \sum_{i\in[n]}H_{4,i}' +2\E \sum_{i\in[n]}H_{5,i}'.
\end{align}

In the following, we are going to consider any $i\in[n]$ and
analyze $\E H_{0,i}'$, $\E H_{3,i}' $, $\E H_{4,i}' $, and $\E H_{5,i}' $, term by term. We first analyze the latter two as their analysis is more straightforward. 
Similar to the analysis for $H_{4,i}$ in the proof of Theorem \ref{thm:upper}, we have
\begin{align*}
\E \br{ \sum_{j\in S} A_{ij}}^2 &\leq \E \br{ \sum_{j\in S_i} A_{ij}}^2 = \E \br{ \E  \br{\br{ \sum_{j\in S_i} A_{ij}}^2 \Bigg| S_i}} \leq  \E \br{ (\abs{S_i} p)^2 + \abs{S_i} p } \\
& \leq \br{np + 1} \E \abs{S_i}p\leq 2(np)^2\ebr{-128np},
\end{align*}
where the last inequality is by Lemma \ref{lem:S_size}. Then,
\begin{align*}
\E H_{4,i}'&= 25\br{\frac{8k}{\beta n(p-q)}}^2 \frac{k}{\beta n } \E  \br{ \sum_{j\in S} A_{ij}}^2 \leq  2\frac{40^2 k^3p^2}{\beta^3 n(p-q)^2} \ebr{-128np}.
\end{align*}
Similarly, we have
\begin{align*}
\E  H_{5,i}'&\leq 25\br{\frac{16k}{\beta n(p-q)}}^2t^2   \E \br{  \sum_{j\in  S_i} \abs{A_{ij}-\E A_{ij}}}^2 \\
&\leq 25\br{\frac{16k}{\beta n(p-q)}}^2t^2   \E \br{\br{2p\abs{S_i}}^2 + 2p\abs{S_i}} \\
&\leq  6\frac{160k^2t^2 n p^2}{\beta^2 n(p-q)^2} \ebr{-128np}.
\end{align*}

For $\E H_{0,i}'$, note that for any positive random variable $X$ and any $s\geq 0$, we have $\E X^2\indic{X\geq s} =\sum_{j=1}^{\infty}$ $\E X^2 \indic{(j+1)s >X\geq j s} \leq \sum_{j=1}^{\infty} (j+1)^2s^2\E \indic{X\geq js}$. Then
\begin{align*}
\E H_{0,i}' &\leq   25 \E \br{\frac{8k}{\beta n(p-q)}  \norm{ ( A_{i\cdot}-\E A_{i\cdot})U^*}}^2 \indic{ \frac{8k}{\beta n(p-q)}   \norm{ ( A_{i\cdot}-\E A_{i\cdot})U^*} \geq \frac{t}{5}} \\
&\leq t^2  \sum_{j=1}^{\infty} (j+1)^2 \pbr{ \frac{8k}{\beta n(p-q)}   \norm{ ( A_{i\cdot}-\E A_{i\cdot})U^*} \geq \frac{t}{5}j}\\
&\leq  t^2   \sum_{l\in[k]}  \sum_{j=1}^{\infty} (j+1)^2\pbr{  \abs{ ( A_{i\cdot}-\E A_{i\cdot})U^*_{\cdot l}} \geq  \frac{\beta n(p-q)}{8k}\frac{t}{5k}j} \\
&\leq  2t^2   \sum_{l\in[k]}  \sum_{j=1}^{\infty} (j+1)^2 \ebr{- \frac{\frac{1}{2}j^2 \br{ \frac{\beta n(p-q)}{8k}\frac{t}{5k}}^2}{p +  \frac{1}{3} j \frac{\beta n(p-q)}{8k}\frac{t}{5k} \sqrt{\frac{k}{\beta n}}}},
\end{align*}
where in the last inequality, we use Bernstein inequality and the fact that $\norm{U^*_{\cdot l}}=1$ and $\norm{U^*_{\cdot l}}_\infty \leq \sqrt{\frac{k}{\beta n}}$ from Lemma \ref{lem:U_star}. For any $t>\sqrt{\frac{k}{\beta n}}$, we have $\frac{1}{p} \br{ \frac{\beta n(p-q)}{8k}\frac{t}{5k}}^2 \rightarrow\infty $ and also $ \br{ \frac{\beta n(p-q)}{8k}\frac{t}{5k}}  \sqrt{\frac{\beta n}{k}}\rightarrow\infty$, under the assumption that $\frac{n(p-q)^2}{\beta^{-1}k^3p}\rightarrow\infty$. By  Lemma  \ref{lem:summation}, we have
\begin{align*}
\E H_{0,i}' &\leq 2 t^2   \sum_{l\in[k]}  \sum_{j=1}^{\infty} (j+1)^2 \ebr{-  \frac{\frac{1}{2}j^2}{\br{\frac{1}{p}{\br{ \frac{\beta n(p-q)}{8k}\frac{t}{5k}}^2}}^{-1} + \frac{1}{3}{j}\br{\br{ \frac{\beta n(p-q)}{8k}\frac{t}{5k}}  \sqrt{\frac{\beta n}{k}}}^{-1} }} \\
&\leq 16t^2 k   \ebr{-  \frac{\frac{1}{2}}{\br{\frac{1}{p}{\br{ \frac{\beta n(p-q)}{8k}\frac{t}{5k}}^2}}^{-1} + \frac{1}{3}\br{\br{ \frac{\beta n(p-q)}{8k}\frac{t}{5k}}  \sqrt{\frac{\beta n}{k}}}^{-1} }} \\
& = 16t^2 k  \ebr{- \frac{\frac{1}{2} \br{ \frac{\beta n(p-q)}{8k}\frac{t}{5k}}^2}{p +  \frac{1}{3} \frac{\beta n(p-q)}{8k}\frac{t}{5k} \sqrt{\frac{k}{\beta n}}}}.
\end{align*}

The analysis for $\E H_{3,i}'$ is similar to that of $\E H_{0,i}'$ but is more involved. By the same argument as in the analysis of $\E H_{0,i}'$, we have
\begin{align*}
\E H_{3,i}'&\leq 25 \E \br{  \frac{8k}{\beta n(p-q)}  \norm{ ( \check  A_{i\cdot}- \check P_{i\cdot})(I_n-U^*U^{*T})  \bar U^{(i)}  } }^2 \\
&\quad \times \indic{  \frac{8k}{\beta n(p-q)}  \norm{ ( \check  A_{i\cdot}- \check P_{i\cdot})(I_n-U^*U^{*T})  \bar U^{(i)}  }  \geq \frac{t}{5}}\indic{\mathf} \\
&\leq t^2 \sum_{j=1}^{\infty}(j+1)^2 \E  \indic{  \frac{8k}{\beta n(p-q)}  \norm{ ( \check  A_{i\cdot}- \check P_{i\cdot})(I_n-U^*U^{*T})  \bar U^{(i)}  }  \geq \frac{t}{5} j}\indic{\mathf}.
\end{align*}
Consider a fixed $j\geq 1$ and any $t>\sqrt{\frac{k}{\beta n}}$. 
Follow the analysis of $\E H_{3,i}$ in the proof of Theorem \ref{thm:upper}, we have
\begin{align*}
&\E  \indic{  \frac{8k}{\beta n(p-q)}  \norm{ ( \check  A_{i\cdot}- \check P_{i\cdot})(I_n-U^*U^{*T})  \bar U^{(i)}  }  \geq \frac{t}{5} j}\indic{\mathf}\\
& =  \E  \indic{    \norm{ ( \check  A_{i\cdot}- \check P_{i\cdot})(I_n-U^*U^{*T})  \bar U^{(i)}  }  \geq \frac{\beta n(p-q)tj}{40k}}\indic{\mathf}\\
&\leq k\sum_{\alpha\in\{-1,1\}} \E \br{\sup_{w\in \mathcal{W'}} \pbr{ \alpha\sum_{j\in S_i^\complement} (A_{ij}-\E A_{ij})w_j\geq \frac{\beta n(p-q)tj}{80k^2}   \Bigg| S_i \cup\{i\}}}
\end{align*}
which is analogous to (\ref{eqnpn:15}). Here $\mathcal{W'}:=\{w\in\mathr^n:\norm{w}\leq \frac{4(C_0+3)k\sqrt{knp}}{\beta n(p-q)},\norm{w}_\infty \leq 2t\}$, analogous to the definition of  $\mathcal{W}$  in (\ref{eqnpn:16}). Consider any $\alpha \in\{-1,1\}$ and any $w\in\mathcal{W'}$. Analogous to (\ref{eqnpn:17}), we can obtain
\begin{align*}
&\log \pbr{ \alpha\sum_{j\in S_i^\complement} (A_{ij}-\E A_{ij})w_j\geq \frac{\beta n(p-q)tj}{80k^2}   \Bigg| S_i \cup\{i\}} \\
&\leq -\frac{\beta n(p-q)tj}{80k^2}  s +   ps^2 \br{ \frac{4(C_0+3)k\sqrt{knp}}{\beta n(p-q)}}^2 \ebr{2{s}t},
\end{align*}
for any $s>0$. To derive the above display, we use Lemma \ref{lem:weighted_bernoulli_simplified} and assume $p\leq 1/2$.
Choose $s =\frac{160k^2}{\beta t}\frac{(p-q)}{p}$ such that $\frac{\beta n(p-q)t}{80k^2}  s  =  \frac{2n(p-q)^2}{p}$. Then we have
\begin{align*}
&\log \pbr{ \alpha\sum_{j\in S_i^\complement} (A_{ij}-\E A_{ij})w_j\geq \frac{\beta n(p-q)tj}{80k^2}   \Bigg| S_i \cup\{i\}} \\
&\leq - \frac{2n(p-q)^2}{p}j  + p  \br{\frac{160k^2}{\beta t}\frac{(p-q)}{p}}^2\br{ \frac{4(C_0+3)k\sqrt{knp}}{\beta n(p-q)}}^2 \ebr{\frac{320k^2(p-q)}{p}}\\
& =  - \frac{2n(p-q)^2}{p}j  + \br{\frac{640(C_0+3)k^{3.5}}{\beta^2 \sqrt{n}t}}^2\ebr{\frac{320k^2(p-q)}{p}}.
\end{align*}
Then
\begin{align*}
&\E  \indic{  \frac{8k}{\beta n(p-q)}  \norm{ ( \check  A_{i\cdot}- \check P_{i\cdot})(I_n-U^*U^{*T})  \bar U^{(i)}  }  \geq \frac{t}{5} j}\indic{\mathf}\\
&\leq 2k \ebr{ - \frac{2n(p-q)^2}{p}j  + \br{\frac{640(C_0+3)k^{3.5}}{\beta^2 \sqrt{n}t}}^2\ebr{\frac{320k^2(p-q)}{p}}},
\end{align*}
and consequently,
\begin{align*}
\E H_{3,i}' &\leq 2k^2t^2  \sum_{j=1}^{\infty} (j+1)^2\ebr{ - \frac{2n(p-q)^2}{p}j  + \br{\frac{640(C_0+3)k^{3.5}}{\beta^2 \sqrt{n}t}}^2\ebr{\frac{320k^2(p-q)}{p}}}\\
&\leq 16k^2t^2 \ebr{ - \frac{2n(p-q)^2}{p} + \br{\frac{640(C_0+3)k^{3.5}}{\beta^2 \sqrt{n}t}}^2\ebr{\frac{320k^2(p-q)}{p}}},
\end{align*}
where the last inequality is due to that $\frac{n(p-q)^2}{p}\rightarrow\infty$.

Combining all the above expressions together, (\ref{eqnpn:14}) leads to
\begin{align*}
\E  \sum_{i \in [n]} \norm{U_{i\cdot}}^2 \indic{\norm{U_{i\cdot}} \geq t} &\leq 4\frac{40^2 k^3p^2}{\beta^3 n(p-q)^2} n\ebr{-128np} + 12\frac{160k^2t^2 n p^2}{\beta^2 n(p-q)^2} n\ebr{-128np} \\
&\quad +32t^2 k n  \ebr{- \frac{\frac{1}{2} \br{ \frac{\beta n(p-q)}{8k}\frac{t}{5k}}^2}{p +  \frac{1}{3} \frac{\beta n(p-q)}{8k}\frac{t}{5k} \sqrt{\frac{k}{\beta n}}}} \\
&\quad + 32k^2t^2n \ebr{ - \frac{2n(p-q)^2}{p} + \br{\frac{640(C_0+3)k^{3.5}}{\beta^2 \sqrt{n}t}}^2\ebr{\frac{320k^2(p-q)}{p}}},
\end{align*}
for any $t>20\sqrt{\frac{k}{\beta n}}$. Taking $t=t_0$ where $t_0$ is defined in (\ref{eqn:t_0}), we have
\begin{align*}
\E  \sum_{i \in [n]} \norm{U_{i\cdot}}^2 \indic{\norm{U_{i\cdot}} \geq t_0} &\leq 4\frac{40^2 p}{k^2 n(p-q)^2} nt_0^2(np)\ebr{-128np} + 12\frac{160k^2 p}{\beta^2 n(p-q)^2} nt^2_0 (np)\ebr{-128np} \\
&\quad +32t_0^2 kn\ebr{-\frac{2n(p-q)^2}{p}}\\
&\quad + 32k^2t_0^2n \ebr{ - \frac{2n(p-q)^2}{p} +72k^2(C_0+3)^2 \beta^{-1}\ebr{\frac{320k^2(p-q)}{p}}}\\
&\leq t_0^2 n \ebr{-\frac{3n(p-q)^2}{2p}},
\end{align*}
where the last inequality holds under the assumption that $\beta^{-1},k=O(1)$, $0<q<p\leq 1/2$, and $\frac{n(p-q)^2}{p}\rightarrow\infty$.

\appendix

\section{Proof of Theorem \ref{thm:lower}}\label{sec:proof_lower}
The proof of Theorem \ref{thm:lower} follows the proof of Theorem \ref{thm:upper} with some modifications as we need to derive a lower bound instead of an upper bound. Define $(a,b):=\argmin_{1\leq a'\neq b'\leq k} J_{n_{a'},n_{b'},p,q}$. We have
\begin{align*}
 \ell(\hat z,z^*)&\geq \ell(\hat z,z^*)\indic{\mathf} \geq \frac{1}{n}\sum_{i\notin S} \indic{\hat z_i \neq z^*_i}\indic{\mathf} \\
&\geq \frac{1}{n}\sum_{i\notin S:z^*_i=b} \indic{\hat z_i \neq z^*_i}\indic{\mathf}  \geq \frac{1}{n}\sum_{i\notin S:z^*_i=b} \indic{\norm{U_{i\cdot}\Lambda - \hat \theta_a}\leq \norm{U_{i\cdot}\Lambda - \hat \theta_{z^*_i}}}\indic{\mathf},
\end{align*}
where in the second inequality we use (\ref{eqnpn:33}).
Consider any $i\notin S$ such that $z^*_i=b$. Analogous to the establishment of (\ref{eqnp:upper_10}), we can have
\begin{align}
&\indic{\norm{U_{i\cdot}\Lambda - \hat \theta_a}\leq \norm{U_{i\cdot}\Lambda - \hat \theta_{z^*_i}}} \indic{\mathf} \nonumber \\
&\geq \mathbb{I}\Bigg\{(p-q)\br{\sum_{j:z^*_j=a} ( A_{ij} -\E A_{ij}) - \sum_{j:z^*_j=z^*_i,j\neq i} ( A_{ij} -\E A_{ij}) }  - 2C_2 \beta^{-0.5} k\sqrt{p}\norm{ ( A_{i\cdot}- \E A_{i\cdot})U^*}  \nonumber\\
&\quad - 2 \sqrt{n_a + n_{z^*_i}}(p-q)\norm{ ( A_{i\cdot}- \E A_{i\cdot})(I_n-U^*U^{*T}) U}\nonumber\\
& \quad \geq \frac{1}{2}\br{1-C_3\sqrt{\frac{\beta^{-2}k^3p}{n(p-q)^2}}} (n_a + n_{z^*_i})(p-q)^2\Bigg\}\indic{\mathf}. \nonumber
\end{align}
Consider a positive sequence $\rho=o(1)$  whose value will be determined later. Analogous to the establishment of (\ref{eqnpn:25}), we have
\begin{align}
&\indic{\norm{U_{i\cdot}\Lambda - \hat \theta_a}\leq \norm{U_{i\cdot}\Lambda - \hat \theta_{z^*_i}}} \indic{\mathf} \nonumber \\
&\geq \mathbb{I}\Bigg\{{\sum_{j:z^*_j=a} ( A_{ij} -\E A_{ij}) - \sum_{j:z^*_j=z^*_i,j\neq i} ( A_{ij} -\E A_{ij}) }  \geq \frac{1}{2}\br{1+4\rho+C_3\sqrt{\frac{\beta^{-2}k^3p}{n(p-q)^2}}} (n_a + n_{z^*_i})(p-q)\Bigg\} \indic{\mathf}  \nonumber\\
&\quad -  \indic{\norm{ ( A_{i\cdot}- \E A_{i\cdot})U^*}  \geq \frac{\rho n (p-q)^2}{C_2 \beta^{-1.5} k^2\sqrt{p}}}  \nonumber\\
&\quad -  \indic{ \norm{ ( A_{i\cdot}- \E A_{i\cdot})(I_n-U^*U^{*T}) U} \geq \rho\sqrt{\frac{\beta n}{k}}(p-q)  } \indic{\mathf}  \nonumber\\
& =: G_{1,i,a}' \indic{\mathf} - G_{2,i} - G_{3,i}.  \nonumber
\end{align}
Note here $G_{1,i,a}'$ is slightly different from $G_{1,i,a}$ in the proof of Theorem \ref{thm:upper} and $G_{2,i},G_{3,i}$ are  exactly the same quantities as they appear in the proof of Theorem \ref{thm:upper}.  Then the decomposition (\ref{eqnpn:28}) also holds for $G_{3,i}$.  Then analogous to (\ref{eqnpn:26}), we have
\begin{align*}
\E \ell(\hat z,z^*)&\geq  \frac{1}{n} \E\sum_{i\notin S:z^*_i=b} \br{G_{1,i,a}'\indic{\mathf}-  G_{2,i} - G_{3,i}}\\
&\geq   \frac{1}{n} \E\sum_{i\notin S:z^*_i=b}  G_{1,i,a}' \indic{\mathf}-  \frac{1}{n}\sum_{i\in [n]} \E G_{2,i} -  \frac{1}{n}\sum_{i\in [n]}\E H_{1,i} - \frac{1}{n}\sum_{i\in [n]}\E H_{2,i} \nonumber \\
&\quad - \frac{1}{n}\sum_{i\in [n]}\br{\E H_{3,i} + \E H_{4,i} + \E H_{5,i}}.
\end{align*}

In the above display, we only need to analyze $  \E\sum_{i\notin S:z^*_i=b}  G_{1,i,a}' \indic{\mathf}$ as all the remaining terms have been analyzed in the proof of Theorem \ref{thm:upper}. Consider any $i\in[n]$ and any $z^*_i=b$. We introduce an auxiliary random variable $X\sim\Ber(p)$ that is independent of $A$. Since $X\geq 0$, we have
\begin{align*}
&\E G'_{1,i,a} \\
&= \pbr{{\sum_{j:z^*_j=a} ( A_{ij} -\E A_{ij}) - \sum_{j:z^*_j=z^*_i,j\neq i} ( A_{ij} -\E A_{ij}) }  \geq \frac{1}{2}\br{1+4\rho+C_3\sqrt{\frac{\beta^{-2}k^3p}{n(p-q)^2}} } (n_a + n_{z^*_i})(p-q)}\\
&\geq \pbr{{\sum_{j:z^*_j=a} ( A_{ij} -\E A_{ij}) - \sum_{j:z^*_j=z^*_i,j\neq i} ( A_{ij} -\E A_{ij}) }  \geq \frac{1}{2}\br{1+4\rho+C_3\sqrt{\frac{\beta^{-2}k^3p}{n(p-q)^2}} } (n_a + n_{z^*_i})(p-q) + X}\\
&= \p \Bigg({\sum_{j:z^*_j=a} ( A_{ij} -\E A_{ij}) - \sum_{j:z^*_j=z^*_i,j\neq i} ( A_{ij} -\E A_{ij}) } -(X-p) \geq  \\
&\quad\quad\quad\frac{1}{2}\br{1+4\rho+C_3\sqrt{\frac{\beta^{-2}k^3p}{n(p-q)^2}} } (n_a + n_{z^*_i})(p-q) +p\Bigg)\\
&\geq \p \Bigg({\sum_{j:z^*_j=a} ( A_{ij} -\E A_{ij}) - \sum_{j:z^*_j=z^*_i,j\neq i} ( A_{ij} -\E A_{ij}) } -(X-p) \geq  \\
&\quad\quad\quad\frac{1}{2}\br{1+4\rho+C_3\sqrt{\frac{\beta^{-2}k^3p}{n(p-q)^2}}  + \frac{pk}{\beta n(p-q)}} (n_a + n_{z^*_i})(p-q) \Bigg)\\
&\geq \ebr{-J_{n_a,n_b,p,q} - \br{4\rho+C_3\sqrt{\frac{\beta^{-2}k^3p}{n(p-q)^2}}  + \frac{pk}{\beta n(p-q)}}  (n_a + n_{z^*_i}) \frac{(p-q)^2}{p} - 4\sqrt{\frac{ (n_a + n_{z^*_i})(p-q)^2p}{q^2}}}\\
&\quad\quad\quad \times \br{\frac{1}{4} -  \sqrt{\frac{2}{(n_a+n_{z^*_i })q}}}.
\end{align*}
where we use Lemma \ref{lem:binomial_diff} in the last inequality. Under the assumption $\frac{n(p-q)^2}{\beta^{-1}kp}\rightarrow \infty$ and $p/q=O(1)$, using Lemma \ref{lem:J}, we have
\begin{align*}
&\E G'_{1,i,a} \\
&\geq  \frac{1}{8}\ebr{- \br{1+  \br{4\rho + C_3\sqrt{\frac{\beta^{-2}k^3p}{n(p-q)^2}} + \frac{pk}{\beta n(p-q)}} \frac{16kp}{\beta q} + 64\frac{kp}{\beta q} \sqrt{\frac{p}{n(p-q)^2}}} J_{n_a,n_b,p,q}}.
\end{align*}
Note that $J_{n_a,n_b,q,q} =\min_{1\leq a'\neq b'\leq k} J_{n_{a'},n_{b'},p,q} = J_{\min}$. We have
\begin{align*}
&  \E\sum_{i\notin S:z^*_i=b}  G_{1,i,a}' \indic{\mathf} \\
  &\geq   \E\sum_{i\in [n]:z^*_i=b}  G_{1,i,a}'  - \abs{S}-n\pbr{\mathf^\complement}  \\
  &\geq \frac{\beta n}{8k}\ebr{- \br{1+  \br{4\rho + C_3\sqrt{\frac{\beta^{-2}k^3p}{n(p-q)^2}} + \frac{pk}{\beta n(p-q)}} \frac{16kp}{\beta q} + 64\frac{kp}{\beta q} \sqrt{\frac{p}{n(p-q)^2}}} J_{\min}} \\
  &\quad - n\ebr{-128np} - 2n^{-2},
\end{align*}
where we use Lemma \ref{lem:S_size}. Analogous to the establishment of (\ref{eqnpn:29}), we have
\begin{align*}
\E \ell(\hat z,z^*)&\geq \frac{\beta }{8k}\ebr{- \br{1+  \br{4\rho + C_3\sqrt{\frac{\beta^{-2}k^3p}{n(p-q)^2}} + \frac{pk}{\beta n(p-q)}} \frac{16kp}{\beta q} + 64\frac{kp}{\beta q} \sqrt{\frac{p}{n(p-q)^2}}}J_{\min}} \\
&\quad - \ebr{-128np} - 2n^{-3} - 4k^2\ebr{  -\frac{3n(p-q)^2}{2p}},
\end{align*}
under the assumption $\beta^{-1},k=O(1), 1\leq q<p\leq 1/2, p/q=O(1), \frac{n(p-q)^2}{p}\rightarrow\infty,$ and $\rho^{-1}=o\br{\log\br{\frac{n(p-q)^2}{p}}}$.
We can take $\rho^{-1}= \br{\log\br{\frac{n(p-q)^2}{p}}}^\frac{1}{4}$. When $p\leq 1/10$ is additionally assume, by following same argument as in the proof of Theorem \ref{thm:upper}, there exists some constant $C_5$ such that
\begin{align*}
&\E \ell(\hat z,z^*)\geq  \ebr{-\br{1+C_5 \br{\log\br{\frac{n(p-q)^2}{p}}}^{-\frac{1}{4}}} J_{\min} }- 2n^{-3}.
\end{align*}

\section{Proofs of Results in Section \ref{sec:pre}}\label{sec:proof_pre}

First we give an equivalent expression for $P$. Define $Z^*\in\{0,1\}^{n\times k}$ to be a matrix such that $Z^*_{ij}=\indic{z^*_i=j}$ for all $i\in[n]$ and $j\in[k]$. In addition, define $B\in\mathr^{k\times k}$ such that $B_{ab}=q\indic{a\neq b} +p\indic{a=b}$ for all $a,b\in[k]$. Then we can verify
\begin{align*}
P = Z^* B Z^{*T}.
\end{align*}
Then following lemmas are about properties of population quantities.

\begin{lemma}\label{lem:U_star}
Define $\Delta:=\text{diag}(\sqrt{n_1},\ldots, \sqrt{n_k})$. There exists some $W\in \matho(k,k)$ such that $U^* = Z^* \Delta^{-1}W$ and $\Lambda^* = W^T \Delta B \Delta W$.
In addition,
\begin{align*}
\norm{U^*_{i\cdot}} = \frac{1}{\sqrt{n_{z^*_i}}},\forall i\in[n].
\end{align*}
Consequently, $$\max_{i\in[n],j\in[k]}\normt{U^*_{ij}} \leq \norm{U^*}_{2,\infty}= \sqrt{k /(\beta n)}.$$
\end{lemma}

\begin{proof}%
Note that $P= Z^* BZ^{*T} = Z^* \Delta^{-1}(\Delta B \Delta )\Delta^{-1} Z^{*T} $ and $Z^* \Delta^{-1} \in\matho(k,k)$.
Apply SVD to the matrix  $\Delta B \Delta$ and we obtain   $\Delta B \Delta= W \check\Lambda W^T$ for some $W\in\matho(k,k)$ and some diagonal matrix $\check \Lambda$. Since $Z^* \Delta^{-1}W\in\matho(n,k)$ and $P = (Z^* \Delta^{-1} W)\check \Lambda  (Z^* \Delta^{-1}W)^T$, we  have $\Lambda^*=\check \Lambda = W^T  \Delta B \Delta W$ and $U=Z^* \Delta^{-1} W$. As a result, for each $i\in[n]$, we have
\begin{align*}
\norm{U_{i\cdot}} = \norm{Z^*_{i\cdot} \Delta^{-1} W} = \frac{1}{\sqrt{n_{z^*_i}}} \norm{W_{z^*_i\cdot}} =  \frac{1}{\sqrt{n_{z^*_i}}} .
\end{align*}
\end{proof}

\begin{lemma}\label{lem:P_line_diff}
For any $a,b\in[k]$ such that $a\neq b$, we have $U^*(\theta_a^*  - \theta_b^*)^T\in\mathr^n$ satisfies
\begin{align}\label{eqn:Ustat_thetastar_j}
(U^*(\theta_a^*  - \theta_b^*)^T)_j =\begin{cases}
(p-q),\text{ if }z^*_j=a,\\
-(p-q),\text{ if }z^*_j=b,\\
0,\text{ o.w.},
\end{cases}
\end{align}
for all $j\in[n]$.
\end{lemma}
\begin{proof}
Let $i,i'\in[n]$ be any two indexes such that $z^*_i=a$ and $z^*_{i'}=b$. From Lemma \ref{lem:population}, we have
\begin{align*}
U^*(\theta_a^*  - \theta_b^*)^T &= U^* ((U^*\Lambda^*)_{i\cdot} - (U^*\Lambda^*)_{i'\cdot})^T  = U^* (U^*_{i\cdot} \Lambda^* - U^*_{i'\cdot} \Lambda^* )^T \\
& = U^* \Lambda^* (U^*_{i\cdot}  -  U^*_{i'\cdot})^T = P_{\cdot i} - P_{\cdot i'}.
\end{align*}
The proof is complete with $P_{ji} =p\indic{z^*_j=a} + q\indic{z^*_j\neq a}$ and $P_{ji'} = p\indic{z^*_j=b} + q\indic{z^*_j\neq b}$ for any $j\in[n]$.
\end{proof}

\begin{proof}[Proof of Lemma \ref{lem:population}]

From Lemma \ref{lem:U_star}, we have  $U^* = Z^* \Delta^{-1}W$ and $\Lambda^* = W^T \Delta B \Delta W$ for some $W\in \matho(k,k)$. Then
\begin{align*}
U^* \Lambda^* = (Z^* \Delta^{-1}W)(W^T \Delta B \Delta W) = Z^* B \Delta W.
\end{align*}
Since $Z^*$ has $k$ unique rows, so does $U^* \Lambda^* $.  Denote $\theta^*_a = B_{a \cdot} \Delta W$ for each $a\in[k]$. Then
\begin{align*}
\norm{\theta^*_a}^2 = \norm{ B_{a \cdot} \Delta}^2 = \sum_{b\in[k]} (B_{ab}\Delta_{bb})^2 =  \sum_{b\in[k]} (B_{ab}n_b)^2 = (p^2-q^2)n_a + q^2n, \forall a\in[k],
\end{align*}
where we use the fact that $\sum_{b\in[k]}n_b = n$.
For each $i\in[n]$, we have $(U^*\Lambda^*)_{i\cdot} = Z^*_{i\cdot}  B \Delta W =\theta^*_{z^*_i}$. We also have for any $a,b\in[k]$ such that $a\neq b$.
\begin{align*}
\norm{\theta^*_a - \theta^*_b}  = \norm{(B_{a \cdot} - B_{b \cdot}) \Delta W} = \norm{(B_{a \cdot} - B_{b \cdot}) \Delta } =(p-q)\sqrt{n_a+n_b}.
\end{align*}
\end{proof}

\begin{proof}[Proof of Proposition \ref{prop:prelim}]

Note that performing $k$-means on $\{U_{i\cdot}\Lambda\}_{i\in[n]}$ is equivalent to  performing $k$-means on $\{U_{i\cdot}\Lambda U^T\}_{i\in[n]}$. This is because $U$ has orthonormal columns and consequently $\|U_{i\cdot}\Lambda - U_{j\cdot}\Lambda\| = \|U_{i\cdot}\Lambda U^T - U_{j\cdot}\Lambda U^T\|$ for all $i,j\in[n]$. Since $U_{i\cdot}\Lambda U^T = \tilde A_{i\cdot}UU^T$,  we have
 \begin{align*}
(\hat z,\{\hat \theta_1,\ldots,\hat \theta_k\}) &= \argmin_{z\in[k]^n, \theta_1,\ldots,\theta_k\in\mathr^{1\times k}} \sum_{i\in[n]}\norm{U_{i\cdot}\Lambda U^T - \theta_{z_i}U^T}^2\\
&= \argmin_{z\in[k]^n, \theta_1,\ldots,\theta_k\in\mathr^{1\times k}} \sum_{i\in[n]}\norm{\tilde A_{i\cdot}UU^T - \theta_{z_i}U^T}^2.
 \end{align*}

 Note that $P$ has $k$ unique rows $\theta^*_1U^{*T},\ldots, \theta^*_k U^{*T}$.  According to Lemma \ref{lem:population},
 \begin{align*}
 \min_{a,b\in[k]:a\neq b} \norm{\theta^*_aU^{*T} - \theta^*_b U^{*T}} = \min_{a,b\in[k]:a\neq b} \norm{\theta^*_a - \theta^*_b }  = \min_{a,b\in[k]:a\neq b} \sqrt{n_a+n_b}(p-q) =: \delta.
 \end{align*}
 Hence, $\delta$ is the minimum distance among all $k$ unique rows of $P$. Then by Proposition 3.1 of \cite{zhang2022leave}, if 
 \begin{align}\label{eqn:prop_1}
 \frac{\delta}{\beta^{-0.5}kn^{-0.5}\norm{\tilde A-P}} \geq 16
 \end{align}
 is satisfied, there exists some $\phi\in\Phi$ and some constant $C>0$ such that 
 \begin{align*}
 \frac{1}{n}\sum_{i\in[n]}\indic{\hat z_i \neq \phi(z^*_i)} \leq  \frac{Ck\norm{\tilde A-P}^2}{n\delta^2},
 \end{align*}
 and,
 \begin{align*}
 \max_{a\in[k]}\norm{\hat \theta_{\phi(a)} U^T - \theta^*_a U^{*T}}\leq C\beta^{-0.5}kn^{-0.5}\norm{\tilde A-P}.
 \end{align*}

 In the following, we are going to give an upper bound for $\normt{\tilde A - P}$. Note that $\normt{\tilde A-P}\leq \normt{\tilde A-\E A} + \normt{\E A -P} = \normt{\tilde A-\E A} + p$ with $\normt{\tilde A -\E A}\leq C_0\sqrt{np}$ assumed.
In addition, we have $\delta \geq \sqrt{2\beta n/k} (p-q)$.
  As a result, there exists a constant $C_1>0$, such that if $\frac{n(p-q)^2}{\beta^{-2}k^3 p}\geq C'_1$, we have (\ref{eqn:prop_1}) satisfied, and consequently obtain the desired upper bounds in (\ref{eqn:prop_prelim_1}) and (\ref{eqn:prop_prelim_2}) for some constant $C_2>0$.
  
\end{proof}

\begin{proof}[Proof of Proposition \ref{prop:exponent}]
First consider any $i\in[n]$ such that $z^*_i = 2$. Then
\begin{align}
\indic{\check z_i \neq z^*_i} &= \indic{\norm{A_{i\cdot}U^* - \theta^*_1}^2\leq \norm{A_{i\cdot} U^*- \theta^*_2}^2}  \nonumber\\
&=\indic{2A_{i\cdot}U^*\br{\theta^*_1- \theta^*_2}^T\geq -\norm{\theta_2^* }^2 + \norm{\theta_1^* }^2} \nonumber\\
& = \indic{2(p-q)\br{\sum_{j:z^*_j=1} A_{ij} - \sum_{j\neq i:z^*_j=2}A_{ij} }\geq \br{p^2-q^2}(n_1-n_2) } \nonumber\\
& =  \indic{{\sum_{j:z^*_j=1} A_{ij} - \sum_{j\neq i:z^*_j=2}A_{ij} }\geq \frac{p+q}{2}(n_1-n_2) } \nonumber\\
& =   \indic{\sum_{j:z^*_j=1} (A_{ij} -\E A_{ij})- \sum_{j\neq i:z^*_j=2}(A_{ij} -\E A_{ij}) \geq \frac{p-q}{2}n -p}, \label{eqnpn:34}
\end{align}
where the third equation is by Lemma \ref{lem:population} and Lemma \ref{lem:P_line_diff} and in the last equation we use $n_1+n_2=n$. Note that $|\{j\neq i:z^*_j=2\}|=n_2-1$.
Using Lemma \ref{lem:binomial_diff}, we have
\begin{align*}
\E \indic{\check z_i \neq z^*_i}  &= \pbr{\sum_{j:z^*_j=1} (A_{ij} -\E A_{ij})- \sum_{j\neq i:z^*_j=2}(A_{ij} -\E A_{ij}) \geq \frac{1}{2}\br{1-\frac{2q}{(n-1)(p-q)}} (n-1)(p-q ) }\\
&\leq \ebr{-J_{n_1,n_2-1,p,q} + \frac{q}{(n-1)(p-q)}(n_1+n_2-1) \frac{(p-q)^2}{q}}\\
&=  \ebr{-J_{n_1,n_2-1,p,q}  + (p-q)}.
\end{align*}
From Lemma \ref{lem:J}, we have 
\begin{align*}
\E \indic{\check z_i \neq z^*_i} &\leq \ebr{-J_{n_1,n_2,p,q} + \frac{(p-q)^2}{4q} + (p-q) } \\
&\leq \ebr{-\br{1- \frac{ \frac{(p-q)^2}{4q} + (p-q)}{J_{n_1,n_2,p,q}}}J_{n_1,n_2,p,q}}\\
&\leq \ebr{-\br{1- \frac{ \frac{(p-q)^2}{4q} + (p-q)}{n_2 \frac{(p-q)^2}{8p}}}J_{n_1,n_2,p,q}}\\
&\leq \ebr{-\br{1-\frac{C_1}{n}}J_{n_1,n_2,p,q}},
\end{align*}
for some constant $C_1>0$. For its lower bound, we introduce an auxiliary random variable $X\sim\Ber(p)$ that is independent of $A$. Since $X\geq 0$, (\ref{eqnpn:34}) leads to
\begin{align*}
\indic{\check z_i \neq z^*_i}  &\geq   \indic{\sum_{j:z^*_j=1} (A_{ij} -\E A_{ij})- \sum_{j\neq i:z^*_j=2}(A_{ij} -\E A_{ij}) \geq \frac{p-q}{2}n + X -p}\\
&= \indic{\sum_{j:z^*_j=1} (A_{ij} -\E A_{ij})- \sum_{j\neq i:z^*_j=2}(A_{ij} -\E A_{ij}) -( X -p) \geq \frac{p-q}{2}n}.
\end{align*}
Using  Lemma \ref{lem:binomial_diff} and Lemma \ref{lem:J} again, we have
\begin{align*}
\E \indic{\check z_i \neq z^*_i}  &\geq \ebr{-J_{n_1,n_2,p,q} - 4 \sqrt{\frac{n(p-q)^2p}{q^2}}}\br{\frac{1}{4} -  \sqrt{\frac{2}{nq}}}\\
&\geq \ebr{-\br{1+ \frac{ 4 \sqrt{\frac{n(p-q)^2p}{q^2}}}{J_{n_1,n_2,p,q}}} J_{n_1,n_2,p,q}}\br{\frac{1}{4} -  \sqrt{\frac{2}{nq}}}\\
&\geq \ebr{-\br{1+ \frac{4 \sqrt{\frac{n(p-q)^2p}{q^2}}}{n_2\frac{(p-q)^2}{8p}}} J_{n_1,n_2,p,q}}\br{\frac{1}{4} -  \sqrt{\frac{2}{nq}}}\\
&\geq \ebr{-\br{1+C_2\sqrt{\frac{p}{n(p-q)^2}}}J_{n_1,n_2,p,q}},
\end{align*}
for some constant $C_2>0$. 

Similarly, for any $i\in[n]$ such that $z^*_i =1$, we have 
\begin{align*}
\ebr{-\br{1+C_2\sqrt{\frac{p}{n(p-q)^2}}}J_{n_2,n_1,p,q}}\leq \E \indic{\check z_i \neq z^*_i}\leq \ebr{-\br{1-\frac{C_1}{n}}J_{n_2,n_1,p,q}}.
\end{align*}
Hence,
\begin{align*}
\E \br{ \frac{1}{n}\sum_{i\in[n]}\indic{\check z_i \neq z^*_i} } &\leq \frac{1}{n}\br{n_1\ebr{-\br{1-\frac{C_1}{n}}J_{n_2,n_1,p,q}} + n_2\ebr{-\br{1-\frac{C_1}{n}}J_{n_1,n_2,p,q}} } \\
& \leq\ebr{-\br{1-\frac{C_1}{n}}\br{J_{n_1,n_2,p,q} \wedge J_{n_2,n_1,p,q}}},
\end{align*}
and similarly
\begin{align*}
\E \br{ \frac{1}{n}\sum_{i\in[n]}\indic{\check z_i \neq z^*_i} } &\geq  \frac{n_1\wedge n_2}{n}  \ebr{-\br{1+C_2\sqrt{\frac{p}{n(p-q)^2}}}\br{J_{n_1,n_2,p,q} \wedge J_{n_2,n_1,p,q}}}.
\end{align*}
\end{proof}

\section{Auxiliary Lemmas}\label{sec:auxiliary}

\begin{lemma}\label{lem:Lambda_star}
We have $\lambda_k^* \geq \frac{\beta n(p-q)}{k}$ and
\begin{align*}
\max_{i\in[k]}\abs{\lambda_i-\lambda_i^*},\max_{i>k} \abs{\lambda_i} \leq  \norm{\tilde A -P}.
\end{align*}
Under the assumption that $\normt{\tilde A -P}\leq \frac{\beta n(p-q)}{2k}$, we have
\begin{align*}
\norm{\Lambda^{-1}}\leq \frac{2k}{\beta n(p-q)}.
\end{align*}
\end{lemma}
\begin{proof}%
From Lemma \ref{lem:U_star}, we know $\lambda^*_1,\ldots,\lambda^*_k$ are the  eigenvalues of $\Delta B\Delta$. Then
\begin{align*}
\lambda^*_k & \geq \min_{v\in\mathr^k:\norm{v}=1} v^T  \Delta B\Delta v  =  \min_{v\in\mathr^k:\norm{v}=1}  ( \Delta v)^T B (\Delta v) \geq (\min_{a\in[k]} n_a) \min_{v\in\mathr^k:\norm{v}=1} v^T   B v \\
& =   (\min_{a\in[k]} n_a)  \min_{v\in\mathr^k:\norm{v}=1} \br{q(v^T \one_k)^2 + (p-q) \norm{v}^2} \geq (p-q) (\min_{a\in[k]} n_a)  = \frac{\beta n(p-q)}{k}.
\end{align*}
The upper bound for $\max_{i\in[k]}\abs{\lambda_i-\lambda_i^*},\max_{i>k} \abs{\lambda_i} $ is from Weyl's inequality. If $\normt{\tilde A -P}\leq \frac{\beta n(p-q)}{2k}$ is further assumed, we have $\lambda_k \geq \frac{\beta n(p-q)}{2k}$. The proof is complete with $\norm{\Lambda^{-1}} = \lambda_k^{-1}$.
\end{proof}

\begin{lemma}\label{lem:S_size}
We have
\begin{align*}
\E \abs{S} \leq n\ebr{-128np}.
\end{align*}
Recall the definition of $S_i$ in (\ref{eqnpn:4}). We also have $\E \abs{S_i}\leq \ebr{-128np}$ for each $i\in[n]$.
\end{lemma}
\begin{proof}
We have $\E \abs{S} =n \pbr{\Binom(n-1,p)\geq \tau}\leq n\ebr{-128np}$ by Chernoff bound. The same upper bound holds for $\E \abs{S_i}$.
\end{proof}

\begin{lemma}\label{lem:UU_difference}
Under the assumption $\max\{p\sqrt{n}, \|\tilde A - P\|\} \leq \beta n(p-q)/(8k^2)$,  we have
\begin{align*}
\inf_{W\in\matho(k,k)} \fnorm{U - U^{(i)}W} 
\leq 6k^{1.5}\norm{U_{i\cdot}},\forall i\in[n].
\end{align*}
\end{lemma}
\begin{proof}[Proof of Lemma \ref{lem:UU_difference}]
Consider any $i\in[n]$.
Note that $U$ (resp. $U^{(i)}$) is the leading eigenspace of $\tilde A$ (resp. $\tilde A^{(i)}$) and  $\tilde A^{(i)}$ is obtained from $\tilde A$ by zeroing out its $i$th row and column. In addition, from Lemma \ref{lem:Lambda_star}, we have $\lambda_k^*\geq {\beta n(p-q)}/{k}$. We also have $\max_{j\in[n]} \norm{P_{j\cdot}}\leq p\sqrt{n}$. 

We are going to have a partition for the eigenspaces. But before that, we need to have a partition of $[k]$. With  $\lambda^*_0:=\infty, \lambda^*_{k+1}:=0$, and $a_1:=1$, define $b_1 :=\argmax \{a_1\leq j \leq k: \lambda^*_{a_1}/(\lambda^*_j -\lambda^*_{j+1})\leq 2k\}$. Note that the set $\{a_1\leq j \leq k: \lambda^*_{a_1}/(\lambda^*_j -\lambda^*_{j+1})\leq 2k\}$ is not an empty set. Otherwise, we have $\lambda^*_j -\lambda^*_{j+1} \leq  \lambda^*_{a_1}/(2k)$ for all $a_1\leq j\leq k$. Then we have $\lambda^*_{a_1}-\lambda^*_{k+1} = \sum_{j=a_1}^k (\lambda^*_j -\lambda^*_{j+1})\leq \lambda^*_{a_1}/2$ which gives  $\lambda^*_{k+1}\geq \lambda^*_{a_1}/2\geq \lambda^*_k/2 >0$, a contradiction with the fact that $\lambda^*_{k+1}=0$. As a result, the aforementioned set is non-empty and $b_1$ is well-defined. If $b_1<k$, then we define $a_2:=b_1+1$ and $b_2 := \argmax \{a_2\leq j \leq k: \lambda^*_{a_2}/(\lambda^*_j -\lambda^*_{j+1})\leq  2k\}$. By the same argument as above, $b_2$ is also well-defined. If $b_2<k$, we repeat this procedure until we have $b_r=k$ for some $r\leq k$. In this way, we have a partition of $[k]=\cup_{s=1}^k\{j\in \mathbb{N}: a_s\leq j\leq b_s\}$.  For any $s\in[r]$, define $U^*_s :=(u^*_{a_{s}},\ldots, u^*_{b_{s}})$, $U_s:=(u_{a_{s}},\ldots, u_{b_{s}})$, and define $U^{(i)}_s$ analogously. Then
\begin{align*}
\fnorm{UU^T -U^{(i)}U^{(i)T}}^2 &= \fnorm{\sum_{s\in[r]}\br{U_s U_s^T - U^{(i)}_sU^{(i)T}_s}}^2 \leq   r\sum_{s\in[r]}\norm{{U_s U_s^T - U^{(i)}_sU^{(i)T}_s}}^2.
\end{align*}

Now consider any $s\in[r]$. With $b_0:=0$ and $a_{r+1}:=k+1$, define $\Delta^*_s:=(\lambda^*_{b_{s-1}} - \lambda^*_{a_{s}})\wedge (\lambda^*_{b_s} - \lambda^*_{a_{s+1}})$ to be the spectral gap for $U^*_s$. 
We have
\begin{align*}
\frac{\lambda^*_{a_{s}}}{\Delta^*_s} &= \frac{\lambda^*_{a_{s}}}{(\lambda^*_{b_{s-1}} - \lambda^*_{a_{s}})\wedge (\lambda^*_{b_s} - \lambda^*_{a_{s+1}})}  =  \frac{\lambda^*_{a_{s}}}{\lambda^*_{b_{s-1}} - \lambda^*_{a_{s}}}\vee  \frac{\lambda^*_{a_{s}}}{\lambda^*_{b_s} - \lambda^*_{a_{s+1}}}\\
&\leq  \frac{\lambda^*_{a_{s}}}{{(2k)^{-1}\lambda^*_{a_{s-1}}}\indic{s\geq 2} + \infty\indic{s=1}}\vee (2k) =2k.
\end{align*}
This  implies that $\Delta_s^* \geq \lambda^*_{a_s}/(2k)\geq \lambda^*_k/2k \geq \beta n(p-q)/(2k^2)$. 
Under the assumption that $\max\{p\sqrt{n}, \|\tilde A - P\|\} \leq \beta n(p-q)/(8k^2)$, by Lemma 3 of \cite{abbe2020entrywise}, we have
\begin{align*}
\norm{U_s U_s^T - U^{(i)}_sU^{(i)T}_s}\leq 3(2k)\norm{(U_sU_s^TU_s^*)_{i\cdot}}=6k\norm{(U_s)_{i\cdot}U_s^TU_s^*}\leq 6k\norm{(U_s)_{i\cdot}}.
\end{align*} 
Hence,
\begin{align*}
\fnorm{UU^T -U^{(i)}U^{(i)T}}^2 \leq \sum_{s\in[r]} 36k^2r\norm{(U_s)_{i\cdot}}^2 = 36k^2r\norm{U_{i\cdot}}^2\leq 36k^3\norm{U_{i\cdot}}^2,
\end{align*}
where the second inequality is due to that $U=(U_1,\ldots, U_{r})$. By properties of the Sin $\Theta$ distance, we have
\begin{align*}
\inf_{W\in\matho(k,k)} \fnorm{U - U^{(i)}W}\leq \fnorm{UU^T -U^{(i)}U^{(i)T}} \leq 6k^{1.5}\norm{U_{i\cdot}}.
\end{align*}
\end{proof}

\begin{lemma}\label{lem:ft}
Recall the definition of $f_t(\cdot)$ in (\ref{eqn:ft_def}). For any $t>0$ and any $x,y\in\mathr^{1\times k}$, we have $\norm{f_t(x)-f_t(y)}\leq \norm{x-y}$.
\end{lemma}
\begin{proof}
Without loss of generality, assume $\norm{x}\geq \norm{y}$. We first state a fact that can be easilty verified: if we shrink $x$ until it has the same norm as $y$, its distance toward $y$ is always decreasing. It implies:
\begin{align}\label{eqn:ft_1}
\norm{sx - y}\leq \norm{x-y},\forall  \frac{\norm{y}}{\norm{x}}\leq s\leq 1.
\end{align}

Now we discuss three cases.
If  $\norm{x}\leq t$. Then we have $f_t(x)=x$ and $f_t(y)=y$, and then the equality holds. If $\norm{y}\geq t$, we have
\begin{align*}
\norm{f_t(x)-f_t(y)} &= \norm{\frac{t}{\norm{x}} x- \frac{t}{\norm{y}}y }= \frac{t}{\norm{y}}\norm{\frac{\norm{y}}{\norm{x}} x -y} \leq \norm{\frac{\norm{y}}{\norm{x}} x -y} \leq \norm{x-y},
\end{align*}
where the last inequality is due to (\ref{eqn:ft_1}). If $\norm{x}\geq t \geq \norm{y}$, we have
\begin{align*}
\norm{f_t(x)-f_t(y)} & = \norm{\frac{t}{\norm{x}} x - y} \leq \norm{x-y},
\end{align*}
where the last inequality is due to (\ref{eqn:ft_1}).
\end{proof}

\begin{lemma}\label{lem:bar_U_i_U_star_diff}
Consider any $i\in[n]$. For any $t>0$, we have
\begin{align*}
\norm{(I_n - U^*U^{*T})\bar  U^{(i)}} _{2,\infty} &\leq t + \sqrt{\frac{k}{\beta n}}.
\end{align*}
 Under the assumption $ \normt{\tilde A^{(i)} - \E A} + p \leq (p-q)\beta n/(2k)$, for any $t>\sqrt{k/(\beta n)}$, we have
 \begin{align*}
\fnorm{(I_n - U^*U^{*T})\bar  U^{(i)}}   \leq \frac{2\sqrt{2} k^{1.5}\br{ \norm{\tilde A^{(i)} - \E A} + p}}{\beta n(p-q)}.
\end{align*}
\end{lemma}
\begin{proof}
Note that $\normt{U^*}_{2,\infty}\leq  \sqrt{k/(\beta n)}$ from Lemma \ref{lem:U_star} and  we have $\lambda_k^*\geq  (p-q)\beta n/k$ from Lemma \ref{lem:Lambda_star}. For any $t>0$, we have $\bar U^{(i)}_{j\cdot} = f_t(U^{(i)}_{j\cdot})= s_j U^{(i)}_{j\cdot} $ for some $s_j\in[0,1]$ for any $j\in[n]$.  Then,
\begin{align*}
\norm{\bar U^{(i)}} &= \max_{x\in\mathr^k} \norm{\bar U^{(i)}x}   = \max_{x\in\mathr^k} \sqrt{\sum_{j\in[n]} \br{\bar U^{(i)}_{j\cdot}x}^2} = \max_{x\in\mathr^k} \sqrt{\sum_{j\in[n]} s_j^2 \br{ U^{(i)}_{j\cdot}x}^2} \\
 &\leq  \max_{x\in\mathr^k} \sqrt{\sum_{j\in[n]}  \br{ U^{(i)}_{j\cdot}x}^2} = \norm{ U} = 1.
\end{align*}
Then,
\begin{align*}
\norm{(I_n - U^*U^{*T})\bar  U^{(i)}} _{2,\infty} &\leq \norm{\bar  U^{(i)}} _{2,\infty} + \norm{U^*U^{*T}\bar  U^{(i)}} _{2,\infty} \leq t+ \norm{U^*} _{2,\infty} \norm{U^{*T}\bar  U^{(i)}}\\
&\leq  t+  \norm{U^*} _{2,\infty} \norm{\bar U^{(i)}} \leq  t +  \norm{U^*} _{2,\infty} \leq t + \sqrt{\frac{k}{\beta n}}.
\end{align*}

On the other hand, we have $\normt{\tilde A^{(i)} - P }\leq \normt{\tilde A^{(i)} - \E A  } + \norm{\E A -P}=   \normt{\tilde A^{(i)} - \E A  }  + p$.
By Davis-Kahan  Theorem, when $\normt{\tilde A^{(i)} - \E A  }  + p\leq (p-q)\beta n/(2k)$,  there exists some orthogonal matrix $W\in \matho(k,k)$ such that
\begin{align*}
\norm{U^* -U^{(i)}W} \leq   \frac{2\norm{A^{(i)} - P } }{\lambda^*_k} \leq   \frac{2k\br{ \norm{\tilde A^{(i)} - \E A} + p}}{\beta n(p-q)}.
\end{align*}
Since $U^* -U^{(i)}W\in\mathr^{n\times (2k)}$, we have
\begin{align*}
\fnorm{U^* -U^{(i)}W}  \leq \sqrt{2k}\norm{U^* -U^{(i)}W} \leq   \frac{2\sqrt{2} k^{1.5}\br{ \norm{\tilde A^{(i)} - \E A} + p}}{\beta n(p-q)}.
\end{align*}
For any $j\in[n]$ and any $t>\sqrt{k/(\beta n)}$, we have $U^*_{j\cdot} = f_t(U^*_{j\cdot})$ and
\begin{align*}
\norm{U^*_{j\cdot} -\bar  U_{j\cdot}^{(i)}W} &= \norm{f_t(U^*_{j\cdot} )- f_t(  U_{j\cdot}^{(i)}W)}  \leq \norm{U^*_{j\cdot} -  U_{j\cdot}^{(i)}W},
\end{align*}
where the last inequality is due to Lemma \ref{lem:ft}. Hence, 
\begin{align*}
\fnorm{U^* - \bar  U^{(i)}W}\leq \fnorm{U^* -U^{(i)}W}  \leq \frac{2\sqrt{2} k^{1.5}\br{ \norm{\tilde A^{(i)} - \E A} + p}}{\beta n(p-q)}.
\end{align*}
We get the desired  $\fnorm{\cdot}$ upper bound with
\begin{align*}
\fnorm{(I_n - U^*U^{*T})\bar  U^{(i)}} = \fnorm{(I_n - U^*U^{*T})\bar  U^{(i)} W} = \fnorm{(I_n - U^*U^{*T})\br{U^*-\bar  U^{(i)}W}}   \leq \fnorm{U^* - \bar  U^{(i)}W}.
\end{align*}
\end{proof}

\begin{lemma}\label{lem:weighted_bernoulli_simplified}
Consider an integer $m>0$ and  independent Bernoulli random variables $\{X_i\}_{i\in[m]}$. Denote $p_{\max}:= \max_{i\in[m]}\E X_i$ and assume $p_{\max}\leq 1/2 $. For any $s\in\mathr$ and any $w\in\mathr^m$, we have
\begin{align*}
\sum_{i\in[m]}\log \E \ebr{sw_i(X_i - \E X_i)} \leq  p_{\max} s^2 \norm{w}^2 e^{\abs{s}\norm{w}_\infty}.
\end{align*}
\end{lemma}
\begin{proof}
Define  $f(t;q):= \log\br{q e^{t} +(1-q)}-qt$ to be a function of $t$ where $0\leq q\leq 1$. Then its  first and second derivatives  are
\begin{align*}
f'(t;q) =  \frac{qe^t}{qe^t + (1-q)} -q\text{\quad and \quad} f''(t;q) = \frac{q(1-q)e^t}{\br{qe^t + (1-q)}^2}.
\end{align*}
Note that $f(0;q) =f'(0;q)=0$ and $f''(t;q)\geq 0$ for all $t\in\mathr$.  If $q\leq 1/2$ is further assumed, we have $ f''(t;q) \leq q e^t/(1-q)\leq  2qe^t$ for any $t$.  Then for $t_0>0$ and any $t\in[-t_0,t_0]$, we have $f''(t;q)\leq 2qe^{t_0}$ and 
\begin{align*}
f(t;q)\leq \frac{1}{2} \br{2qe^{t_0}}t^2\leq qe^{t_0}t^2.
\end{align*}
Hence,
\begin{align*}
\sum_{i\in[m]}\log \E \ebr{sw_i(X_i - \E X_i)}  &= \sum_{i\in[m]} f(sw_i;\E X_i)\leq   \sum_{i\in[m]} (\E X_i) e^{\abs{s}\norm{w}_\infty} s^2 w_i^2 \leq p_{\max} s^2 \norm{w}^2 e^{\abs{s}\norm{w}_\infty}.
\end{align*}
\end{proof}

\begin{lemma}\label{lem:indicator}
Consider any $s>0$ and any integer $m>0$. For any $\{a_i\}_{i\in[m]}$ such that $a_i\geq 0,\forall i\in[m]$, we have
\begin{align*}
\br{\sum_{i\in[m]}a_i}^2\indic{\sum_{i\in[m]}a_i\geq s}\leq m^2\sum_{i\in[m]} a_i^2\indic{a_i \geq \frac{s}{m}}.
\end{align*}
\end{lemma}
\begin{proof}
We have
\begin{align*}
\br{\sum_{i\in[m]}a_i}^2\indic{\sum_{i\in[m]}a_i\geq s} &\leq \br{\sum_{i\in[m]}a_i}^2 \sum_{l\in[m]}\indic{\sum_{i\in[m]}a_i\geq s, a_l\geq \max_{j:j\neq l} a_j}\\
& = \sum_{l\in[m]} \br{a_l + \sum_{j:j\neq l} a_j}^2\indic{\sum_{i\in[m]}a_i\geq s, a_l\geq \max_{j:j\neq l} a_j}\\
&\leq \sum_{l\in[m]} (ma_l)^2\indic{\sum_{i\in[m]}a_i\geq s, a_l\geq \max_{j:j\neq l} a_j}\\
&\leq \sum_{l\in[m]} (ma_l)^2\indic{a_l \geq \frac{s}{m}}.
\end{align*}
\end{proof}

\begin{lemma}\label{lem:summation}
Consider any two scalars $s,t>8$. Then 
\begin{align*}
\sum_{j=1}^{\infty} (j+1)^2\ebr{- \frac{\frac{1}{2}j^2 }{\frac{1}{s} + \frac{1}{3}\frac{j}{t}}} \leq 8\ebr{ -\frac{\frac{1}{2}}{\frac{1}{s} + \frac{1}{3} \frac{1}{t}}}.
\end{align*}
\end{lemma}
\begin{proof}
For any $j\geq 1$, we have
\begin{align*}
\frac{\frac{\frac{1}{2}(j+1)^2 }{\frac{1}{s} + \frac{1}{3}\frac{j+1}{t}}}{\frac{\frac{1}{2}j^2 }{\frac{1}{s} + \frac{1}{3}\frac{j}{t}}} = \frac{(j+1)^2}{j^2} \frac{\frac{1}{s} + \frac{1}{3}\frac{j}{t}}{\frac{1}{s} + \frac{1}{3}\frac{j+1}{t}}= \frac{(j+1)^2}{j^2}  \frac{1}{1+\frac{\frac{1}{3t}}{\frac{1}{s} +\frac{j}{3t}}} = \frac{(j+1)^2}{j^2}  \frac{1}{1+ \frac{1}{\frac{3t}{s} + j}} \geq  \frac{(j+1)^2}{j^2}  \frac{1}{1+ \frac{1}{j}} \geq \frac{j+1}{j}.
\end{align*}
Hence,
\begin{align*}
 \frac{\frac{1}{2}j^2 }{\frac{1}{s} + \frac{1}{3}\frac{j}{t}} \geq \frac{j}{j-1} \frac{j-1}{j-2}\ldots \frac{2}{1} \frac{\frac{1}{2}}{\frac{1}{s} + \frac{1}{3t}} = j   \frac{\frac{1}{2}}{\frac{1}{s} + \frac{1}{3t}}. 
\end{align*}
As a result,
\begin{align*}
\sum_{j=1}^{\infty} (j+1)^2\ebr{- \frac{\frac{1}{2}j^2 }{\frac{1}{s} + \frac{1}{3}\frac{j}{t}}} &\leq  \sum_{j=1}^{\infty} (j+1)^2 \ebr{-j   \frac{\frac{1}{2}}{\frac{1}{s} + \frac{1}{3t}}}.
\end{align*}
Since $s,t>8$, we have $ \frac{\frac{1}{2}}{\frac{1}{s} + \frac{1}{3t}} >2$. Then the first term in $\cbr{\ebr{-j   \frac{\frac{1}{2}}{\frac{1}{s} + \frac{1}{3t}}}}_{j\geq 1}$ dominates. As a result, the above display is upper bounded by $8\ebr{ -\frac{\frac{1}{2}}{\frac{1}{s} + \frac{1}{3} \frac{1}{t}}}$.
\end{proof}

\begin{lemma}\label{lem:J}
Recall the definition of $J_{m_1,m_2,p,q}$ in (\ref{eqn:J_def}). For any positive integers $m_1,m_2$ and any $p,q$ such that $0<q < p\leq 1/2$, we have
\begin{align*}
m_2\frac{(p-q)^2}{8p} \leq J_{m_1,m_2,p,q} &\leq (m_1 + m_2)\frac{(p-q)^2}{4q},
\end{align*}
and
\begin{align*}
 J_{m_1,m_2+1,p,q} \leq    J_{m_1,m_2,p,q} +  \frac{(p-q)^2}{4q}.
\end{align*}
In addition, define $t^* :=\argmax_{t}\br{(m_1-m_2)t \frac{p+q}{2} - m_1 \log\br{q e^t + 1-q} - m_2\log \br{pe^{-t}+1-p}}$. We have
\begin{align}\label{eqnpn:18}
0< t^*\leq \frac{p-q}{q}.
\end{align}
If $p\leq 1/10$ is further assumed, we have
\begin{align*}
J_{m_1,m_2,p,q}\wedge J_{m_2,m_1,p,q}\leq  (m_2\vee m_1)\frac{4(p-q)^2}{3p}.
\end{align*}
\end{lemma}
\begin{proof}
We introduce auxiliary functions:
\begin{align*}
&g_1(t):= t\frac{p+q}{2} - \log \br{qe^t +1-q} \\
&g_2(t):=  -  t\frac{p+q}{2}   - \log \br{pe^{-t} + 1-p}\\
&f(t):= m_1 g_1(t) + m_2 g_2(t).
\end{align*}
Then $J_{m_1,m_2,p,q}  = \max_t f(t)=f(t^*)$. Through calculation, we  have
\begin{align*}
g_1'(t) &=  - \br{1 -\frac{p+q}{2} -  \frac{1-q}{q e^t + 1-q} }\\
g_2'(t) &= 1 -\frac{p+q}{2}  -\frac{1-p}{pe^{-t}+1-p}\\
g_1''(t)  &=  -\frac{(1-q)q e^t}{\br{q e^t + 1-q}^2}\\
g_2''(t) &=- \frac{(1-p)pe^{-t}}{\br{pe^{-t}+1-p}^2}\\
f'(t) &= m_1 g_1'(t) + m_2 g_2'(t)\\
f''(t) &=m_1 g_1''(t) + m_2 g_2''(t).
\end{align*}
Note that $g_1''(t),g_2''(t), f''(t)<0$ for all $t\in\mathr$ which implies $g_1'(t),g_2'(t),f'(t)$ are all decreasing functions. Define 
\begin{align*}
t_1 := \log\br{\frac{1-q}{1-\frac{p+q}{2}} \frac{\frac{p+q}{2}}{q}} = \log\br{1+ \frac{\frac{p-q}{2}}{\br{1-\frac{p+q}{2}}q}}\text{ and } t_2 := \log\br{\frac{1-\frac{p+q}{2}}{1-p} \frac{p}{\frac{p+q}{2}}} = \log\br{1+\frac{\frac{p-q}{2}}{\br{1-p}\frac{p+q}{2}}},
\end{align*}
such that $g_1'(t_1)=0$ and $g_2'(t_2)=0$.
Under the assumption that $0<q< p\leq 1/2$, we have $\br{1-\frac{p+q}{2}}q\leq (1-p)\frac{p+q}{2}$ and consequently $0< t_2\leq t_1$. Using the fact $f'(t)=m_1 g_1'(t)+m_2 g_2'(t)$ and that $g_1',g_2'$ are decreasing function, we have $f'(t_2)=m_1 g_1'(t_2)\geq m_1 g_1'(t_1) =0$ and $f'(t_1) =m_2 g_2'(t_1)\leq m_2 g_2'(t_2)=0$. That $f'(t)$ is a decreasing function leads to
\begin{align}\label{eqnpn:19}
0< t_2 \leq t^*\leq t_1.
\end{align}

Let us first study $g_2(t_2)$ and $g_1(t_1)$ which are important quantities for further analysis. Through calculation, we can show that $g_2(t_2)$ can be simplified into
\begin{align*}
g_2(t_2) = \br{1-\frac{p+q}{2}}\log \frac{1-\frac{p+q}{2}}{1-p} + \frac{p+q}{2}\log \frac{\frac{p+q}{2}}{p}.
\end{align*}
Define  $h(\delta):= (p+\delta)\log \frac{p+\delta}{p} +\br{1-\br{p+\delta}}\log\frac{1-\br{p+\delta}}{1-p}$. We have $h(0)=0$ and $h''(\delta) = \frac{1}{(p+\delta)(1-(p+\delta))}$.  Since $0<q< p\leq 1/2$, we have $h''(\delta) \geq \frac{1}{p(1-p)}$ for any $-\frac{p-q}{2}\leq  \delta \leq 0$ and consequently,
\begin{align}\label{eqnpn:23}
g_2(t_2)=h\br{-\frac{p-q}{2}}\geq \frac{1}{2} \frac{1}{p(1-p)}\br{\frac{p-q}{2}}^2 \geq \frac{(p-q)^2}{8p}.
\end{align}
Similarly, we have
\begin{align}\label{eqnpn:30}
g_2(t_2)\leq \frac{1}{2} \br{\max_{-\frac{p-q}{2}\leq  \delta \leq 0} h''(\delta) }\br{\frac{p-q}{2}}^2 &\leq \frac{1}{2} \frac{1}{\frac{p+q}{2}\br{1-\frac{p+q}{2}}}\br{\frac{p-q}{2}}^2 \leq \frac{(p-q)^2}{4q}.
\end{align}
Using the same argument, we have
\begin{align}\label{eqnpn:31}
g_1(t_1) = \br{1-\frac{p+q}{2}} \log  \frac{1-\frac{p+q}{2}}{1-q}  + \frac{p+q}{2} \log \frac{ \frac{p+q}{2}}{q} \leq \frac{1}{2}\frac{1}{q(1-q)}\br{\frac{p-q}{2}}^2 \leq  \frac{(p-q)^2}{4q}.
\end{align}

Now we are ready to establish the bounds for $J_{m_1,m_2,p,q}$. For the upper bound, we have
\begin{align*}
J_{m_1,m_2,p,q} \leq m_1 g_1(t_1) + m_2g(t_2) \leq (m_1+ m_2) \frac{(p-q)^2}{4q},
\end{align*}
by (\ref{eqnpn:30}) and (\ref{eqnpn:31}). For the lower bound, note that $g_1(t_2)\geq g_1(0)=0$. Then by (\ref{eqnpn:23}), we have
\begin{align*}
J_{m_1,m_2,p,q}   \geq f(t_2)\geq m_2g(t_2) \geq m_2\frac{(p-q)^2}{8p}.
\end{align*}

From (\ref{eqnpn:19}) and the definition of $t_1$, we also have
\begin{align}\label{eqnpn:32}
0< t^* \leq  \frac{\frac{p-q}{2}}{\br{1-\frac{p+q}{2}}q} \leq \frac{p-q}{q},
\end{align}
where the last inequality is due to that $0< p,q\leq 1/2$. Note that $t^*$ is a function of $m_1,m_2$ and (\ref{eqnpn:32}) still holds if we vary values of  $m_1,m_2$, This implies the maximizer of $m_1 g_1(t)+(m_2+1)g_2(t)$ is also within $[0,\frac{p-q}{q}]$. Hence,
\begin{align*}
J_{m_1,m_2+1,p,q} &= \max_{t\in [0,\frac{p-q}{q}]}  \br{m_1 g_1(t)+(m_2+1)g_2(t)} =  \max_{t\in [0,\frac{p-q}{q}]}  \br{f(t)+g_2(t)} &\leq  \max_{t\in [0,\frac{p-q}{q}]}  \br{f(t) +g_2(t_2)} \\
& = J_{m_1,m_2,p,q} + g_2(t_2)\leq   J_{m_1,m_2,p,q} +  \frac{(p-q)^2}{4q},
\end{align*}
where the last inequality is due to (\ref{eqnpn:30}).

In the last part of the proof, we are going to derive an improved upper bound for $J_{m_1,m_2,p,q}\wedge J_{m_2,m_1,p,q}$ under an additional assumption that $p\leq 1/10$. From (\ref{eqnpn:19}), we have $g_1(t^*)\geq g_1(0)=0$. Let us first consider the case that $m_1\leq m_2$. Then
\begin{align*}
J_{m_1,m_2,p,q} &= f(t^*) = m_1 g_1(t^*) + m_2g_2(t^*) \leq m_2 g_1(t^*)  + m_2g_2(t^*) \\
&\leq m_2  \max_{t\geq 0} (g_1(t) + g_2(t)) \leq m_2\frac{4(p-q)^2}{3p},
\end{align*}
where the last equation is due to Lemma \ref{lem:I}. 
 If $m_1>m_2$ instead, by the same argument, we have $J_{m_2,m_1,p,q}\leq m_1 \frac{4(p-q)^2}{3p}$. Hence, $J_{m_1,m_2,p,q} \wedge J_{m_2,m_1,p,q} \leq (m_2\vee m_1)\frac{4(p-q)^2}{3p}$ holds for both cases.
\end{proof}

\begin{lemma}\label{lem:I}
Consider any $0<q<p<1$. Define $I_{p,q}:= -2\log(\sqrt{pq} + \sqrt{(1-p)(1-q)})$. Then $I_{p,q}>0$ and
\begin{align*}
I_{p,q}=\max_{t} \br{- \log \br{pe^{-t} + 1-p}- \log \br{qe^t +1-q} }.
\end{align*}
If $p\leq 1/10$ is further assumed, we have $I_{p,q}\leq \frac{4(p-q)^2}{3p}$.
\end{lemma}
\begin{proof}
The equation for $I_{p,q}$ is by direct calculation and its proof is omitted here. For the upper bound of $I_{p,q}$, we have
\begin{align*}
I_{p,q} &= - \log \br{1-(p+q) + 2\sqrt{pq}\br{\sqrt{pq} + \sqrt{(1-p)(1-q)}}}\\
& = - \log\br{1-\br{\sqrt{p}-\sqrt{q}}^2 - 2\sqrt{pq}\br{1-\sqrt{pq} - \sqrt{(1-p)(1-q)} } }\\
& = - \log\br{1-\br{\sqrt{p}-\sqrt{q}}^2 - 2\sqrt{pq}\frac{(\sqrt{p}-\sqrt{q})^2}{1-\sqrt{pq} + \sqrt{(1-p)(1-q)} } }\\
& = - \log\br{1-\br{1+ \frac{2\sqrt{pq}}{1-\sqrt{pq}  + \sqrt{(1-p)(1-q)}}} \br{\sqrt{p}-\sqrt{q}}^2 }.
\end{align*}
Note that   $-\log(1-x)\leq 1.2x,\forall 0\leq x\leq 0.3$. When $p\leq 1/2$.
\begin{align*}
\br{1+ \frac{2\sqrt{pq}}{1-\sqrt{pq}  + \sqrt{(1-p)(1-q)}}} \br{\sqrt{p}-\sqrt{q}}^2 \leq 2\br{\sqrt{p}-\sqrt{q}}^2 \leq 2p.
\end{align*}
Hence, if we further assume $p\leq 1/10$, the above display is smaller than 0.3, and we have
\begin{align*}
I_{p,q} &\leq 1.2 \br{1+ \frac{2\sqrt{pq}}{1-\sqrt{pq}  + \sqrt{(1-p)(1-q)}}} \br{\sqrt{p}-\sqrt{q}}^2 \\
&\leq 1.2\br{1+ \frac{2p}{1-0.1+0.9}} \frac{(p-q)^2}{p}\\
&\leq  \frac{4(p-q)^2}{3p}.
\end{align*}
\end{proof}

\begin{lemma}\label{lem:binomial_diff}
Consider any two integers $m_1,m_2>0$ and any $p,q$ such that $0<q < p\leq 1/2$. 
Let $\{X_i\}_{i\in[m_1]}\iid \Ber(q)$, $\{Y_j\}_{j\in[m_2]}\iid \Ber(p)$ and assume they are independent of each other. 
Consider any $\rho$ such that $0\leq \rho \leq 1$. We have
\begin{align*}
\pbr{\sum_{i\in[m_1]}\br{X_i -q} - \sum_{j\in[m_2]}\br{Y_j -p}\geq \frac{1-\rho}{2}\br{m_1+m_2}(p-q)} \leq  \ebr{- J_{m_1,m_2,p,q}+ \frac{\rho}{2}\br{m_1+m_2}\frac{(p-q)^2}{q}}.
\end{align*}
and
\begin{align*}
&\pbr{\sum_{i\in[m_1]}\br{X_i -q} - \sum_{j\in[m_2]}\br{Y_j -p}\geq \frac{1+\rho}{2}\br{m_1+m_2}(p-q)} \\
&\geq \ebr{- J_{m_1,m_2,p,q}  - \rho\br{m_1+m_2} \frac{(p-q)^2}{q} -4 \sqrt{\frac{(m_1+m_2)(p-q)^2p}{q^2}}}\br{\frac{1}{4} -  \sqrt{\frac{2}{(m_1+m_2)q}}}.
\end{align*}
\end{lemma}
\begin{proof}
We first prove the upper bound. By Chernoff bound, we have
\begin{align*}
&\log\pbr{\sum_{i\in[m_1]}\br{X_i -q} - \sum_{j\in[m_2]}\br{Y_j -p}\geq \frac{1-\rho}{2}\br{m_1+m_2}(p-q) }\\
&=\log\pbr{\sum_{i\in[m_1]}{X_i } - \sum_{j\in[m_2]}{Y_j }\geq  (m_1-m_2) \frac{p+q}{2}-\frac{\rho}{2}\br{m_1+m_2}(p-q) }\\
&\leq \min_{t> 0}  \br{ -(m_1-m_2) t\frac{p+q}{2} +\frac{\rho}{2}\br{m_1+m_2}(p-q)t  + m_1 \log\br{qe^t + 1-q}  + m_2 \log\br{pe^{-t} + 1-p}}\\
& = -\max_{t> 0} \br{ (m_1-m_2) t\frac{p+q}{2} -\frac{\rho}{2}\br{m_1+m_2}(p-q)t  - m_1 \log\br{qe^t + 1-q}  - m_2 \log\br{pe^{-t} + 1-p}}\\
&\leq   - \br{J_{m_1,m_2,p,q}  - \frac{\rho}{2}\br{m_1+m_2}(p-q)t^*}\\
&\leq  - {J_{m_1,m_2,p,q}   + \frac{\rho}{2}\br{m_1+m_2}}\frac{(p-q)^2}{q}.
\end{align*}
where $t^*$ is defined in the statement of Lemma \ref{lem:J}. Here the second to last inequality holds as $t^*>0$ according to (\ref{eqnpn:18}) and the last inequality holds as $t^*\leq \frac{p-q}{q}$ according to (\ref{eqnpn:18})  as well.

We next prove the lower bound. Define
\begin{align*}
\tilde f(t) &:= (m_1-m_2) t\frac{p+q}{2} +\frac{\rho}{2}\br{m_1+m_2}(p-q)t  - m_1 \log\br{qe^t + 1-q}  - m_2 \log\br{pe^{-t} + 1-p}\\
&= m_1 \br{t\br{\frac{p+q}{2} + \frac{\rho(p-q)}{2}} -  \log\br{qe^t + 1-q}}  + m_2\br{ - \br{\frac{p+q}{2} - \frac{\rho(p-q)}{2}} - \log\br{pe^{-t} + 1-p}},
\end{align*}
 $\tilde t:=\argmax_{t} \tilde f(t)$, and $\tilde J_{m_1,m_2,p,q,\rho}:= \tilde f(\tilde t)=\max_t \tilde f(t)$. 
Define
\begin{align*}
\tilde t_1:= \log \br{\frac{1-q}{1-\br{\frac{p+q}{2} + \frac{\rho(p-q)}{2}}} \frac{\frac{p+q}{2} + \frac{\rho(p-q)}{2}}{q}} = \log\br{1+ \frac{(1+\rho)\frac{p-q}{2}}{\br{1-\br{\frac{p+q}{2} + \frac{\rho(p-q)}{2}}}q}}
\end{align*}
and
\begin{align*}
\tilde t_2:= \log\br{\frac{1-\br{\frac{p+q}{2} - \frac{\rho(p-q)}{2}}}{1-p}\frac{p}{{\frac{p+q}{2} - \frac{\rho(p-q)}{2}}}} =\log\br{1 + \frac{(1+\rho)\frac{p-q}{2}}{(1-p)\br{\frac{p+q}{2} - \frac{\rho(p-q)}{2}}}}.
\end{align*}
Following the same argument used to derive (\ref{eqnpn:19}) in the proof of Lemma \ref{lem:J}, we have $\tilde t$ sandwiched between $\tilde t_1$ and $\tilde t_2$. Hence, similar to (\ref{eqnpn:32}), we have
\begin{align}\label{eqnpn:21}
0<\tilde t\leq \tilde t_1\vee \tilde t_2 \leq   \frac{2(p-q)}{q},
\end{align}
where we use the assumption  $0\leq \rho \leq 1$ and $0<q<p\leq 1/2$. In addition, we have
\begin{align}
1\leq e^{\tilde t} &\leq e^{\tilde t_1}\vee e^{\tilde t_2}\leq 1 + \frac{(1+\rho)\frac{p-q}{2}}{\br{1-\br{\frac{p+q}{2} + \frac{\rho(p-q)}{2}}}q} \vee  \frac{(1+\rho)\frac{p-q}{2}}{(1-p)\br{\frac{p+q}{2} - \frac{\rho(p-q)}{2}}}  \nonumber\\
&\leq 1+\frac{2(p-q)}{q} \leq \frac{2p}{q}.\label{eqnpn:22}
\end{align}

Using (\ref{eqnpn:21}), we have
\begin{align}
&\tilde J_{m_1,m_2,p,q,\rho}\nonumber\\
&=(m_1-m_2) \tilde t \frac{p+q}{2} +\frac{\rho}{2}\br{m_1+m_2}(p-q)\tilde t  - m_1 \log\br{qe^{\tilde t }+ 1-q}  - m_2 \log\br{pe^{-\tilde t } + 1-p}\nonumber\\
&\leq \max_{t}\br{(m_1-m_2) t\frac{p+q}{2}  - m_1 \log\br{qe^t + 1-q}  - m_2 \log\br{pe^{-t} + 1-p}} + \frac{\rho}{2}\br{m_1+m_2}(p-q)\tilde t \nonumber\\
&=J_{m_1,m_2,p,q}  + \frac{\rho}{2}\br{m_1+m_2}(p-q)\tilde t \nonumber\\
&\leq J_{m_1,m_2,p,q}  + \rho\br{m_1+m_2} \frac{(p-q)^2}{q}.\label{eqnpn:20}
\end{align}
Define $M:=\tilde t \br{\frac{p+q}{2}+\frac{\rho(p-q)}{2}} -\log\br{qe^{\tilde t }+ 1-q}$ and $N:=-\tilde t  \br{\frac{p+q}{2} -\frac{\rho(p-q)}{2}} - \log\br{pe^{-\tilde t } + 1-p}$. We are going to use the  Cram\'{e}r-Chernoff argument to establish the lower bound. We have
\begin{align*}
&\pbr{\sum_{i\in[m_1]}\br{X_i -q} - \sum_{j\in[m_2]}\br{Y_j -p}\geq \frac{1+\rho}{2}\br{m_1+m_2}(p-q)} \\
&= \pbr{\sum_{i\in[m_1]}\br{X_i -\br{\frac{p+q}{2} +\frac{\rho(p-q)}{2}}} - \sum_{j\in[m_2]}\br{Y_j -\br{\frac{p+q}{2} -\frac{\rho(p-q)}{2}}}\geq 0}\\
&\geq  \pbr{(m_1+m_2)\delta \geq \sum_{i\in[m_1]}\br{X_i -\br{\frac{p+q}{2} +\frac{\rho(p-q)}{2}}} - \sum_{j\in[m_2]}\br{Y_j -\br{\frac{p+q}{2} -\frac{\rho(p-q)}{2}}}\geq 0}\\
&= \sum_{(x,y)\in \mathcal{X}} \br{ \prod_{i\in[m_1]} h_1(x_i) }  \br{ \prod_{j\in[m_2]} h_2(y_j) }.
\end{align*}
In the above display, $\delta>0$ is some quantity whose value will be given later. The set $\mathcal{X}:= \{(x,y):x\in\mathr^{m_1},y\in\mathr^{m_2}, (m_1+m_2)\delta\geq \sum_{i\in[m_1]}x_i - \sum_{j\in[m_2]}y_j\geq 0 \}$ and $h_1(\cdot),h_2(\cdot)$ are defined to be the probability mass functions of $X_1 -\br{\frac{p+q}{2} +\frac{\rho(p-q)}{2}}$ and $Y_1 -\br{\frac{p+q}{2} -\frac{\rho(p-q)}{2}}$, respectively. We further define
\begin{align*}
M &:=\E \ebr{\tilde t \br{X_1 -\br{\frac{p+q}{2} +\frac{\rho(p-q)}{2}}}} = qe^{\tilde t \br{1-\br{\frac{p+q}{2} +\frac{\rho(p-q)}{2}}}} + (1-q) e^{-\tilde t \br{\frac{p+q}{2} +\frac{\rho(p-q)}{2}}},\\
N &:=\E \ebr{-\tilde t \br{Y_1 -\br{\frac{p+q}{2} -\frac{\rho(p-q)}{2}}}} = pe^{-\tilde t \br{1-\br{\frac{p+q}{2} -\frac{\rho(p-q)}{2}}}} + (1-p) e^{\tilde t \br{\frac{p+q}{2} +\frac{\rho(p-q)}{2}}}.
\end{align*}
Then
\begin{align*}
&\pbr{\sum_{i\in[m_1]}\br{X_i -q} - \sum_{j\in[m_2]}\br{Y_j -p}\geq \frac{1+\rho}{2}\br{m_1+m_2}(p-q)} \\
& ={M^{m_1} N^{m_2}}  \sum_{(x,y)\in \mathcal{X}} \br{ \prod_{i\in[m_1]} \frac{ \ebr{\tilde t x_i}h_1(x_i)}{ \ebr{\tilde t x_i} M} }  \br{ \prod_{j\in[m_2]} \frac{\ebr{-\tilde t y_j} h_2(y_j)}{\ebr{-\tilde t y_j}N} }\\
& \geq \frac{M^{m_1} N^{m_2}}{\ebr{(m_1+m_2)\tilde t \delta}}\sum_{(x,y)\in \mathcal{X}} \br{ \prod_{i\in[m_1]} \frac{ \ebr{\tilde t x_i}h_1(x_i)}{  M} }  \br{ \prod_{j\in[m_2]} \frac{\ebr{-\tilde t y_j} h_2(y_j)}{N} },
\end{align*}
where the last inequality is due to that $\sum_{i\in[m_1]} x_i - \sum_{j\in[m_2]}y_j\leq (m_1+m_2)\delta$ as $(x,y)\in \mathcal{X}$. 
Define $h_1'(w):=\frac{\ebr{\tilde t w}h_1(w)}{  M}$ and $h_2'(w):=\frac{\ebr{-\tilde t w} h_2(w)}{N} $. Since $\sum_w h'_1(w)=\sum_w h'_2(w)=1$, they are both  probability mass functions. Let $U_1,\ldots U_{m_1}$ be i.i.d. random variables distributed according to $h_1'(\cdot)$ and let $V_1,\ldots V_{m_2}$ be i.i.d. random variables distributed according to $h_2'(\cdot)$. We further assume they are independent of each other. Then
\begin{align*}
&\pbr{\sum_{i\in[m_1]}\br{X_i -q} - \sum_{j\in[m_2]}\br{Y_j -p}\geq \frac{1+\rho}{2}\br{m_1+m_2}(p-q)} \\
&\geq  \frac{M^{m_1} N^{m_2}}{\ebr{(m_1+m_2)\tilde t \delta}}\sum_{(x,y)\in \mathcal{X}} \br{ \prod_{i\in[m_1]}h_1'(x_i) }   \br{ \prod_{j\in[m_2]} h_2'(y_j)}\\
&=\frac{M^{m_1} N^{m_2}}{\ebr{(m_1+m_2)\tilde t \delta}}\pbr{(m_1+m_2)\delta \geq \sum_{i\in[m_1]}U_i -\sum_{j\in[m_2]}V_j \geq 0}.
\end{align*}
Note that $M^{m_1} N^{m_2} = \ebr{-\tilde J_{m_1,m_2,p,q,\rho}}\geq \ebr{- J_{m_1,m_2,p,q}  - \rho\br{m_1+m_2} \frac{(p-q)^2}{q}}$ where we use (\ref{eqnpn:20}). Hence,
\begin{align*}
&\pbr{\sum_{i\in[m_1]}\br{X_i -q} - \sum_{j\in[m_2]}\br{Y_j -p}\geq \frac{1+\rho}{2}\br{m_1+m_2}(p-q)} \\
&\geq \ebr{- J_{m_1,m_2,p,q}  - \rho\br{m_1+m_2} \frac{(p-q)^2}{q} - (m_1+m_2)\tilde t \delta}\\
&\quad \times \pbr{\br{m_1+m_2}\delta \geq  \br{\sum_{i\in[m_1]}U_i -\sum_{j\in[m_2]}V_j} \geq 0}.
\end{align*}
Note that
\begin{align*}
\E \br{\sum_{i\in[m_1]}U_i -\sum_{j\in[m_2]}V_j} &=m_1\E U_1 -m_2 \E V_1 = \tilde f'(\tilde t)=0,
\end{align*}
where the second equation is by calculation and the third equation is by the fact that $\tilde t$ is the maximizer of $\tilde f$. 
Then $\sum_{i\in[m_1]}U_i -\sum_{j\in[m_2]}V_j = \sum_{i\in[m_1]}\br{U_i-\E U_i} -\sum_{j\in[m_2]} \br{V_j-\E V_j}$. Through calculation, we have
\begin{align*}
\E \abs{U_1-\E U_1}^2 = \text{Var}(U_1) = \frac{(1-q)q e^{\tilde t}}{\br{q e^{\tilde t} + 1-q}^2},\\
\E \abs{V_1-\E V_1}^2 = \text{Var}(V_1) = \frac{(1-p)pe^{-\tilde t}}{\br{pe^{-\tilde t}+1-p}^2}.
\end{align*}
In addition, we have $ \abs{U_1-\E U_1} \leq 1$ and consequently $\E \abs{U_1-\E U_1}^3\leq  \E \abs{U_1-\E U_1}^2$. Similarly, we have $\E \abs{V_1-\E V_1}^3\leq  \E \abs{V_1-\E V_1}^2$. Then by Berry-Essen Theorem, we have
\begin{align*}
&\pbr{\br{m_1+m_2}\delta \geq  \br{\sum_{i\in[m_1]}U_i -\sum_{j\in[m_2]}V_j} \geq 0}\\
&=\pbr{\frac{\br{m_1+m_2}\delta}{\sqrt{m_1 \text{Var}(U_1) + m_2 \text{Var}(U_2)}} \geq  \frac{\sum_{i\in[m_1]}U_i -\sum_{j\in[m_2]}V_j}{\sqrt{m_1 \text{Var}(U_1) + m_2 \text{Var}(U_2)}} \geq 0}\\
&\geq \ebr{\frac{\br{m_1+m_2}\delta}{\sqrt{m_1 \text{Var}(U_1) + m_2 \text{Var}(U_2)}} \geq \mathn(0,1)\geq 0} - \frac{1}{\sqrt{m_1 \text{Var}(U_1) + m_2 \text{Var}(U_2)}}
\end{align*}
Through direct calculation, we have
\begin{align*}
m_1 \text{Var}(U_1) + m_2 \text{Var}(U_2) &= m_1  \frac{(1-q)q e^{\tilde t}}{\br{q e^{\tilde t} + 1-q}^2} + m_2 \frac{(1-p)pe^{-\tilde t}}{\br{pe^{-\tilde t}+1-p}^2} \\
&\leq  m_1 \frac{qe^{\tilde t}}{1-q} + m_2 \frac{pe^{-\tilde t}}{1-p}\\
&\leq m_1 \frac{q}{1-q} \frac{2p}{q} + m_2 \frac{p}{1-p}\\
&\leq 4(m_1+m_2)p,
\end{align*}
where the second inequality is due to (\ref{eqnpn:22}).
Note that as a function of $t$, $ \frac{(1-q)q e^{ t}}{\br{q e^{ t} + 1-q}^2}$ first increase and then decrease when $t\geq 0$ and grows. By  (\ref{eqnpn:22}) again, we have
\begin{align*}
  \frac{(1-q)q e^{\tilde t}}{\br{q e^{\tilde t} + 1-q}^2} &\geq   \frac{(1-q)q }{\br{q  + 1-q}^2} \wedge   \frac{(1-q)q  \frac{2p}{q}}{\br{q \frac{2p}{q} + 1-q}^2} \geq \frac{q}{2}.
\end{align*}
The same lower bound holds for $ \frac{(1-p)pe^{-\tilde t}}{\br{pe^{-\tilde t}+1-p}^2}$. As a result, 
\begin{align*}
m_1 \text{Var}(U_1) + m_2 \text{Var}(U_2) &\geq (m_1+m_2)\frac{q}{2}.
\end{align*}
Taking $\delta = 2\sqrt{\frac{p}{m_1+m_2}}$, we have
\begin{align*}
&\pbr{\br{m_1+m_2}\delta \geq  \br{\sum_{i\in[m_1]}U_i -\sum_{j\in[m_2]}V_j} \geq 0}\\
&\geq \pbr{1\geq \mathn(0,1)\geq 0} - \sqrt{\frac{2}{(m_1+m_2)q}}\geq \frac{1}{4} -  \sqrt{\frac{2}{(m_1+m_2)q}}.
\end{align*}
With this choice of $\delta$, using (\ref{eqnpn:21}), we have
\begin{align*}
&\pbr{\sum_{i\in[m_1]}\br{X_i -q} - \sum_{j\in[m_2]}\br{Y_j -p}\geq \frac{1+\rho}{2}\br{m_1+m_2}(p-q)} \\
&\geq \ebr{- J_{m_1,m_2,p,q}  - \rho\br{m_1+m_2} \frac{(p-q)^2}{q} -4 \sqrt{\frac{(m_1+m_2)(p-q)^2p}{q^2}}}\br{\frac{1}{4} -  \sqrt{\frac{2}{(m_1+m_2)q}}}.
\end{align*}

\end{proof}

\bibliographystyle{plain}
\bibliography{spectral}

\end{document}